\documentclass[10pt, reqno]{amsart}

\usepackage{latexsym}
\usepackage{amsfonts, amsmath, amscd}
\usepackage{amssymb}
\usepackage[all]{xy}

\usepackage{enumerate}
\usepackage{color}
\usepackage[usenames,dvipsnames]{xcolor}

\usepackage{hyperref}
\hypersetup{colorlinks=false}

\usepackage{verbatim}
\usepackage{amsthm}

 \usepackage{wrapfig}
\usepackage{graphicx, psfrag}
\usepackage{pstool}
 \usepackage{caption}
\usepackage{subcaption}
\usepackage{tikz-cd}

\usepackage{tikz}
\usepackage{pinlabel}   

 \usepackage{hyperref}
 
\numberwithin{equation}{section}

\renewcommand{\theequation}{\arabic{section}.\arabic{equation}}

\newtheorem{theorem}{Theorem}[section]
\newtheorem{prop}[theorem]{Proposition}

\newtheorem{lemma}[theorem]{Lemma}
\newtheorem{cor}[theorem]{Corollary}

\newtheorem{defn}[theorem]{Definition}

\newtheorem{remark}[theorem]{Remark}
\newenvironment{rem}{\begin{remark}\rm}{\end{remark}}
\newtheorem{example}[theorem]{Example}

\newtheorem*{structuretheorem*}{Structure Theorem}

\def\be{\begin{equation}}
\def\ee{\end{equation}}
\def\bear{\begin{eqnarray}}
\def\eear{\end{eqnarray}}
\def\best{\begin{eqnarray*}}
\def\eest{\end{eqnarray*}}

 \def\non{\noindent}
 \def\bd{\partial}
 \def\ra{\rightarrow}
\def\lra{\longrightarrow}
\def\rg{\rangle}
\def\lg{\langle}
\def\r#1{\right#1}
\def\l#1{\left#1}
\def\ti{\times}
\def\del{\overline \partial}

\def\ma#1{\mathop {#1} \limits}

\def\al{\alpha}

\def\ga{\gamma}
\def\de{\delta}
\def\ep{\varepsilon}

\def\si{\sigma}
\def\Si{\Sigma}
\def\De{\Delta}
\def\phi{\varphi}
\def\w{\omega}
\def\om{\omega}

\def\Bbb{\mathbb}

\def\Z{{ \Bbb Z}}
\def\R{{ \Bbb R}}

\def\Q{{ \Bbb Q}}
\def\cx{{ \Bbb C}}

\def\cal{\mathcal}
\def\A{{\mathcal A}}
\def\C{{\mathcal C}}
\def\D{{\mathcal D}}
\def\E{{\mathcal E}}
\def\F{{\mathcal F}}
\def\J{{\mathcal J}}
\def\L{{\mathcal L}}
\def\M{{\mathcal M}}
\def\oM{\overline{\mathcal M}}
\def\N{\mathcal{N}}
\def\O{\mathcal{O}}
\def\P{{\mathcal P}}
\def\W{{\mathcal W}}

\def\wt#1{\widetilde{#1}}
\def\wh#1{\widehat{#1}}
\def\ov#1{\overline{#1}}

\def\cok{\mathrm{coker\,}}
\def\ind{\mathrm{index\,}}
\def\im{\mbox{\rm im}\;}
\def\Aut{\mathrm{Aut}}

\def\Hom{\mathrm{Hom}}
\def\Fred{\mathrm{Fred}}
\def\codim{\mbox{codim }}

\def\sign{\mathrm{sign}\,}

\definecolor{aqua}{cmyk}{0,1,0, .3}
\definecolor{teal}{rgb}{0.2, .8, .7}

\def\ev{\mathrm{ev}}

 \def\cH{\check{\mathrm{H}}}  
 
 \def\Ji{{\mathcal J}_{isol}}
 
\def\pf{\non {\it Proof. }}  

\def\pd{\partial}

\title[The GV conjecture]{The  Gopakumar-Vafa formula for symplectic manifolds} 

\author{Eleny-Nicoleta Ionel}
\address{Department of  Mathematics,  Stanford University}
\email{ionel@math.stanford.edu}
\author{Thomas H. Parker}
\address{Department of  Mathematics, Michigan State University}
\email{parker@math.msu.edu}

\thanks{The research of EI was supported in part by  NSF grant DMS-0905738 and that of TP by  NSF grant DMS-1011793}

\begin{document}

\vskip.15in

\begin{abstract}
The   Gopakumar-Vafa conjecture predicts  that the Gromov-Witten invariants of a Calabi-Yau 3-fold can be canonically expressed  in terms of  integer invariants  called  BPS numbers.   Using the methods of symplectic  Gromov-Witten theory, we prove  that  the Gopakumar-Vafa conjecture holds for any  symplectic Calabi-Yau 6-manifold, and hence for Calabi-Yau 3-folds.  The results  extend to all symplectic 6-manifolds and to the genus zero GW invariants of semipositive manifolds. 
 \end{abstract}

\maketitle

\vspace{-4mm}

\setcounter{equation}{0}


The   Gopakumar-Vafa conjecture \cite{gv} predicts  that the Gromov-Witten invariants $GW_{A, g}$ of a Calabi-Yau 3-fold can be expressed in   terms of some other invariants $n_{A, h}$, called  BPS numbers, by a transform between their generating functions:
\bear\label{Intro.1}
\ma\sum_{\substack{A\ne 0 \\ g}} GW_{A, g}\, t^{2g-2} q^{A}\ =\  \ma\sum_{\substack{A\ne 0 \\ h}} n_{A, h}
\ma\sum_{k=1}^\infty \frac 1k\l(2\sin \frac{kt}2\r)^{2h-2}  q^{kA}.
\eear
The content of the conjecture is that, while the  $GW_{A,g}$ are rational numbers, the BPS numbers $n_{A,h}$ are {\em integers.} (Gopakumar and Vafa also conjectured that  for each $A\in H_2(X, \Z)$, the coefficients of \eqref{Intro.1} satisfy $n_{A, h}=0$ for large $h$; we do not address this finiteness statement here.)    It is natural to enlarge the context by regarding this as a conjecture about the Gromov-Witten invariants
of any closed symplectic 6-manifold $X$ that satisfies the topological Calabi-Yau condition $c_1(X)=0$.

Formula (\ref{Intro.1}) can be viewed as a statement about  the structure of the space of solutions to the $J$-holomorphic map equation.  For a generic almost complex structure $J$, each $J$-holomorphic map is the composition $f= \phi\circ \rho$ of a multiple-cover $\rho$  and an embedding $\phi$. The embeddings are  well-behaved:  they  have no nontrivial automorphisms, and the moduli space of $J$-holomorphic embeddings is a   manifold.  But  multiply-covered maps cause severe analytical problems with transversality. 
In the symplectic construction of the GW invariants, these problems are avoided by lifting to a cover of the moduli space and turning on a lift-dependent perturbation  $\nu$ of the equation;  this  destroys the  multiple-cover structure and only shows that the numbers $GW_{A,g}$ are rational.   But it also suggests an interpretation of the GV formula:  the righthand side of \eqref{Intro.1}  might be a sum over embeddings, with the sum over $k$ counting the contributions of  the multiple covers of each embedding.

This viewpoint is very similar to C.~Taubes' work  \cite {t} relating  Gromov invariants to the  Seiberg-Witten invariants of 4-manifolds, and our approach   has been fundamentally influenced by Taubes.   It is also similar to the 4-dimensional situation described by Lee and Parker in  \cite{LP1} and \cite{LP2}.  In both cases, the set of $J$-holomorphic embeddings in each homology class is discrete and compact for generic $J$ --- a simplifying circumstance that does not appear to be true in the context of formula \eqref{Intro.1}.   Rather, for generic $J$ and with a fixed bound $E$ on area and genus,  the moduli space $\M_{emb}(X)$ of embeddings is a countable  set, possibly with accumulation points.  With this picture in mind, our proof is based on three main ideas.

 The first is the observation that, again for fixed $J$ and $E$,  the full moduli space $\oM(X)$ can be decomposed (in many ways) into finitely many ``clusters'' $\O_j$.  Each  cluster  consists of all of the $J$-holomorphic maps, including multiple covers,  whose image lies in the  $\ep$-tubular neighborhood of some  smooth, embedded $J$-holomorphic ``core curve'' $C\subset X$.  A cluster is an open and closed subset of the moduli space; it may have complicated internal structure, but there is a well-defined total contribution $GW(\O)$ of all the maps in the cluster to the series \eqref{Intro.1}.  These   contributions $GW(\O)$  are local,  depending only on  $\ep$ and  $J$ in the neighborhood,   and it suffices to show that the GV conjecture holds for the contribution of each cluster.

The second observation is that there exist certain standard ``elementary clusters" whose local invariants are explicitly computable.  Results of Junho Lee \cite{L} show that, for each embedded genus $g$ curve $C$, there exists an almost complex structure $J$ in an $\ep$ neighborhood $U$ of $C$ in $X$ that makes $C$ ``super-rigid'', meaning that all $J$-holomorphic maps into $U$ are in fact maps into $C$.  For $g=0$, one can take $J$ to be the standard structure of the bundle  $\O(-1)\oplus\O(-1)$, but for higher genus $J$ is a non-integrable almost complex structure.  In Section~3 we compute the GW series $GW(\O)$ of elementary clusters based on a calculation of Bryan and Pandharipande 
\cite{b-p-GWcurves}.   The resulting formula shows that the local version of the GV conjecture  holds for elementary clusters.

The proof is completed by an isotopy argument in the spirit of Taubes' work, and extending  arguments in  \cite{ip1}.  For a fixed cluster, we deform  $J$ in a neighborhood of the core curve to make it the $J$ of an elementary cluster.  During the isotopy, the cluster series $GW(\O_t)$ can change according to several types of   wall-crossing formulas.  For a generic isotopy,  the core curve could disappear in a ``creation-annihilation'' singularity. 
  To avoid this, we use a generic isotopy in which  the restriction of $J$ to the core curve is fixed;  singularities then occur only when two  core curves pass through one another momentarily.   In  Sections~6 and 7, we  use Kuranishi models to show that the cluster series   is invariant modulo contributions of finitely many clusters whose core curves have higher degree or genus.  The Gopakumar-Vafa  conjecture follows by induction.

 \vspace{-2mm}
  \begin{figure}[ht!]
\labellist
\small\hair 2pt
\pinlabel ${\text{$\ep$-nbd of $(C,J_0)$}}$ at -34 34   
\pinlabel ${\text{($C, J_1)$  elementary}}$ at  290 36  
\pinlabel ${\text{core curve $C$}}$ at 40 54  
\endlabellist 
\centering
\includegraphics[scale=1]{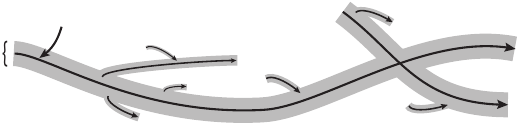}
\caption{\small As $J_t$ changes, embedded curves of higher genus or degree can emerge from,  or sink into,  an $\ep$-tubular neighborhood of $C$, and the core curve $C$ can pass through another embedded curve with the same  genus and degree.}
\end{figure}

 
\vspace{-2mm}

Our main results, Theorems~\ref{Theorem7.1} and \ref{Theorem8.1}, can be stated as a structure theorem.

 \begin{structuretheorem*}
	 \label{T.GWintro-0} For any closed symplectic Calabi-Yau 6-manifold $X$, there exist {\em unique integer} invariants  $e_{A, g}(X)$,  indexed by non-zero  classes $A\in H_2(X, \Z)$ and  genera $g\ge 0$, such that the series \eqref{Intro.1}  has the form
\bear\label{IntroTheoremEq} 
GW(X)\ =\ \sum_{A\ne 0}\sum_{g\ge 0}\  e_{A, g}(X) \cdot GW^{elem}_{g}(q^{A}, t),
\eear
where $ GW^{elem}_{g}(q,t)$ is the universal power series \eqref{3.defGelem}, which  depends on $g$,  but not on $X$. Furthermore, all coefficients $n_{A, g}$ in \eqref{Intro.1} are integers.

\end{structuretheorem*}
There is an extensive literature revolving around this GV conjecture.  J.~Bryan and R.~Pandharipande have a series of papers about it, including two (\cite{b-p-BPS} and \cite{b-p-GWcurves}) relevant to our approach.  For algebraic 3-folds,  several  BPS-type  integer invariants  have been defined using holomorphic bundles, including the Pandharipande-Thomas \cite{p-t} and Donaldson-Thomas  invariants, with conjectural GV-type correspondences GW/DT/PT between them. For toric 3-folds,  Maulik, Oblomkov, Okounkov and Pandharipande proved the GW/DT correspondence by calculating both sides explicitly in  a computational tour de force \cite{mnop}. Pandharipande  and Pixton \cite{p-p} established the GW/PT correspondence for CY complete intersections in products of projective spaces.  Other instances have been observed when a change of variables in the GW series produces integer invariants, including a formula for Fano classes ($A\in H_2(X)$ with $c_1(X)A>0$) in symplectic 6-manifolds  proved by  A.~Zinger \cite{z},  and the computation of Klemm and Pandharipande   for  Calabi-Yau 4-folds \cite{k-p}.  In  Section~\ref{section9} we combine our result on the Calabi-Yau classes   first with Zinger's to obtain a GV-type formula for  all symplectic 6-manifolds, and  then with  Klemm and Pandharipande's to obtain a GV-type formula for  genus zero invariants of semipositive symplectic manifolds.

We thank the referees for their meticulous reviews and numerous insightful suggestions.

 \vspace{5mm}
\setcounter{equation}{0}
\section{Curves in symplectic Calabi-Yau 6-manifolds} 
\label{section1}
\bigskip

The Gromov-Witten invariants of a closed symplectic manifold $(X,\w)$ are
constructed in two steps.  One first forms the universal moduli space and its
stabilization-evaluation map $se$
\bear
\label{0.1}
\xymatrix{
\oM(X) \ar[r]^{  se  \qquad \ }\   \ar[d]_\pi &\ \  \bigsqcup_{g,n} \oM_{g,n}\times X^n \\
\J & 
}
\eear
over a space  $\J$ of  $\w$-tame almost complex structures  on $X$.  This moduli space 
consists of equivalence classes (up to reparametrizations of the domain) of pairs $(f, J)$,  where $f:C\to (X, J)$ is a stable pseudo-holomorphic map whose domain $C$ is a nodal marked Riemann surface;  it has components $\oM_{ A, g, n}(X)$ labelled by the genus $g$ of $C$, the number $n$ of marked points,  and the homology class $A=f_*[C]\in H_2(X; \Z)$ with $\w(A)\ge 0$.

 As is standard in the subject,  we take $\J$ to be a space $\J^l$ of $C^l$ almost complex structures with $l$ large (and sometimes  $\J=\J^\infty$),    take  $f$ in  a corresponding  space of maps (described in Section~4), and give $\ov\M(X)$   the Gromov  topology.       The results in Sections~1-3  depend on    the specifics of these spaces only through Lemma~\ref{Lemma1.1}, whose rather technical proof is deferred until Section~5 and Appendix~A.

 It is frequently convenient to define the {\em  energy} of a triple $(A, g, n)$  by 
$$
E(A, g, n) \ =\ \max\{\w(A), g, n\}  \ge 0
$$
and to restrict attention to the subset $\oM^E(X)$  ``below energy $E$'',
meaning the union of all components with $E(A, g, n)\le E$, and the corresponding
fibers $\oM^{J, E}(X)$ of (\ref{0.1}).   The restriction of $\pi$ to  each $\oM_{A,g,n}(X)$ is proper, and the fibers
carry a $d$-dimensional  virtual fundamental class 
\bear
\label{1.VFC}
[\oM^J_{A,g,n}(X)]^{vir}\ \in \  \cH^d(\oM^J_{A,g,n}(X); \Q)^\vee, 
\eear
where 
\bear
\label{1.dimformula}
d\ =\ 2c_1(X)A+(\dim X-6)(1-g) +2n, 
\eear
 and $\cH^*(\cdot)^\vee$ denotes the dual of  \v{C}ech cohomology with rational coefficients  (cf. \cite[Definition~9.3.1]{pardon}).  Moreover, \eqref{1.VFC} is deformation invariant in the sense that for every path  $\gamma$ in $\J$ from $J_0$ to $J_1$, 
\bear\label{defm.invar}
[\oM^{J_0}_{A,g,n}(X)]^{vir}=[\ov \M^{J_1}_{A,g,n}(X)]^{vir}\quad \mbox{ in } \quad  \cH^d(\oM^\gamma_{A,g,n}(X))^\vee
\eear
(cf. the proof \cite[Lemma~9.3.2]{pardon}). Here $\ov\M^\gamma$ denotes the parameterized moduli space 
\best
\ov\M^\gamma= \big\{\, (t, ([f], J))\in [0,1]\times \ov\M(X)  \ \big|\  \gamma(t)=J  \; \big\}
\eest
obtained by pulling back $\ov\M$ along $\gamma$.

\medskip

Gromov-Witten invariants are especially simple if $c_1(X)=0$ and $\dim X=6$;
such spaces are often called {\em symplectic Calabi-Yau 6-manifolds}.  In this case all
terms in (\ref{1.dimformula}) vanish when $n=0$; the virtual fundamental class then has dimension $d=0$ for all $g$ and $A$, and we drop $n$ from the notation.  In this case, the Gromov-Witten  invariants  
\bear\label{gw.ag}
 GW_{A,g}(X)= \lg \; [\oM^J_{A,g}(X)]^{vir}, 1 \rg \in \Q
\eear
are obtained by pairing the virtual fundamental class with $1$ in $\cH^0(\oM)$. These are assembled in a formal power series
\bear
\label{1.GWformula.number}
GW(X)\ =\ \sum_{\w(A)>0} \sum_{g=0}^\infty\;\; GW_{A, g}(X)\;\;
t^{2g-2}\ q^A 
\eear
 in the ``rational Novikov ring'' $\Lambda$ generated by $t$ and $\{q^A\}$ with  $tq^A= q^At$ and $q^{A+B}= q^A q^B$,  and whose elements have finitely many non-zero terms below each energy level.   Note that the first sum in \eqref{1.GWformula.number}, over all positive $A\in H_2(X; \Z)$, omits  any contributions  from the class $A=0$.  For consistency,
the term  ``$J$-holomorphic map''  will always mean a {\em non-trivial} map  (i.e. $A\not= 0$), and a
``$J$-holomorphic curve'' in $X$ is the image of such a map. When $f$ is an embedding, we identify $C$ with its image  $f(C)$ in $X$. 
\smallskip

\begin{rem}
\label{2.Remark}
The use of virtual fundamental cycles here and in Section~2  
 makes the presentation clear and succinct,
but is not essential for understanding the arguments in this paper. 
In fact, all of the symplectic manifolds $X$ that we consider are semipositive, and 
one might alternatively regard the GW invariants,  and the  local contributions to the GW invariants introduced in Section~2,  as counts of perturbed $J$-holomorphic maps.
\end{rem}

\smallskip
Using the terminology of \cite[\S 2.5]{ms},  a point $x\in C$ is called an {\em injective point} for a   map $f:C\to X$ if $df(x)\not= 0$ and $f^{-1}(f(x))=\{x\}$ when $C$ is smooth;  if $C$ is nodal  we also require that $x$ is not a node.  A  pseudo-holomorphic map $f:C\to X$ from a   nodal   (not necessarily connected)  curve is {\em simple} if it has an injective point on each  irreducible component.  The open subset of simple maps in a moduli space will be denoted by the subscript {\em simple}; for example
\begin{equation}
\label{1.Msimple}
\oM^{J,E}(X)_{simple}
\end{equation}
denotes the set of simple maps in $\oM^{J,E}(X)$, while $\M(X)_{simple}$ denotes the open subset of   $\oM(X)$ consisting of simple maps with smooth domain. The Micallef-White Theorem, \cite{mw} or \cite[Prop. 2.5.1]{ms}, implies that the set of injective  points of a simple $J$-holomorphic map is open and dense in $C$, and hence the image under $f$ of this set is a submanifold of $X$.

 A  pair $ p=(f,J)$ representing a point in $\ov\M^{J,E}(X)$   with smooth domain is {\em regular} if the linearization $D_p$, given by  \eqref{1.formulaL}  and completed in Sobolev norms as in \eqref{def.dl.comp},   is onto.  It is a {\em regular embedding} if it is both regular and  $f$ is an embedding. Since the index of $D_p$ is 0 by \eqref{1.dimformula},  each regular  pair  has a well-defined $\sign (f, J)=\pm 1$,  given by the mod 2 spectral flow from $D_p$ to any invertible complex operator.   Finally, by a Baire subset of the parameter space we mean a countable intersection of open and dense sets. 
\begin{lemma}
\label{Lemma1.1}
Let $(X,\w)$ be a symplectic Calabi-Yau 6-manifold.  Then for each $E>0$ there is a Baire subset $\J_E^*$ of $\J$ such that for each $J\in\J_E^*$
\begin{enumerate}[(a)]
\item All simple $J$-holomorphic maps below energy $E$ are regular  embeddings with  disjoint images.
\vspace{1.5mm}
 
\item The projection $\pi$ in (\ref{0.1}) is a local diffeomorphism around each regular  embedding.
\end{enumerate}
\end{lemma}
Part (a) of Lemma~\ref{Lemma1.1}  is proved  as Corollary~\ref{corA.3}, and part (b) is proved  as Proposition~\ref{MmfdThm}a; these proofs use standard techniques.  For our purposes, it is best to enlarge $\J_E^*$ to a set $\Ji^E$ that emphasizes slightly weaker properties.
\begin{defn}
\label{1.defJisol}
 Denote by $\Ji^E$ the set of all $J\in \J$such that the moduli space 
(\ref{1.Msimple}), with the Gromov topology, consists of isolated points  that are  embeddings (not necessarily regular)  with  disjoint images.
\end{defn}

\begin{cor} 
\label{Cor1.3}
$ \Ji^E$ is  dense in $\J$.  \end{cor} 
\begin{proof}  For each $J\in \J_E^*$,  each simple $J$-holomorphic map $f$ is an embedding,  and is regular by part (a) of Lemma~\ref {Lemma1.1}. By part (b), each such point $(f, J)$ is an isolated point 
of  the fiber $\M^{J,E}(X)_{simple}$ of $\pi$. 
\end{proof}

Unfortunately, the images in $X$ of the pseudo-holomorphic maps that appear in Lemma~\ref{Lemma1.1} may accumulate. To focus on the images, 
consider the space  $\mathrm{Subsets}(X)$ of all  non-empty compact subsets of $X$.  Fix a background Riemannian metric on $X$ with distance function $d$.  Then 
$\mathrm{Subsets}(X)$ is a metric space with the  Hausdorff distance, defined by
$$
d_H(A, B)\ =\  \sup_{a\in A}\,  \inf_{b\in B}\, d(a,b)\ +\  \sup_{b\in B}\,\inf_{a\in A}\, d(a,b)
$$
for  $A, B \subseteq X$. Let $c$ be the  ``underlying curve'' map 
\bear
\label{1.firstunderlyingcurve}
c:\oM(X)\ma\lra    \mathrm{Subsets}(X)\times \J
\eear
 that associates to each $(f,J)$ the pair $(f(C), J)$ with $f(C)$  regarded as subset of $X$.  This map $c$ is continuous, and Gromov compactness implies that its restriction to $\oM^E(X)$ is proper.  Let $\C(X)$, $\C^E$ and $\C^{J,E}$ denote, respectively, the images of  $\oM(X)$,  $\oM^{E}(X)$ and  $\oM^{J,E}(X)$ under $c$.  With this notation, there is a commutative diagram
\bear
\label{1.maptoC}
\begin{tikzcd}[row sep=tiny]
\oM^{E}(X) \arrow{dd}{\pi} \arrow{dr}{c} & \\
&   \C^E.\arrow{dl}  \\
  \J & 
\end{tikzcd}
\eear
 Viewed differently,  convergence in $\C^E$ defines a topology on $\oM^{E}(X)$ that we will call the ``rough topology''.

 In general, non-trivial $J$-holomorphic maps $f:C\to X$ from nodal curves  can be multiple covers, and simple maps can converge to multiply covered maps. Definition~\ref{1.defJisol} constrains how limits are multiply covered, as the following lemma shows.
 
\begin{lemma} 
\label{Lemma1.4} 
For   $J\in \J^E_{isol}$,
\begin{enumerate}[(a)]
\item Every $J$-holomorphic map $f:C\ra X$ below energy $E$ is a composition $\phi\circ \rho$ 
of a holomorphic map $\rho:C \ra C_{red} $ of complex curves and a $J$-holomorphic embedding $\phi:C_{red}\ra X$.  This  decomposition is unique up to reparametrizations of $C_{red}$. When $f$ is simple, $\phi=f$. 
          
\item  If $f:C\ra X$ is a limit   in the rough topology of a sequence  $\{f_n\}$ in $\oM^{J,E}(X)$ with $d_H(f_n, f)\not= 0$ for all $n$, then  the factorization $f=\phi\circ \rho$ has  either $\deg \rho>1$ or $\mbox{genus}(C)>\mbox{genus}(C_{red})$.
\end{enumerate}
\end{lemma}

\begin{proof}
The  disjoint union of the irreducible components of $C$  is   a smooth closed curve $\wt{C}$ called the normalization of $C$.  Let $\wt{f}:\wt{C}\to X$ be the composition of the  canonical map $\wt{C}\to C$ with $f$.   The arguments of \cite[\S 2.5] {ms} show that $\wt f$ factors as $\wt f =\phi\circ \wt\rho$ where
$\phi: C_{red}\ra X$ is a simple $J$-holomorphic map from a  smooth, possibly disconnected domain $C_{red}$,  and $\wt\rho:\wt C \ra C_{red}$ is a map of complex curves; $\wt\rho$ may take some components of $\wt C$ to points. But the assumption $J\in \J^E_{isol}$ implies that  $C_{red}$ is connected and $\phi$ is an embedding; $\wt\rho$ then descends to a holomorphic map $\rho:C \to C_{red}$, unique up to reparametrization, with $f=\phi\circ\rho$.

Part (b) immediately follows from Gromov compactness and the fact that simple maps are isolated for $J\in  \J^E_{isol}$, and therefore the limit  map $f=\phi\circ\rho$ is not simple.\end{proof}

To each class $A\in H_2(X, \Z)$ in the positive cone $\om(A)>0$ we associate a positive integer 
\best
d(A)= \ell cm\; \big\{\; k \in \mathbb N_+ \; \big|\; A= k B \mbox{ where  } B\in H_2(X, \Z)\;  \big\} 
\eest 
called the {\em degree} of $A$, where $\ell cm$ denotes the lowest common multiple. Let $\Omega(d)$ be the number of prime factors of $d$, counted with multiplicity.
For any map $f$ from a genus $g$ curve representing a degree $d$ class, we define its {\em  level} to be  
 \bear
 \label{defLevel}
 \ell(f) = \Omega(d)+g.
 \eear

The components of the moduli space are filtered by the degree and genus, and therefore by their level;  the level  filtration will be used frequently in later sections.  For each $E$, the  sets
\bear
\label{2.defM_m}
 \oM^E_{m}(X)\ =\    \big\{ (f, J)\in \oM^E(X)\  \big|\  \ell(f)\le m  \big\}
\eear
filter $\oM^E(X)$, and their images under (\ref{1.maptoC}) 
$$
 \C^E_{m}\ =\ c\left(\oM^E_{m}(X)\right)
 $$
filter the image  $\C^E=c(\oM^E(X))$. 

For each $J\in \J_{isol}^E$,  the map $c$, when applied to multiply-covered maps, decreases the level but respects the filtration; for such   $J$  the fiber $\C^{J,E}_m$  of $\C^E_m\to \J$  is the collection of embedded $J$-holomorphic curves with level at most $m$. In this notation, $m=0$ corresponds to genus zero curves representing primitive classes,  and  $ \C^E_{m}$ is the collection of embedded pseudo-holomorphic curves in $X$ with $g+\Omega(d) \le m$ and energy at most $E$. 
\begin{lemma} 
\label{Lemma1.5} 
For any fixed $J\in \J^{E}_{isol}$, 
\begin{enumerate}[(a)]
\item $\C_m^{J, E}\subseteq \C_{m+1}^{J, E}$ is a filtration of $\C^{J, E}$  with $\C^{J, E}_{m}=\C^{J, E}$ for $m$ large.
\item $\C_{m}^{J, E}$ and $\C^{J, E}=\bigcup \C_{m}^{J, E}$ are compact countable subsets of the metric space $\C(X)$.\smallskip
\item For any neighborhood $U$ of $\C_{m-1}^{J, E}$,  the set $\C_{m}^{J, E}\setminus U$ is a finite collection of embedded  $J$-holomorphic curves.
\end{enumerate}
In particular, there are finitely many genus zero $J$-holomorphic curves with energy less than $E$ representing primitive classes. 
\end{lemma}
\begin{proof} 
The inclusion in (a) is true by definition, and the  second part of (a) holds  because, by Gromov compactness,
  only finitely many homology classes are represented by $J$-holomorphic maps below energy $E$. 

Next, each set $\C_{m}^{J, E}$ is compact because it is the  image of a compact set, namely the fiber $\oM^{J, E}_{m}(X)$ of  (\ref{2.defM_m}),  under the continuous map (\ref{1.maptoC}). Then $\C_{m}^{J, E}\setminus U$ is a closed subset of the compact metric space $\C_{m}^{J, E}$, so any infinite sequence $\{C_i\}$ has an accumulation point $C_0$.  Because only finitely many homology classes are represented by $J$-holomorphic maps below energy $E$ we may assume, after passing to a subsequence, that they all have the same genus and same homology class $[C_i]=k\beta$ for the same primitive class $\beta$ and the same $k$  with $\Omega(k)+g\le m$.  By Lemma~\ref{Lemma1.4} the limit is a multiple cover of a curve of a strictly lower level, but that's impossible  because $U$ is open. Thus $\C_{m}^{J, E}\setminus U$ is finite. Finally, taking $U_k$ to be the $1/k$ tubular neighborhood of $\C^{J, E}_{m-1}$,  we conclude that $\C_{m}^{J, E}\setminus\C_{m-1}^{J, E} = \bigcup_k\left(\C_{m}^{J, E}\setminus U_k\right)$ is countable, and hence $\C^{J, E}_m$ and $\C^{J, E}$ are countable. 
\end{proof}

 \vspace{5mm}
\setcounter{equation}{0}
\section{Clusters in symplectic manifolds} 
\label{section2}
\bigskip

Now suppose that $X$ is a symplectic Calabi-Yau 6-manifold (we retain this assumption  until Section~4).
For each subset $S\subset \J$, we can consider the moduli space $\oM^S(X)=\pi^{-1}(S)$ over $S$ or, with an energy bound,  
$$
\oM^{S, E}(X)=\Big\{  ([f],J)\ \Big|\ [f]\in \oM^{J,E}(X),\  J\in S  \Big\}.
$$ 
A {\em decomposition}  of the moduli space $\oM^{S, E}(X)$  is a way of writing it as a finite disjoint union $\bigsqcup_i \O_i$ of subsets $\O_i$ that are both open and closed  in the Gromov topology.  Given such a decomposition and a  compact subset $V$ of $S$, the sets $\{ \O_i\cap \pi^{-1}(V)\}$ are a decomposition of $\oM^{V, E}(X)$, giving a natural isomorphism
 \bear\label{hom.decomp}
 \cH^*\left(\oM^{V, E}(X)\right)^\vee\ \cong\ \bigoplus_i\, \cH^*(\O_i\cap \pi^{-1}(V))^\vee.
 \eear
Note that for every $J\in S$ and every $A, g$ with energy at most $E$, the inclusion  $ \oM^{J}_{A, g}(X) \hookrightarrow \oM^{J, E}(X)$ induces a map
\best
\cH^*\Big( \oM^{J}_{A, g}(X) \Big)^\vee\ra \cH^*\Big( \oM^{J, E}(X)\Big)^\vee.
 \eest
 whose image is the lefthand side of \eqref{hom.decomp} with  $V=\{J\}$.  
The  image of virtual fundamental class $ [\ov\M^J_{A, g}(X)]^{vir}$ then decomposes under \eqref{hom.decomp} into a sum of components 
$$
pr_i [\ov\M^J_{A, g}(X)]^{vir}\in \cH^*(\O_i\cap \pi^{-1}(J))^\vee.
$$
 As in \eqref{gw.ag},  we set 
 \bear\label{defGW.ag}
GW_{A, g}(\O_i\cap \pi^{-1}(J))=\lg \;pr_i  [\ov\M^J_{A, g}(X)]^{vir} , 1\;\rg \in \Q
\eear
and define the {\em contribution} of $\O_i\cap \pi^{-1}(J)$ to the GW series to be  the sum of the form \eqref{1.GWformula.number} whose
 coefficients are given by \eqref{defGW.ag} for all $A, g$ with energy at most $E$, and are 0 otherwise. Then 
\bear
\label{2.GWsum}
 GW^E(X)\ =\ \sum_i  GW^E(\O_i\cap \pi^{-1}(J))
\eear
where, on both sides, $GW^E$ denotes  the $GW$ series truncated at energy $E$.

Equation \eqref{defm.invar} implies that the coefficients \eqref{defGW.ag} are  deformation invariant, as follows.  For any path $\gamma$ in $S$ from 
$J_0$ to $J_1$, and any $A, g$ with energy at most $E$, we can take $V$ to be the image of $\gamma$, and consider the map 
$$
\ov{\M}^\gamma_{A,g}(X) \ra  \ov{\M}^{V,E}(X) 
$$
induced by $(t, [f], J)\mapsto ([f], J)$.  The equality  \eqref{defm.invar}  pushes forward by this map to give an equality in the  group  on the lefthand side of \eqref{hom.decomp}. Applying the isomorphism \eqref{hom.decomp} and projecting onto the $i$'th component then shows that  the coefficients \eqref{defGW.ag} for $J=J_0$ and $J=J_1$ are equal.

 \medskip

 \begin{lemma}
 \label{lemma2.1} 
Each open subset $U$ of  $\C(X)$ with $\bd U \cap \C^{J, E} = \emptyset$ has a well-defined contribution $GW^{E} (U, J)$  to $GW^E(X)$.  The collection of $J$ for which  $\bd U \cap \C^{J, E} = \emptyset$ is open in   $\J$, and the contribution $GW^{E} (U, J)$ is locally constant as a function of $J$.  
 \end{lemma}
 \begin{proof} The assumption implies that the intersections of both $U$ and its complement $U^c$ with 
 $\C^{J, E}$ are open subsets of the image $\C^{J, E}$ of $c$. Since $c$ is continuous,  $\O=c^{-1}(U\cap \C^{J, E})$ and $\O^c=c^{-1}(U^c \cap \C^{J, E})$ are  open and closed subsets of $\oM^{J,E}(X)$. Define  $GW^E(U, J)$ to be the contribution $GW^E(\O\cap \pi^{-1}(J))$ of $\O$ to the sum \eqref{2.GWsum} associated to the decomposition $\O\sqcup \O^c$.

 Next note that the condition $\bd U \cap  \C^{J, E} =  \emptyset$ is an open condition on $J$.  To see why, fix $U$ and $E$.  If   this condition held for some $J$, but failed to hold for a sequence  $J_k\ra J$ in $\J$, there would be a sequence of $J_k$-holomorphic curves $C_k$ in $\bd U$ with energy $E(A,g,n)$ bounded by $E$.  Applying Gromov compactness, one could find  a subsequence converging to a $J$-holomorphic curve $C$.  But  $\bd U$ and $\C^{J, E}$ are both closed in the Gromov topology,  so the limit curve $C$ would lie in $\bd U \cap  \C^{J, E}$, contradicting the hypothesis.

Therefore there exists a  ball $V$ around $J$ in $\J$ such that  $\bd U \cap \C^{E,V} = \emptyset$. This gives rise to the same decomposition $\O\sqcup \O^c$ but now of the moduli space 
$\oM^{V, E}(X)$ over the entire ball $V$, therefore the contribution of $U$ is defined for each $J\in V$ and it is constant on $V$. 
  \end{proof}
 
The geometric content of the contribution  $GW^{E} (U, J)$ is most clearly seen by choosing $U$  of the form $B(C,\ep)\ti \J$ for  a ball  $B(C,\ep)$ of small radius $\ep$ in the Hausdorff distance centered at a $J$-holomorphic curve $C$.  The    subset   of $ \oM^J(X)$ that lies in $c^{-1}(U)$ is then a collection of $J$-holomorphic maps whose images are uniformly $\ep$-close to $C$ in $X$.   We will call such a collection a cluster if the following properties hold.
\begin{defn}
\label{D.cluster} 
 A {\em cluster} $\O=(C, \ep, J)$ in $X$ below energy $E$ (an ``$E$-cluster'') consists  of an almost complex structure $J\in \J$, an embedded $J$-holomorphic curve $C$ and a radius $\ep>0$ with the following properties:
\begin{enumerate}[(a)]
\item All  non-constant $J$-holomorphic maps in the ball $B(C, \ep)$ with energy $\le E$ represent $k[C]$ for some $k\ge 1$, and have genus $g\ge g(C)$; 
\item $C$ is the only $J$-holomorphic map in its degree and genus in the ball $B(C, \ep)$.
\item There are no $J$-holomorphic curves with energy  $\le E$ at precisely $\ep$ Hausdorff distance from  $C$.
\end{enumerate}
\end{defn}
The curve $C$ is called the {\em core} of the cluster.    Note that, by the definition of Hausdorff distance, curves in $B(C,\ep)$ lie  in the $\ep$-tubular neighborhood of $C$ in $X$.

The next lemma shows that small balls in the Hausdorff metric are often clusters.  In fact, conditions (a) and (b) of Definition~\ref{D.cluster}  are automatic for small $\ep$ when $C$ is regular.  Condition~(c) is the important one: it  implies that $\O=(C, \ep, J)$  has a well-defined contribution 
\bear\label{GW(O)inQ}
GW^E(\O)\in  \Lambda.
\eear

\begin{lemma}[Cluster existence]
\label{L.cluster.exists} 
For each $J\in \J_{isol}^{E}(X)$ and each  simple $J$-holomorphic curve $C$,  the set $S$ of $\ep>0$ for which the ball $B(C,\ep)$  is an $E$-cluster is open and dense in a non-empty interval $[0,\ep_C]$, and  the contribution \eqref{GW(O)inQ} is locally constant on $S$. 
\end{lemma}
\begin{proof} 
By definition, for any $J\in \J_{isol}^{E}(X)$, all simple $J$-holomorphic curves are embedded and isolated in their degree and genus. Since $\C^{J, E}$ is compact, and an embedded curve $C$ can appear as an accumulation point only of curves representing $k[C]$ and having genus at least that of $C$, Lemma~\ref{Lemma1.4} implies that there is an $\ep_C>0$ such that  conditions (a) and (b) of Definition~\ref{D.cluster}  hold for all $\ep\le \ep_C$.

Finally, by Lemma~\ref{Lemma1.5}, the image of $\C^{J,E}$ under the distance function $B(C,\ep_C)\to [0,\ep_C]$ is countable and compact, so its complement -- which is the set of $\ep$  that satisfy condition (c) of  Definition~\ref{D.cluster} -- is open and dense.
\end{proof}
  
\begin{prop}[Cluster decompositions]
\label{P.cluster.decomp.exist} 
Given $E$ and  $J\in \J_{isol}^E$, each open subset $U$ of  $\C(X)$ such that $\bd U \cap \C^{J, E}=\emptyset$ has a finite $E$-cluster decomposition $\{\O_i=(C_i, J,\ep_i)\}$, and hence
\bear
\label{GW(U)=sum} 
 GW^E(U)\ =\ \sum_i  GW^E(\O_i).  
\eear
\end{prop}

\begin{proof}
We will inductively construct cluster decompositions of the sets $U_{m} = U\cap \C^{J,E}_{m}$, beginning with the trivial case $U_{-1}=\emptyset$. This  induction is finite because $\C_{m}^{J,E}= \C^{J,E}$ for $m$ sufficiently large by Lemma~\ref{Lemma1.5}a.

Suppose that $\{B_i=B(C_i,\ep_i)\}$ is a cluster decomposition of $U_{m}$.  This means that  the balls $B_i$ are disjoint, that  the compact set $U_{m}$ lies in  $V=\bigsqcup B_i$, and there are no $J$-holomorphic curves on $\partial V = \bigsqcup \partial B_i$.    Lemma~\ref{Lemma1.5}   shows that $U_{m+1}\setminus V$ is a finite collection of curves $\{C_j\}$.  None of these $C_j$ lie on $\partial V$, so   we can choose radii $\ep_j>0$ such that the balls $B_j'=B(C_j, \ep_j)$ are clusters (by Lemma~\ref{L.cluster.exists})  and are disjoint from each other and from the balls $B_i$.  These clusters $B_j'$, together with the original $B_i$ are a cluster decomposition of $U_{m+1}$, completing the induction step.
\end{proof}
\begin{cor}[Cluster refinement]
\label{L.cluster.ref} 
Fix any cluster $\O=(C, J,\ep)$ with $J\in \J_{isol}^{E}$. For any  $\ep' \in (0, \ep)$ for which  $\O'=(C, J, \ep')$ is a cluster, there exists finitely many higher level clusters $\{\O_i=(C_i, J,\ep_i)\}$ such that 
\bear\label{GW(o)=GW(oo)} 
 GW^E(\O)\ =\ GW^E(\O')+ \sum_i  GW^E(\O_i).
\eear
\end{cor}
\begin{proof} Consider $U= B(C, \ep)\setminus B(C, \ep')$. Because both $\O$, $\O'$ are clusters, (a) there are no curves in $\C^{J, E}$ on $\bd U$ and (b) all the curves in $\C^{J, E} \cap U$ have strictly higher level compared to that of $C$ (as $C$ is the only bottom level curve in both clusters). 
By Proposition~\ref{P.cluster.decomp.exist}, there exists a cluster decomposition \eqref{GW(U)=sum} of $U$, 
where  the $\O_i$ are strictly higher level. But condition (a) also implies that $U$ and $\O'$ give a decomposition of $\O$ so $GW(\O)= GW(\O')+ GW(U)$, which implies  \eqref{GW(o)=GW(oo)}. 
\end{proof}

 \vspace{5mm}
\setcounter{equation}{0}
\section{Elementary clusters and their contributions}
\label{section3}
\bigskip

The GW series can be explicitly calculated for one very special type of embedded curve.  In this section we describe how this can be done by combining  ideas already in the literature.  We first construct such ``elementary curves'' using a remarkable non-integrable almost complex structure discovered by Junho Lee \cite{L}, and then point out that the GW series for these curves has been calculated by J.~Bryan and  R.~Pandharipande \cite{b-p-GWcurves}. 
\begin{defn}\label{3.D.elem}
In a symplectic Calabi-Yau 6-manifold $X$, a cluster $\O\!=(C, \ep, J)$ is called  {\em elementary}  if
\begin{enumerate}[(a)]
\item  The core curve $C$ is {\em balanced}, meaning that  its normal bundle splits  as  $N_C=L\oplus L$ in such a way that the normal operator $D^N$,   given by  \eqref{4.defn.normal.op} below,  splits as $D'\oplus D'$. 
\item The only non-trivial $J$-holomorphic maps into $B(C, \ep)$ are multiple covers  of the embedding $C\hookrightarrow X$.
\item For each such cover $\rho$, the pullback operator $\rho^*D'$ is injective ($C$ is {\em super-rigid}).  
\end{enumerate}
\end{defn}
Property~(c) implies that the core $C$ of an elementary cluster is a regular $J$-holomorphic map, and  one can show it also implies (b) for sufficiently small $\ep>0$ by the rescaling argument of \cite{IP2}.  Property (a) allows us to actually calculate the GW contribution.

 When  $C$ is a rational curve, the unit disk bundle in $\O(-1)\oplus \O(-1)$ is an elementary cluster.  The following proposition uses non-integrable almost complex structures to construct similar examples  for any curve. The proof begins with the choice of a  spin structure on $C$, i.e. a holomorphic line bundle $L\to C$ together with a (holomorphic) identification of $L^2$  with the canonical bundle $K_C$ of $C$.

\begin{prop}
\label{3.Junhocurves} 
For every smooth complex curve $C$, there exists an elementary cluster whose  core is $C$.  
\end{prop}

\begin{proof}
Fix a curve $C$ of genus $g$,  a spin structure $L$, and a K\"{a}hler structure $(J, g, \w)$ on the total space $Y$ of $L\oplus L\to C$ compatible with its holomorphic structure. The canonical bundle  $K_Y=\pi^*(L^{-2}\otimes K_C)$ of $Y$ is then trivial, so $Y$ is a Calabi-Yau 6-manifold.  We will perturb $J$ to obtain an almost complex structure on the unit disk bundle $U\subset Y$ that makes  $U$ an elementary cluster.

First consider the total space $Z$ of $p:L\to C$.  Its canonical  bundle $K_Z=T^{2,0}Z$ is the pullback $p^*(L^{-1}\otimes K_C)= p^*(L)$, which has a canonical section $\beta$ that vanishes transversally along the zero section.  Pullback  $\beta$, regarded as a 2-form, by the bundle projections $p_1, p_2:Y\to Z$ onto the first and second copy of $L$ and set $\alpha=p_1^*\beta+ p_2^*\beta$.  Then 
$\alpha$ is a closed $(2,0)$-form on $Y$ that vanishes to first order along the zero section of $L\oplus L$, which is the core curve $C$ of the disk bundle $U\subset Y$.  Following Junho Lee,  define a bundle map $K_\alpha:TY\to TY$ by $g(u,K_\alpha v)=(\alpha+\overline{\alpha})(u,v)$ for all $u,v\in TY$ and set
$$
J_\alpha = (Id +JK_\alpha)^{-1}J (Id +JK_\alpha).
$$
 Then $J_\alpha$ is  an $\w$-tame almost complex structure  on $U$  after replacing  $\alpha$ by $t\alpha$ for small $t>0$ (cf. Section~2 of \cite{L}).   In this context, Lee proved that {\em the image of every non-trivial $J_\alpha$-holomorphic map into $D$ is everywhere tangent to $\ker K_\alpha$} (cf. \cite[(2.4)]{L}). It is straightforward to check that at each point $p\in Y$ not on the zero section,  $\ker K_\alpha$ is vertical. Consequently,  any  map  whose image is tangent to $\ker K_\alpha$ lies in a fiber of $Y\to C$ or in the zero section.  We conclude that every  $J_\alpha$-holomorphic map into $U$ that represents $k[C]$ for some $k\not= 0$ is a map into the core curve $C$ of the disk bundle $U$.  
 
  Along the zero section of $Y$ and for $v\in TC$, $\nabla_v K_\alpha$ decomposes under the splitting $TC\oplus L \oplus L$   as  $0 \oplus \nabla_v K_\beta \oplus  \nabla_v K_\beta$.   Correspondingly, the  normal operator $D^N$, given by  \eqref{4.defn.normal.op},  splits as  $D^N_\beta\oplus D^N_\beta$ as can be seen from \cite[(8.4)]{LP1}.  As in  Section~4 of \cite{LP1}, the normal projection of $D_\beta$ is the sum $\del +R_\beta$ of the  $\del$-operator on $L$ and a bundle map $R_\beta: L \to T^{0,1}C\otimes L$ that satisfies $JR_\beta=-R_\beta J$.  The injectivity condition (c) of Definition~\ref{3.D.elem} is then exactly the statement of Proposition~8.6 of \cite{LP1}.
\end{proof}

\medskip

 By Lemma~\ref{lemma2.1} an elementary cluster has a well-defined  $\Lambda$-valued series $GW(\O)$, 
 independent of $\ep$. As usual,  there is an associated {\em disconnected} invariant
 $$
 Z(\O)\ =\ \exp(GW(\O))
 $$
obtained by exponentiating in the Novikov ring. It turns out that $Z(\O)$ is the more easily calculated.
    
\begin{prop} 
\label{Prop3.3}
The disconnected GW invariant of an elementary cluster $\O$  whose core $C$ has genus $g$ is given by 
 \bear
\label{BPelemContribution}
Z(\O) \  =\  1+\sum_{d\ge 1}  \sum_{\mu\vdash d}\; \prod_{\square\in \mu}  
\l(2\sin \frac{h(\square)t} 2\r)^{2g-2} q^{dC},
\eear
where the second sum is over all partitions $\mu$ of $d$, the product is over the boxes in the Ferrers diagram of  $\mu$, and $h(\square)$ is the hooklength 
of $\square\in \mu$.  
\end{prop} 

\begin{proof} 
 Because   the linearization  on covers of the core curve $C$ is injective, the contribution to the GW invariant of its multiple covers can be calculated using the Euler class of Taubes obstruction bundle (this is the 6-dimensional version of the setup in  \cite{LP2}, and a special case of Theorem 1.2 in \cite{z}).  

Consider the moduli space $\ov\M_{d, \chi}^\circ(C)$ of degree $d$ holomorphic maps $\rho$ to $C$ whose domain is possibly disconnected and has Euler characteristic $\chi$,  and where $\rho$ is non-trivial on each connected component.  This carries a virtual fundamental cycle $[\ov\M_{d,\chi}^\circ (C)]^{vir}$ of {\em even} complex dimension $b= d(2-2g)-\chi$.  The operators $D^N$ and $D'$ of Definition~\ref{3.D.elem} induce  families of real operators ${\cal D}^N$ and ${\cal D}'$ over $\oM^\circ_{d, \chi}(C)$ whose fibers at $\rho$ are the pullback operators $\rho^*D^N$ and $\rho^*D'$.   By Definition~\ref{3.D.elem}a, the corresponding  index bundles satisfy $\mathrm{Ind}\; {\cal D}^N = \mathrm{Ind}\; {\cal D}'\oplus \mathrm{Ind}\; {\cal D}'$.   {\em A priori}, these are real virtual bundles, but Definition~\ref{3.D.elem}c insures that  the Taubes obstruction bundle  $Ob=-\mathrm{Ind}_\R\; {\cal D}^N$  is an actual vector bundle of rank $b$,  equal to the direct sum of two copies of  $Ob'=-\mathrm{Ind}_\R\; {\cal D}'$. The bundles  $Ob$ and $Ob'$ each come with a canonical orientation determined, on each connected component of the space of covers, by the spectral flow  to an injective complex operator over one fixed cover $\rho$. Computing this spectral  flow   along  a path of operators that respect the direct sum decomposition, one sees that $Ob=Ob'\oplus Ob'$ as {\em oriented} real bundles. 

With this notation, the elementary contribution is equal to the integral  of the Euler class 
\best
Z_{d, \chi}(\O)= \int_{[\ov\M_{d,\chi}^\circ (C)]^{vir}} e(Ob), 
\eest
\vspace{-2mm}   
where 
\best
e(Ob)&=&e(Ob'\oplus Ob')= e(Ob')\cup e(Ob')=   (-1)^{b/2} c_{b}(Ob'\otimes_\R\cx).
\eest
This Chern class  factors through $K$-theory (in general Euler classes do not).  Because 
$D'=\del_L+R$ is a 0'th order deformation of the complex operator $\del_L$,   the complexification of the index bundles of $D'$ and $\del_L$ are equal in $K$-theory, so
\best
c_{b}(Ob'\otimes_\R\cx)\ =\ c_{b}(-\mathrm{Ind}\; \del_L\otimes_\R \cx)\  =\ 
 c_b(- \mathrm{Ind\;} \del_{L}\oplus (-\mathrm{Ind\;} \del_{L})^*).
\eest
Combining the last three displayed equations gives
\bear\label{eq.Z.red}
Z_{d, \chi}(\O)&=&  (-1)^{b/2}  \int_{ [\ov\M_{d,\chi}^\circ (C)]^{vir}} c_b(-\mathrm{Ind\;} 
\del_{L}\oplus  (-\mathrm{Ind\;} \del_{L})^*).
\eear

\medskip

 The right-hand side of \eqref{eq.Z.red}  can be evaluated using equivariant  techniques. The   torus
 $T=\cx^*\times \cx^*$ acts on the total space $Y$ of the  holomorphic bundle $N_C=L\oplus L$.  With the antidiagonal  $\cx^*$-action, $Y$ is an equivariant local Calabi-Yau 3-fold. Bryan and Pandharipande defined a `residue'  generating function $Z^{T}(Y)$  whose  coefficients 
$$
Z_{d, \chi}^{T}(Y)\ =\  \int_{ [\ov\M_{d,\chi}^\circ(C)]^{vir}} c_b(-\mathrm{Ind\;} 
\del_{L\oplus L})
$$
are {\em equivariant} integrals (defined by localization). They proceeded to  express them in terms of  {\em ordinary} integrals:
\bear\label{Z.red.equiv}
Z_{d, \chi}^{T}(Y)\ =\ \sum_{b_1+b_2=b} \int_{ [\ov\M_{d,\chi}^\circ(C)]^{vir}}(t_1/t_2)^{(b_2-b_1)/2}c_{b_1}(-\mathrm{Ind\;} \del_{L})c_{b_2}(-\mathrm{Ind\;} \del_{L})
\eear
where  $t_1, t_2$ are the weights of the action (cf. page 105 of \cite{b-p-GWcurves}).  For the antidiagonal action $t_1=-t_2$,  \eqref{Z.red.equiv} reduces to   \eqref {eq.Z.red} after noting that  $c_k(E^*)=(-1)^k c_k(E)$ for $E=-\mathrm{Ind\;} \del_{L}$.  On the other hand,  for the antidiagonal action,  Bryan and Pandharipande also explicitly calculated \eqref{Z.red.equiv}  to be the coefficient of the series appearing on the right hand side of \eqref{BPelemContribution} (Corollary~7.3 of  \cite{b-p-GWcurves}). This completes the proof.
\end{proof}

The series \eqref{BPelemContribution} is a universal power series that depends only on the genus $g$ of $C$.  Thus we set
\bear
\label{3.defGelem}
GW^{elem}_g(q,t)\ =\ \log Z^{elem}_g(q,t),
\eear
where 
\bear
\label{DefZq}
Z^{elem}_g(q,t)\  =\  1+\sum_{d\ge 1}  \sum_{\mu\vdash d}\; \prod_{\square\in \mu}  
\l(2\sin \frac{h(\square)t} 2\r)^{2g-2} q^{d}.
\eear
 In fact, taking log of \eqref{DefZq} and separating the $d=1$ term of the series, 
\bear
\label{logZ}
GW^{elem}_g(q,t)\ =\ q \l( 2\sin \frac{t}2\r)^{2g-2} +\ \sum_{d\ge 2}\sum_{h\ge g}
GW_{d, h}(g)\ q^{d} t^{2h-2}, 
\eear
for some coefficients $GW_{d, h}(g)\in \Q$. 
Since the  coefficient of the leading term  $q t^{2g-2}$ is $+1$,  the core curve $C$ of any elementary cluster has  $\sign C>0$.  

Now apply the ``BPS transform'', which takes an arbitrary element of the Novikov ring to another by
$$
\sum_{A, g} N_{A,g} \, t^{2g-2} q^A\ =\ \sum_{A, g}\ n_{A, g} \ma\sum_{k=1}^\infty \frac 1k\l(2\sin \frac{kt}2\r)^{2g-2}  q^{kA}.
$$
This transform is well-defined and invertible (Proposition 2.1 of \cite{b-p-BPS}).  Thus for   an elementary cluster $\O$ whose core $C$ has genus $g$
we can write
\bear
\label{3.3}
GW(\O)\ =\ GW^{elem}_g(q^C, t) \ =\ \ma\sum_{d\ne 0 }\sum_h \ n_{d, h}(g) 
\ma\sum_{k=1}^\infty \frac 1k\l(2\sin \frac{kt}2\r)^{2h-2}  q^{kdC}
\eear
for uniquely determined  coefficients  $n_{d, h} (g)$ that are, {\em a priori}, rational numbers. These coefficients  have been explicitly calculated for low degree ($d\le 2$) and for low genus ($g\le 1$). A combinatorics argument now handles the case $g\ge 2$, yielding a basic fact:
\begin{prop} 
\label{theorem3.4}
The local Gopakumar-Vafa conjecture is true for elementary clusters.    More specifically,  the coefficients of the series (\ref{3.3}) associated with a 
 genus $g$ elementary cluster $\O$ satisfy 
\begin{enumerate} [(a)]
\item (Integrality) $n_{d, h} (g)\in \Z$. 
\item (Finiteness) for each $d$ fixed, $n_{d, h}(g) =0$  for $h<g$ or $h$ large. 
\item For $g=0$, all $n_{d,h}(g)$ vanish except $n_{1, 0}(g)=1$.
\item For $g=1$, all $n_{d,h}(g)$ vanish except $ n_{d, 1} (g)=1$ for each $d\ge 1$.
\end{enumerate} 
\end{prop} 

\begin{proof} 
When the core curve has  genus zero and normal bundle  $O(-1)\oplus O(-1)$, $\O$  is an elementary cluster and its contribution to the $GW$ invariant was first  calculated by  Faber and Pandharipande.   Specifically, letting  
$c(h,d)$ denote the coefficient of $q^{d} t^{2h-2}$ in (\ref{logZ}) with $g=0$,   equations (34), (35) and the middle displayed equation on page 192 of  \cite{f-p} imply the formulas
$$
c(h, d) = d^{2h-3} \, c(h,1)
\hspace{2cm}
\sum  c(h,1)\, t^{2h-2} \ =\   \left(2 \sin(\tfrac{t}{2})\right)^{-2}
$$ 
(these are equations (1) and (2) in \cite{p}). Consequently,  the genus 0 elementary GW series is
\best
GW^{elem}_0(q, t) \ =\ \sum_{h,k}c(h,k)\ t^{2h-2}\, q^k\ =\ \sum_k \frac 1k\l(2\sin \frac{kt}2\r)^{-2}  q^{k}.
\eest
Comparing with (\ref{3.3})  gives (c).

 For genus $g=1$, (\ref{DefZq}) reduces to the generating function for the number $p(d)$  of  unordered  partitions  of $d$:
$$
Z_1^{elem}(q, t) \ =\ 1+\sum_{d\ge 1} \ p(d)\ q^d\ =\ 
\prod_{d=1}^\infty \left(\frac{1}{1-q^d}\right). 
$$
Hence
$$
GW_1^{elem}(q, t) \ =\  \log Z^{elem}_1(q, t) \ =\ -\sum_{d=1}^\infty \ \log(1-q^d)\ =\ \sum_{d=1}^\infty  \sum_{k=1}^\infty \ \frac{\, q^{kd}}{k}.
$$
Comparing with (\ref{3.3})  gives (d). 

In the higher genus case, both (a) and (b) are consequences of an algebraic fact about power series with integral coefficients  that follows by combining several results in the paper \cite{p-t}  of Pandharipande and Thomas. Making the change of variable $Q= e^{it}$,  
 \eqref{3.3} becomes
 \best
 \log Z_g^{elem}  \ =\   \sum_{d\ge 1}\sum_{ h} n_{d, h}(g)\ \sum_{k\ge 1} \frac { (-1)^{h-1}}k 
 (Q^{k} + Q^{-k}-2)^{h-1} q^{kd}. \nonumber  
 \eest
On the other hand, for $g\ge 1$,  \eqref{DefZq}  becomes 
 \bear
\label{ZlocAnd}
Z_g^{elem}\ =\  1+\sum_{d\ge 1}  \sum_{\mu\vdash d}  \prod_{\square\in \mu}(-1)^{g-1} \l(Q^{h(\square)}+Q^{-h(\square)}-2\r)^{g-1} q^d 
\ =\  \sum _{d=0}^\infty \sum_n A_{n, d}\  Q^n q^d
\eear
where, for each $d$, the inner sum is a Laurent polynomial in $Q$ with integer coefficients $A_{n, d}$. 
These coefficients $A_{n, d}$  uniquely determine the numbers $n_{d, h}$.  But by Theorem~3.20 of 
 \cite{p-t},  the integrality of the $A_{n, d}$ implies that all of the $n_{d, h}$ are also integers. Thus statement (a) holds.

For $g\ge 2$, the coefficient of $q^d$ in  \eqref{DefZq} is a Taylor series in $t^2$,  $Z=1+ t^{2g-2}q+ O(t^{2g})$, so $\log Z= t^{2g-2} q+ O(t^{2g})$.  Comparing with  (\ref{3.3}) one sees that $n_{d,h}(g)=0$ for all $h<g$, as  in (b).  

Finally, for genus $g\ge 2$, the inner sum  in \eqref{ZlocAnd}  is  a Laurent polynomial in $Q$, symmetric in $Q\ra Q^{-1}$, and with degree bounded by $(g-1)\sum h(\square)\le d^2(g-1)$ (since for a partition of $d$, the hooklength $h(\square)$ of each box is at most $d$). This property is preserved under taking the log: 
$$
\log  Z_g^{elem}  \ =\  \log \sum_{d}\sum _n a_{d, n} Q^{n} q^d  \ = \   \sum_{d\ge 1} \sum_ {h\ge g } n_{d, h}(g)\ \sum_{k\ge 1} \frac { (-1)^{h-1}}k 
 (Q^{k} + Q^{-k}-2)^{h-1} q^{kd},
$$
where $|n|\le (g-1) d^2$. As in the proof of Lemma 3.12 of \cite{p-t}, this implies the vanishing of $n_{d, h}(g)$ for large $h$. In fact, a proof by induction on $d$ using the above bound implies that $n_{d, h}(g)= 0$ for $h-1> d^2(g-1)$.
\end{proof}

\vspace{5mm}

\setcounter{equation}{0}
\section{Analytic preliminaries}
\label{section4}
\bigskip

This section is a review of the analytic setup for the  moduli space for a general closed symplectic manifold $(X,\w)$.
Consider the universal moduli space  of simple maps  (with smooth, connected domains and smooth $J$) 
\bear\label{4.M.simple} 
\begin{CD}
\M_{simple}\\
@VV{\pi}V\\
\hspace{8mm} \J = \J^\infty
\end{CD}
\eear
with the projection $\pi([f],J)=J$. Note that $\M_{simple}$ is an open subset of the universal moduli space $\oM(X)$ of  \eqref{0.1}, and simple maps have trivial automorphism  groups (cf.   \cite{ms}, Proposition 2.5.1).  
It contains  the open subset  $\M_{emb}$  of maps that are embeddings.  

To set up the analysis, we first work locally around a  pair $p=(f, J)$ that represents a point in  the moduli space \eqref{4.M.simple}.   Thus $p$ consists of a simple $J$-holomorphic map $f:C\to X$ whose domain is a smooth, connected marked complex curve $C=(\Sigma, x_1, \dots, x_n, j)$,   $J$ is a smooth almost complex structure,  and $j$ is in the space $\J(\Sigma)$ of complex structures on $\Sigma$.

 \subsection{Slice and linearization.}  
 
 The moduli space $\M_{g,n}(X)$ is  naturally a subset of the quotient of  $Map(\Si, X)\times \J(\Si)\times\J$ by the action of the diffeomorphism group $\mathrm{Diff}(\Si, {\bf x})$ of $\Si$ that preserve each point in the set  ${\bf x}=\{ x_1, \dots, x_n\}$ of  marked points. In practice, one   chooses a local  slice for the diffeomorphism action  and regards the moduli space locally as a subset of the slice.  For now, we assume that  the domain $C_0$ has no automorphisms; this assumption will be removed at the end of this subsection.    To define a slice, choose local holomorphic coordinates on a ball $B\subset \M_{g,n}$ centered at $[C_0]\in \M_{g,n}$.  Then there is a  local universal deformation  $\gamma:{\mathcal U}_B\to B$  of $C_0$ with sections $x_1,\dots, x_n$.  This means, in particular, that the  central  fiber $\gamma^{-1}(0)$ is identified with $C_0$ as a marked complex curve, and every small deformation $C$ of $C_0$ is equivalent under $\mathrm{Diff}(\Si, {\bf x})$ to one and only one fiber  $C_b$, $b\in B$,  of $\gamma$.  Fix a smooth trivialization  $\tau$ of 
 ${\mathcal U}_B\ra B$ in which the universal deformation is $B\times  (\Si, {\bf x}) \to B$, 
\bear\label{4.tau}
\xymatrix@C=.5pc @R=.5pc{
{\mathcal U}_B \ar[rr]^{\tau\quad}\ar[dr]_\gamma&& B\ti (\Si, {\bf x})\ar[dl]^{pr_1}
\\
&B }
\eear 
 This trivialization, regarded as a family of complex structures $j_b$ on $(\Sigma, {\bf x})$, defines an embedding 
 \bear\label{B.to.slice}
 \si: B\to \J(\Sigma), \quad \mbox{given by } b\mapsto j_b,  
 \eear
 whose image  ${\mathcal S}_{\tau}$ is a local slice for the action of the diffeomorphism group on $\J(\Sigma)$.    The  linearization of this embedding 
  at $C=C_b$ gives isomorphisms
\bear\label{nat.iso}
 T_{C}\M_{g,n} \cong T_b B \ \overset{\cong} \longrightarrow \ T_{j_b}{\mathcal S}_\tau.
\eear 
Furthermore, the tangent space to the orbit ${\mathcal O}_j$  of  $\mathrm{Diff}(\Si, {\bf x})$ on $\J(\Sigma)$  is the image of 
$$
\del_{TC}:\Omega_{\bf x}^0(TC)\to \Omega^{0,1}(TC), 
$$
 where $\Omega_{\bf x}^0(TC)$ denotes the space of smooth sections of $TC$ that vanish at the marked points. 
 
  The slice ${\mathcal S}_{\tau}$ is transverse to this orbit at $j$, giving  an  isomorphism
\bear
\label{TCisomH01}
  T_j{\mathcal S}_{\tau} \ \cong \  T_j\J(\Sigma)/{\mathcal O}_j \ =\ \Omega^{0,1}(TC) /{\mbox{im $\del_{TC}$}}\ =\  H^{0,1}(TC),
\eear
where the last equality defines the vector space $H^{0,1}(TC)$.   Consequently, the map
\bear
\label{1.DC}
D_C: \Omega_{\bf x}^0(TC)\oplus   T_j{\mathcal S}_\tau \ \to \ \Omega^{0,1}(TC)
\eear
defined by $D_C(\zeta, k) =\del_{TC}\zeta +j k$ is a complex-linear isomorphism.

\medskip

Given local trivializations $\tau_1$ and $\tau_2$ as in \eqref{4.tau}  over two overlapping charts $B_1, B_2$ in $\M_{g,n}$ containing $[C_0]$, after restricting them to the overlap $B_{12}=B_1\cap B_2$, they determine a smooth transition function 
\bear\label{4.transition}
\phi=\tau_2\circ \tau_1^{-1}: B_{12} \ti \Si \ra B_{12}\ti \Si.
\eear 
The restriction $\phi_b$ to the fiber over $b\in B_{12}$ is a diffeomorphism of $\Si$ preserving the marked points ${\bf x}$. The corresponding maps \eqref{B.to.slice} restrict to embeddings $\si_1, \si_2: B_{12}\to \J(\Sigma)$,  and  
\bear
\label{4.1BS2}
\si_{12}=\si_2\circ \si_1^{-1}:  {\mathcal S}_1\to {\mathcal S}_2
\eear
 is  a diffeomorphism between  ${\mathcal S}_1= \si_1(B_{12})$ and  ${\mathcal S}_2=\si_2(B_{12})$ with 
 \best
 \si_{12}(j_b)= (\phi_b)^*(j_b) \quad \mbox{ for all }b\in B_{12}. 
 \eest
\smallskip

To include maps,  fix  $(f_0, J_0)$,   where $f_0$ is  a   $J_0$-holomorphic map  whose domain $C_0$ has $\Aut (C_0)=1$.   Then 
\bear\label{D.slice.S0}
{\mathcal Slice}_\tau\ = \  \l[Map(\Sigma, X)\times   {\mathcal S}_{\tau} \r] \times\J
\eear
is  a local slice for the action of  $\mathrm{Diff}(\Si, {\bf x} )$ on $Map(\Sigma, X)\times \J(\Si) \times\J$.    Elements of \eqref{D.slice.S0} have the form   $(f, j, J)$ where $f:\Si\ra X$ and $j\in {\mathcal S}_{\tau}$, making  $C=(\Si, j, x_1, \dots, x_n)$   a marked curve with complex structure $j$.  For notational simplicity, we will frequently combine the first two factors, writing elements  of \eqref{D.slice.S0} as pairs $(f, J)$, where  
the letter $f$ denotes a map $f:C\ra X$ and therefore implicitly includes its domain $C$, regarded as a marked complex curve. 

The slice \eqref{D.slice.S0} comes with a projection
$$
\pi: {\mathcal Slice}_\tau \to \J
$$
defined by $\pi(f, J)=J$,  and a complex vector bundle $\F\to {\mathcal Slice}_\tau$ whose fiber over $p=(f, J)$ is  $\Omega^{0,1}(f^*TX)$. Near $(f_0, J_0)$,  the moduli space, considered as a subset of ${\mathcal Slice}_\tau$,  is the zero set of the section  $\Phi$ of $\F$ defined by 
\bear
\label{4.defF}
\Phi(f, J)=\del_{J}f.
\eear

Under the isomorphism \eqref{nat.iso}, the tangent  bundle to the slice can be written as
\bear
\label{4.ETJ}
T{\mathcal Slice}_\tau\ =\ \E \oplus T\J,
\eear
where $\E\ra {\mathcal Slice}_\tau$ is the complex bundle whose fiber at $p=(f, J)$ is   $\E_p=  \Omega^{0}(f^*TX)\oplus  T_j{\mathcal S}_\tau$.

 The linearization of the $J$-holomorphic map equation \eqref{4.defF}  at a solution  $p$ on the slice is the  real operator  $\L_p:\E_p  \oplus T_J\J\to \F_p$  given by 
\bear
\label{4.formulabigL}
{\cal L}_p(\xi,  K) =   D_p\xi + \tfrac12 K\circ df \circ j,
\eear
where   $D_p:\E_p\to\F_p$  is the  linearization under variations that fix $J$.  Explicitly,   $D_p$  applied to   
$\xi=(\zeta, k)\in  \Omega^{0}(f^*TX)\oplus  T_j{\mathcal S}_\tau$  is    
\bear
\label{1.formulaL}
 D_p(\xi)(w) \ =\ \tfrac12\big[ \nabla_w \zeta+J\nabla_{jw}\zeta+
(\nabla_\zeta J)(df(jw))+ J df (k(w)) \big], 
\eear
where $\nabla$ is any  torsion-free connection on $TX$;  at a solution $p$,  $D_p$  is independent of the connection   (Lemma~2.1.2  \cite{IS1} or  Lemma~1.2.1 of \cite{IS2}).   Both $\L_p$ and  $D_p$ depend on the  $J$ only through the 1-jet  of $J$ along the image of $f$,  and the spaces $\E_p$ and $\F_p$ depend only on the 0-jet  of $J$ along the image of $f$. Furthermore, when  the variation $\zeta=f_*\zeta^T$ is tangent to  $C$,  \eqref{1.formulaL} reduces to   
 \bear\label{L.tg.restr}
D_p(f_*\zeta^T, k) = f_*D_C(\zeta^T, k),
 \eear 
 where $D_C$ is the  isomorphism \eqref{1.DC}.  
Consequently, when $f:C \ra X$ is a $J$-holomorphic immersion with normal bundle $N_C=f^*TX/TC$, the linearization \eqref{1.formulaL} uniquely descends to a normal operator 
\bear\label{4.defn.normal.op}
D_p^N: \Gamma(N_C) \ra \Omega^{0,1}(N_C). 
\eear

 Now consider a {\em simple} map  $f_0: C_0 \ra X$ whose automorphism group
 $\Aut(C_0)$ is non-trivial. As noted after \eqref{1.Msimple}, the injective points of $f_0$ are dense on each component.  Hence we can choose 
$\ell$ additional injective points  on $C_0$ so that  the marked curve  $\wt C_0= (C_0, x_1, \dots, x_{n+\ell}$) has $\Aut(\wt C_0)=1$.  We  can then fix a  trivialized  local universal deformation of the  $(n+\ell)$-marked curve $\wt C_0$ over  a ball  $\wt B \subset \M_{g, n+\ell}$. 
 Set  $\wt \Sigma = (\Sigma, x_1, \dots, x_{n+\ell})$, and for each image point $y_i=f_0(x_i)$, $n<i\le n+\ell$,  choose a   codimension~2 ball $V_i$ through $y_i$  transverse to $f_0(C_0)$. Standard results (cf. Section 3.4 of \cite{ms}) show that the space $Map_\ell(\wt \Si, X)$ of maps satisfying $f(x_i)\in V_i$ for $n<i\le n+\ell$ is locally a manifold near  $f_0$, and hence
$$
{\mathcal Slice}_\tau \ = \ Map_\ell(\wt \Si, X)\times \wt B\times\J
$$
is a  local slice for the action of $\mathrm{Diff}(\wt \Si)$ on $Map_\ell(\wt \Si, X)\ti\J(\wt\Si)\ti \J$.  Thus defined, each point in the slice is a pair $p=(f, J)$ where $f:\wt C\to X$ is a  map whose domain has no non-trivial automorphisms. The linearization 
\bear
\label{4.DEF}
D_p:\E_p\to\F_p
\eear
 is still given by \eqref{1.formulaL}, but where  $\E$ is now the bundle over the slice whose  fiber over $p=(f, J)$ is  
\bear
\label{4.EpCtilde}
\E_p\ =\ \left\{ \zeta \in\Omega^{0}(f^*TX)\, \Big|\, \zeta(x_i)\in TV_i \mbox{ for all } n+1\le i\le n+\ell \right\}\oplus T_{\wt C}\M_{g,n+\ell}. 
\eear

For notational simplicity, we will henceforth write $\wt C$ as $C$, and always restrict to sections $\zeta$ with  $\zeta(x_i)\in TV_i$.  All the variations we construct in this  and subsequent sections will be supported away from all  marked points of $\wt C$.

\subsection{Sobolev completions.} 
\label{subsection4.2} 

 Throughout this paper, we work with the following set of Banach space completions (cf. \cite[Section 3.1]{ms}).   For numbers
\bear\label{4.2.rlcondition}
l \ge 6,\quad  r > 2, \qquad 1\le m \le l,
 \eear
 let $\J^l$ denote the space of tame $C^l$ almost complex structures on $X$,  let 
${Map}^{m,r}(\Si, X)$ be the completion of the space of smooth maps $\Sigma\to X$ in the Sobolev $(m, r)$ norm (i.e. the $m$-jet is in $L^r$), and let 
\bear\label{D.slice.S}
{\mathcal Slice}_\tau^{m,r, l}=\l[\vphantom{K} Map^{m, r}(\Sigma, X)\times   {\mathcal S}_{\tau} \r] \times\J^l.
\eear
These are smooth separable Banach manifolds.

Similarly, for each $m$ in the range \eqref{4.2.rlcondition},  the vector bundles $\E$ and $\F$ extend to  vector bundles  
$\E^{m, r}$ and $\F^{m-1, r}$ over the slice  \eqref{D.slice.S}, whose fibers at $p=(f, J)$ are, respectively, 
\bear\label{def.ef.comp}
\E^{m, r}_p= W^{m, r}(f^*TX)\oplus T_j\mathcal S_{\tau}  \quad \mbox{and}\qquad 
\F_p^{m-1, r}=W^{m-1, r}(\Lambda_C^{0,1}\otimes_\cx f^*TX), 
\eear
where  $W^{m,r}(E)$ denotes the space of  Sobolev $(m,r)$ sections of a vector bundle $E$, and where $\Lambda^{0,1}_C$ is the bundle $(T_\cx^*C)^{0,1}$  over the domain $C$.  
The bundle $\E^{m,r}$ is smooth (it is the  tangent bundle of ${Map}^{m,r}(\Si, X)\times S_\tau$)  and $\F^{m-1,r}$ is of class $C^{l-m}$ (cf. \cite{ms}, page 50).

 In this context,  \eqref{4.defF} defines a $C^{l-m}$ section of $\F^{m-1, r}$ over the slice \eqref{D.slice.S} whose zero locus is a local model of the moduli space. We will focus on the subset 
\bear\label{4.2Msimpletau}
M_{simple} \subset {\mathcal Slice}_\tau^{m,r, l}
\eear
 of  pairs $p=(f, J)$ where $f$ is a {\em simple} $J$-holomorphic map.  By  elliptic regularity, all such maps $f$ are of class $W^{l+1,r}$ \cite[Proposition 3.1.10]{ms}, and hence  the set $M_{simple}$ 
 is independent of  $m$ in the range \eqref{4.2.rlcondition}, and its elements are pairs $(f, J)$ where both $f$ and $J$ are of class $C^l$. 

\medskip

For two local trivializations $\tau_1, \tau_2$ as in \eqref{4.tau} over the same $B=B_{12}$, there is a transition map $\phi$ as in \eqref{4.transition}. This induces a map $\wh \phi$ between two slices \eqref{D.slice.S0}  given by 
\bear\label{transition.moduli}
\wh \phi(f,j_b,J)\ =\ (f\circ\phi_b,\, (\phi_b)^*(j_b),\, J)\ =\ (f\circ\phi_b,\, \si_{12}(j_b),\, J),
\eear
for the map $\phi_b$ defined after \eqref{4.transition} and $\si_{12}$ as in \eqref{4.1BS2}.   In this formula,  $\sigma_{12}$ is smooth, and $\{\phi_b\,|\, b\in B\}$  is   a smooth family of diffeomorphisms of the closed surface $\Sigma$.  As a result, the regularity of $\widehat{\phi}$ is determined by the regularity of the map $T$ defined by $T(f, \phi)=f\circ \phi$. Formally computing its differential, one finds that 
\best
dT_{f, \phi} (\xi, v)= \phi^* \xi + df(v), \qquad \mbox{ for all } \xi \in \Gamma(f^* TX) \mbox{ and } v\in \Gamma(TC).  
\eest
More generally, one finds that  the $k$'th derivative of $T$ depends on the  $k$-jet of $f$.  It follows that 
  \eqref{transition.moduli} induces a $k$-times differentiable, hence $C^{k-1}$, map 
\bear\label{4.slice-to-slice}
\wh \phi:{\mathcal S}lice_{\tau_1}^{m, r,l} \ra {\mathcal S}lice_{\tau_2}^{m-k, r,l}.
\eear
Thus a change of trivializations induces a map of slices which   loses regularity.  
We will return to this technical issue in Proposition~\ref{5.1new}.

\subsection{Extensions and adjoints.} 
\label{subsection4.3} 

The linearizations \eqref{1.formulaL}   extend (non-canonically) to a family of operators parameterized by   points $p=(f,J)$  in the slice  \eqref{D.slice.S} as follows. Fix a  Riemannian metric $g_0$ on $X$ and let $\nabla^0$ denote its Levi-Civita connection. Using the notation of   \eqref{1.formulaL}, define $D_p$  by  
\bear\label{D.formula.lin}
D_{p}(\zeta, k) = D_p^0\zeta + 
\tfrac 14 (J df+ df j) k.
\eear
where 
\bear\label{D0.formula.lin}
(D^0_{p}\zeta)(w) =\tfrac 12 \l( \nabla^0_w\zeta + J\nabla^0_{jw}\zeta \r)  +  
\tfrac 14(\nabla^0_\zeta J) (df j+J df)(w).
\eear
Then    $D^0_p$ agrees with \cite[(3.1.4)]{ms}, and \eqref{D.formula.lin} agrees with \eqref{1.formulaL}   if $f$ is $J$-holomorphic  because \eqref{1.formulaL}  is independent of the connection and $J df=df j$. 
Similarly extend    \eqref{4.formulabigL} by the formula
$$
\L_p(\zeta, k) = D_p\xi+\tfrac14  K (df j+J df).
$$
 As in Section~3.1 of \cite{ms},  $D_p$ and $\L_p$ extend to bounded linear operators 
\bear\label{def.dl.comp}
D_p:\E_p^{m, r}\ra \F_p^{m-1, r}  \quad \mbox{ and } \quad  {\mathcal L}_p: \E_p^{m, r}\oplus T_J\J^l \ra  \F_p^{m-1, r},
\eear
and $D_p$ is a compact perturbation of   $D_p^0$,  and hence is Fredholm.  Moreover, if $p=(f, J)$ is a $J$-holomorphic pair,  then  $\ker D_p$ and $\cok D_p$ are independent of $m$ in the range \eqref{4.2.rlcondition}.
\medskip

Next fix a Riemannian metric $g_S$ compatible with the complex structure on the local universal family of curves parameterized by $\mathcal S_\tau$.  By restriction,  $g_S$  induces a Riemannian metric on each  curve in the local family which,  under the trivialization associated with the slice,  gives rise to a family of metrics on $\Sigma$ parameterized by the $\mathcal S_\tau$. 
These metrics, together with their associated   volume forms and the  fixed metric $g_0$  on $X$
determine $L^2$ inner products $\langle\ ,\ \rangle_{L^2}$ on   $T_j \mathcal S_{\tau}$,  $\E_p$ and $\F_p$ for each $p=(f, J)\in {\mathcal Slice}_\tau$.

 Let  $D^*_p$ denote the formal $L^2$ adjoint of the operator $D_p$ of 
\eqref{D.formula.lin}, which is uniquely defined by 
\bear
\label{4.2.D*def}
\langle D_p\xi, \eta \rangle_{L^2}\ =\  \langle \xi, D^*_p \eta\rangle_{L^2} 
\eear
for all $\xi\in\Gamma(f^*TX)\oplus  T_j\mathcal S_\tau$ and  $\eta\in\Omega^{0,1}(f^*TX)$. 
The adjoint operator depends on the choice of metrics.

\medskip

  Assume $p=(f,J)$ is in a slice \eqref{D.slice.S},  where $f:C\ra X$ is simple and $J$-holomorphic , and $C$ is a smooth connected complex curve. 
 For an element $\xi=(\zeta, k)$ of $\E_p^{m, r}$  and an injective point $x$ of $f$,  we define $\xi^N(x)$ to be the  component $\zeta^N(x)$ of $\zeta(x)$ normal to $f_*(T_xC)$ with respect to  the metric $g_0$ on $X$. 

\smallskip

 We will repeatedly use the following  simple consequence of  elliptic theory.
\begin{lemma}
\label{lemma4.1kc}
 Fix $p=(f, J)$ in the set $M_{simple}$ of  \eqref{4.2Msimpletau}.  Suppose that    $\kappa \in \E^{0,s}_p$ and $c\in \F^{0,s}_p$, $\tfrac{1}{s}+\tfrac{1}{r}=1$,  are nonzero weak solutions of   $D_p\kappa=0$ and $D^*_pc=0$.
Then $\kappa\in\E^{l,r}_p$ and $c\in\F^{l,r}_p$, and   there is an injective point $x\in C$ such that  $c(x)\not= 0$ and  $\kappa^N(x) \ne 0$.
\end{lemma}

\begin{proof}
The equation $D^*_pc=0$ means that the $L^2$ inner product 
 $\langle D_p(\zeta, k), c\rangle_{L^2}$ is zero for all $(\zeta, k)$ and therefore, taking $k=0$,   $(D_p^0)^*c=0$.   Lemma~3.4.4 of \cite{ms} then shows that $c$ is in the   Sobolev $(l,r)$ space, hence is continuous, and also shows that  $c$ cannot vanish identically on  any open set in $C$.

   Similarly, $\kappa=(\zeta, k)$ is a weak solution of $D_p^0\zeta= -\tfrac12 Jdfk$ with $k$ smooth  and $f, J$ of class $C^l$, and hence $Jdf k\in \F^{l-1,r}$. Elliptic regularity as in \cite[Proposition C.2.3]{ms}  implies that $\zeta$ is in the  Sobolev $(l,r)$ space, so   $\kappa$ is in $\E^{l,r}$ and is continuous.  If $\kappa$ were everywhere tangent to $C$,  it  would satisfy $D_C\kappa=0$ for the operator  in \eqref{1.DC}.  But then $\kappa$ would be smooth, and would contradict the fact that   \eqref{1.DC} is an isomorphism.  Thus $\kappa^N\not= 0$ on some non-empty  open set. 
  
  The lemma follows because  $f$ is  at least $C^2$ so, by  Micallef-White Theorem \cite{mw},  the injective points are open and dense in $C$.  
 \end{proof}

\vspace{5mm}

\setcounter{equation}{0}
\section{The structure of the moduli space}
\label{section5}
\bigskip

 We now consider the completion of the universal moduli space \eqref{4.M.simple} in the Sobolev norms introduced in Section~\ref{subsection4.2}.  For simplicity, we will specify the Sobolev norm only when 
needed.   Thus we fix $(l ,r)$ as in \eqref{4.2.rlcondition} and, without changing notation, let  
\bear\label{5.M.simple} 
\begin{CD}
\M_{simple}\\
@VV{\pi}V\\
\hspace{8mm} \J=\J^l
\end{CD}
\eear
be the  universal moduli space  of equivalence classes $[p]$ (up to reparametrizations of the domain)  of pairs $p=(f,J)$, where $J \in \J^l$ and  $f:C\ra X$ is a simple $J$-holomorphic map  of class  $W^{l, r}$ whose domain $C$ is a smooth, connected complex curve. This section and the next provide  a series of facts about the structure of the moduli space \eqref{5.M.simple}. These results are proven  locally by regarding the moduli space as a subset of a slice  \eqref{D.slice.S}.   Proposition~\ref{5.1new} and Lemma~\ref{lemma5.1}  hold for any closed symplectic manifold $X$;  after that we specialize to Calabi-Yau 6-manifolds, for which the index of $\pi$, given by \eqref{1.dimformula}, is zero.


\subsection{The structure of $\M_{simple}$.} 
 It is well-known that the moduli space \eqref{5.M.simple} of simple  maps is a manifold.  We give a precise statement and proof for later use.

 \begin{prop} 
\label{5.1new} 
The universal moduli space in \eqref{5.M.simple} has the following structure:\\[-1mm]

\hspace{.6cm} \begin{minipage}{5.8in} 
\begin{enumerate}[(a)] \itemsep2mm
\item  The set $M_{simple} \subset {\mathcal Slice}_\tau^{m,r, l}$  in \eqref{4.2Msimpletau}  is a  $C^{l-m}$  separable Banach submanifold whose  tangent space  at $p$ is  the kernel of the operator $\L_p$ in \eqref{def.dl.comp}.  
\item     For  $2k \le l-2$,  $\M_{simple}$ is a $C^k$ separable Banach manifold, locally $C^k$ diffeomorphic to the subset $M_{simple}$ of the slice in (a) for each $m$ in the range $1\le m \le l-k$. \\[-2mm]
\end{enumerate}
\end{minipage}
  In particular, $\M_{simple}$ is at least $C^2$ using Sobolev norms in the range \eqref{4.2.rlcondition}.

\end{prop}

\begin{proof}  
(a)   As in \eqref{4.2Msimpletau}, the set  $M_{simple}$ is the zero set of the $C^{l-m}$ section $\Phi$ of $\F^{m-1, r}$   defined  by \eqref{4.defF}.  By the Implicit Function Theorem, 
$M_{simple}$ is a $C^{l-m}$ submanifold of the slice at those points $p$ where  $D\Phi_p$, which is  the operator
${\cal L}_p$  in \eqref{def.dl.comp},  is onto.  By Lemma~\ref{lemma4.1kc}, the surjectivity of ${\cal L}_p$  is independent of $m$ in the range \eqref{4.2.rlcondition}, so it suffices to consider the case  $m=1$.    Surjectivity  fails at  $p=(f,J)$ only if there is a non-zero $c$  in   the dual space $(\F_p^{0,r})^*=\F_p^{0,s}$, $s=\frac{r}{r-1}>1$,   that is $L^2$ orthogonal to ${\cal L}_p(\xi ,K)$ for all $(\xi, K)\in \E_p^{1,r}\oplus T_J\J^l$. By
\eqref{4.formulabigL} and \eqref{4.2.D*def},  this implies that $D_p^*c=0$  weakly, and 
\bear
\label{new5.1}
0=\int_C \lg c,\,  Kf_*j\rg
\eear
for every  variation $K$ in $J$.   By Lemma~\ref{lemma4.1kc}, $c\in\F^{l,r}$ and  there is   an  injective point $x\in C$  where $c(x)\ne 0$.  One can then find a variation  $K_0\in T_J\J^l$ in $J$  that satisfies  $ K_0 f_* j= c$ at the point $x$  (cf. \cite[Lemma~3.2.2]{ms}).  
 Choose local coordinates  $y=(y_1, y_2, \dots)$ on $X$ centered at $f(x)$.  Fix  a  non-negative bump function  $\beta(y)$ supported in this coordinate chart   and for each $\ep>0$ set  $\beta_\ep(y)= c(\ep)\beta (y/\ep)$,  where $c(\ep)$ is  the constant determined by the normalization condition $\int_C f^*\beta_\ep =1$. Then for each continuous function $\phi$ on $C$ we have
\bear\label{eq.beta.delta}
\lim_{\ep\to 0} \int_C f^*\beta_\ep\cdot \phi \ =\ \phi(x).
\eear 
Substituting $K=\beta_\ep K_0$ in \eqref{new5.1} and taking the limit as $\ep\ra 0$ gives a contradiction.   Thus 
\bear\label{5.2bsliceslice}
M_{simple}\subset Slice^{m,r,l}_\tau \
\eear
is a $C^{l-m}$ submanifold. This can be improved: for any $l\ge m \ge m'\ge 1$,  the inclusions 
\bear\label{5.2bsliceslice2}
M_{simple}\subset Slice^{m,r,l}_\tau \subset Slice^{m',r,l}_\tau 
\eear
show that the $C^{l-m}$ atlas obtained from \eqref{5.2bsliceslice} can be refined to an $C^{l-m'}$ atlas  inherited from the enlarged  slice \eqref{5.2bsliceslice2}.    It follows that for  each $k$,
  the inclusion \eqref{5.2bsliceslice} induces a $C^k$ atlas on $M_{simple}$, which is independent of $m$  in the range  $1\le m\le l-k$.

 \smallskip
 (b)  The moduli space is covered by images  of slices of the form \eqref{D.slice.S} under the maps $(f,J)\mapsto ([f],J)$, and any two slices with overlapping image are related by 
a  transition map  \eqref{4.slice-to-slice}.   Although  \eqref{4.slice-to-slice}  appears to lose regularity, its restriction to the moduli space does not.
Specifically,  for any $k \ge 0$ with $l-2k-1\ge 1$, the transition map 
$$
\wh \phi:{\mathcal S}lice_{\tau_1}^{l-k, r,l} \ra {\mathcal S}lice_{\tau_2}^{l-2k-1, r,l}
$$
is $C^{k}$ (cf. \eqref{4.slice-to-slice}), and maps the $C^k$ submanifold $M_{simple}\subset {\mathcal S}lice_{\tau_1}^{l-k, r,l}$ into the corresponding subset  $M_{simple}'$ of 
${\mathcal S} lice_{\tau_2}^{l-2k-1, r,l}$ (because reparameterizations of simple maps are simple). 
 The latter   inherits a $C^{2k+1}$ structure from the slice ${\mathcal S}lice_{\tau_2}^{l-2k-1, r,l}$,  and hence a $C^{k}$ structure.  Moreover, by the last sentence of part (a), this is the same $C^k$ structure that $M'_{simple}$  inherits from its embedding into the slice ${\mathcal S}lice_{\tau_2}^{l-k, r,l}$.   Thus $\wh\phi$ restricts to a $C^k$ bijection from $M_{simple}$ to $M'_{simple}$; reversing the roles of $\tau_1$ and $\tau_2$ shows that this is a $C^k$ diffeomorphism.  This gives a $C^k$ 
atlas on the moduli space. In particular, this applies with $k=[l/2]-1$, and hence  the moduli space has a $C^2$-atlas provided $l\ge 6$.

\end{proof}


\subsection{The   wall $\W$.}

The moduli space $\M_{simple}$ has a distinguished subset:  the  ``wall''    $\W\subset\M_{simple}$ defined as the stratified set 
\bear\label{Wall}
\W = \bigcup_{s\ge 1} \W^s,  \qquad \W^s = \Big\{ [p]  \in \M_{simple}   \, \Big|\, \dim \ker D_p  = s\, \Big\}.
\eear
 Lemma~\ref{lemma4.1kc} implies that $\ker D_p$ and $\ker D_p^*$, and hence $\ind D_p$, are independent of the choice of Sobolev norm in the range \eqref{4.2.rlcondition}.  Furthermore, the dimension of  $\ker D_p $ and of $\cok D_p$ are preserved under smooth reparametrizations  of the domain, 
so \eqref{Wall} is well-defined.

Observe that each $C^l$ embedded complex curve  $\iota_C:  C\hookrightarrow X$ determines a set  $\J_C\subset \J= \J^l$ consisting of all   $J\in \J$ for which $\iota_C$ is $J$-holomorphic.  The  restriction  of (\ref{5.M.simple})  over $\J_C$ has a canonical section $J\mapsto (\iota_C, J)$ whose image is $\M_C= \{\iota_C\}\times \J_C$.

\begin{lemma}
\label{lemma5.1}
For each $C^l$ embedded complex curve $C$,  $\J_C$ is a smooth submanifold of  $\J=\J^l$.
\end{lemma}

\begin{proof}
Let $j$ denote the complex structure on $C$.
Identify $C$ with its image in $X$, and let $E\to C$ be the vector bundle $E=\mbox{End}(TX|_C)$.  At each $x\in C$, the fiber $E_x$ of $E$ contains nested submanifolds $E'_x=\{J\in E_x|J^2=-Id.\}$ and $E_x'' = \{J\in E'_x | \, J|_{T_xC}=j\}$.  As $x$ varies, these define $C^l$  fiber bundles $E'$ and $E''$ over $C$.  Let $\J'$ and $\J''$ denote the spaces of $C^l$ sections of $E'$ and $E''$ respectively.   By standard theory,  $\J'$ is a smooth   Banach manifold and $\J''$ is a  submanifold of $\J'$.  Restricting an almost complex structure $J$ on $X$ to $C$ defines a smooth map
$$
\rho_C:\J\to \J'
$$
with $\J_C=\rho_C^{-1}(\J'')$.  The lemma follows if we prove that $\rho_C$  is a submersion.

At each $J\in\J_C$, the differential of $\rho_C$ is simply the restriction $(d\rho_C)_J(K)=K|_C$, and the tangent bundle to $\J'$ is the set of all $C^l$ sections $Y$ of $E$ that satisfy $JY+YJ=0$.  But every such $Y$ extends to a section $K$ of $T_J\J$: extend $Y$ to a tubular neighborhood of $C$ in $X$, multiply by a smooth cuttoff function to obtain an $\hat{Y}$, and take $K=\frac12(\hat{Y}+J\hat{Y}J)$.  Thus $\rho_C$ is a submersion.
\end{proof}

 Henceforth (until the end of Section~8) assume that $X$ is a symplectic Calabi-Yau 6-manifold.   The  results  below then  show that  $\W^1$ is a codimension~1 submanifold  of $\M_{simple}$ with a distinguished submanifold $\A\subset \W^1$,   that  the other strata $\W^r$ have higher codimension, and that the same is true for the subsets $\W^r\cap \M_C$ and $\A\cap\M_C$  of $\M_C$.

 \begin{prop} 
\label{MmfdThm}  
Let $X$ be a symplectic Calabi-Yau 6-manifold.   Then, as a subset of the universal moduli space in \eqref{5.M.simple}, the wall has the following structure:   \\[-1mm]

\hspace{.6cm} \begin{minipage}{5.8in} 
\begin{enumerate}[(a)] \itemsep2mm
\item $\W$ is the set of  critical points of the projection \eqref{5.M.simple}, and  $\pi$  is a local diffeomorphism on $\M_{simple}\setminus \W$. 
\item $\W^1$ is a codimension $1$ submanifold  of  $\M_{simple}$.
\item For each embedded curve $C$, $\M_C=\{\iota_C\}\times \J_C$ is transverse to  $\W^1$.
\end{enumerate}
\end{minipage}
\end{prop}

\begin{proof} 
(a)     In a slice \eqref{5.2bsliceslice}, the projection $d\pi_p(\xi, K)=K$ of a non-zero  element   in $\ker\, {\cal L}_p$  is zero if and only if $\xi=(\zeta, k)$ is a non-zero element of $\ker\, D_p$.   On the other hand,  $D_p$ has index~0, so is onto at each $p\notin \W$.  At such $p$,  for each $K$  we can use (\ref{4.formulabigL}) to obtain $\xi$ with  $\L_p(\xi, K)=0$;  then $(\xi, K)$ is tangent to $\M$ and $d\pi_p(\xi, K)=K$.  Thus $d\pi$ is an isomorphism at each point not in $\W$, and $\W$ is the collection of critical points of $\pi$.

 \smallskip

 (b)    Fix a representative $p_0$ of a point in the wall $\W^1$   and fix a slice ${\mathcal S}lice={\mathcal S}lice_{\tau}^{1, r,l}$ containing $p_0$. Let $\Fred\to {\mathcal Slice}$ be the  fiber bundle  whose fiber over $p=(f,J)$ is the space of index~$0$ Fredholm operators  from  $\E_p=\E_p^{1, r}$ to $\F_p=\F_p^{0, r}$.  By choosing  a smooth local trivialization of $\E$ and  a $C^{l-1}$ local trivialization of $\F$, we  can identify $\Fred$ with the space of Fredholm operators between two fixed Banach spaces.  Then
$\Fred$ is the union of strata $\Fred^s=\{ D\in \Fred\, |\, \dim\ker D=s\}$, where each $\Fred^s$ is  a submanifold of codimension $s^2$ whose normal bundle at $D$ is naturally identified with $\mbox{Hom}(\ker D, \cok D)$  (cf. \cite{K}, \S1.1b, c).   Associating to $p$ the operator \eqref{def.dl.comp} with $m=1$ defines a  section 
 \bear
\label{4.PsiD}
\Psi(p)= D_{p},
\eear 
of the $C^{l-1}$ bundle $\Fred$.  In fact,  $\Psi$  is the vertical derivative of  the section $\Phi$ described before \eqref{4.2Msimpletau}, and hence is of class  $C^{l-2}$.
On the other hand, $\M_{simple}$ is $C^k$ locally diffeomorphic to the submanifold $M_{simple}$ of the slice for $k$ as in  Proposition~\ref{5.1new}b.  Noting that $k\le l-2$, it  suffices to show that  the restriction of  $\Psi$ to $M_{simple}$  is transverse to $\Fred^1$.

For this purpose, we  consider  a deformation $p_t=(f, J_t)$ of $p_0$  in $M_{simple}$, where both the map and the complex structure on the domain is fixed, while $J_t$ is a path in $\J= \J^l$ whose restriction to $f(C)$ is fixed and which changes only in a small neighborhood $U$ of the image $f(x)$ of an  injective point $x$. Along the path $p_t$,  $\E_{p_t}$ is fixed, since it is independent of $J$, as is $\F_{p_t}$, which depends on $J$ only through its restriction along $f(C)$. Thus we have a 1-parameter family of Fredholm maps $D_{p_t}: \E_p\ra \F_p$ with fixed domain and target. 
The initial derivative $\dot{p}_0$ has  the form $v=(0,0, K)\in  \ker \L_p=T_p M_{simple}$,   where $K$ vanishes along  $f(C)$ and is supported in $U$. Hence  the variation $(\delta_v D)_p=\l. \frac {d} {dt}\r|_{t=0} D_{p_t} $  in the direction $v$ is obtained by replacing $J$ by $J_t$ in  \eqref{1.formulaL} and taking the $t$-derivative at $t=0$. Because  the formula \eqref{1.formulaL} is independent of the connection, we can take the variation with the connection fixed; the formula then shows that $(\de_v D)_p$ is $\frac 12 (\nabla K)f_*j$. Moreover, when applied to an element 
 $\xi=(\zeta, k)$ of $\E_p$,  this variation depends only on the normal component $\xi^N=(\zeta^N, 0)$ of  $\zeta$ because $K\equiv 0$ on $f(C)$:
\bear
\label{4.variationDkappa}
 (\de_v D)_p\xi\ =\ \tfrac 12 (\nabla_{\xi^N} K)f_*j.
\eear 
This formula is tensorial in  $\xi$, does not depend on the connection, and is tensorial in   $K\in T_J\J$ as long as $K\equiv0$ along $f(C)$.

 Now fix   generators $\kappa$ of $\ker D_p$ and $c$ of $\ker D^*_p$;  these are continuous by Lemma~\ref{lemma4.1kc}.  Since the normal space to $\Fred^1$ at $p$ is 1 dimensional, it is enough  to construct a variation in $p$ of the form $(0, 0,K)$ such that  the $L^2$ inner product $\lg c, \; (\de_v D)_p\kappa\rg_{L^2}$ is non-zero.

To find such a $v$,  choose an injective point $x\in C$ of $f$ with both $\kappa^N(x)\ne 0$ and $c(x) \ne 0$ as in Lemma~\ref{lemma4.1kc}.  
As in the proof of  Proposition~\ref{5.1new}b,   there is a $K_0\in T_J  \J^l$  supported near $f(x)$ such that $K_0 df j=c$ at  the  point $x$.  Because $f$ is $C^l$, there is a neighborhood of $x$ in $C$ whose image under $f$ is   an embedded $C^l$ submanifold of $X$.  Hence we can choose  a local   $C^l$ coordinate system  $\{z, y_1, y_2, \dots\}$ centered  at $f(x)$ with $z$ a local complex coordinate  on $f(C)$,  and $\{y_i\}$ real coordinates  vanishing along $f(C)$, and such that $\frac{\partial\ }{\partial y_1}{\big|}_{f(x)}=\kappa^N(x)$.
Then $K=y_1K_0$  lies in $T_J\J^l$, vanishes along $f(C)$,  and satisfies 
\bear\label{4.nabla.K.1}
 (\nabla_{\kappa^N} K) f_*j=c \quad \mbox {at the single point $x$}.
 \eear 
 
Finally, set $v_\ep=(0, 0, 2 \beta_\ep K)$  with $\beta_\ep$  as in \eqref{eq.beta.delta}.  Replacing $K$ by $2\beta_\ep K$ in  \eqref{4.variationDkappa} and using \eqref{4.nabla.K.1},  one sees  that 
\bear
\label{4.16}
\lim_{\ep\to 0}\ \int_C\lg  c, \; (\de_{v_\ep}  D)_p\kappa \rg \ =\  |c(x)|^2 \ \not=\ 0.
\eear
In particular,  there is a variation with $\lg c,\; (\de_v D)_p\kappa\rg_{L^2} \ne 0$, which proves statement (c).
\smallskip

Statement (c) holds because the transversality above was obtained using a variation $v_\ep=(0, 2\beta_\ep K)$ tangent to $\M_C$.
\end{proof} 


\subsection{The structure of $\W\setminus\W^1$}  

The  proof of Proposition~\ref{MmfdThm}b extends to show that the part of $\W$ not in $\W^1$  has codimension~3 in the following sense.
\begin{defn}
\label{4.defcodimkset}
We say that a subset $S$ of a manifold $M$ {\em has codimension $k$}  if it  is contained in a countable union $\bigcup \rho_\ell(S^\ell)$, where each $\rho_\ell:S^\ell\to M$ is a Fredholm map of  separable Banach manifolds with $\ind \rho_\ell\le -k$.
\end{defn}

\begin{lemma} 
\label{lemma4.7}
$\W \setminus \W^1$ has  codimension~3 in $\M_{emb}$, and  $(\W \setminus \W^1)\cap \M_C$ has codimension~3 in $\M_C$.  
\end{lemma}

 C.~Taubes obtained  a similar result in dimension~4 using  analytic perturbation theory (Step~5 of the proof of Lemma~5.1 in \cite{t}).  
  A proof in the spirit of the above arguments is given in the  second appendix as Proposition~\ref{PropA4};  Lemma~\ref{lemma4.7} is a special case.


\subsection{The structure of $\W^1$}  
We next examine the portion of the top stratum $\W^1$ of the wall \eqref{Wall} that lies in the  open subset $\M_{emb}$ of $\M_{simple}$ consisting of embedded maps.  Our goal is to show that the projection
\vspace{-2mm}    
\bear
\label{4.pifromW}
\pi_{\W}:\W^1\to\J
\eear
obtained by restricting \eqref{5.M.simple} to $\W^1$ is an immersion off a set of codimension~1.   For this purpose, we  first introduce locally defined functions $\psi$ that vanish transversally along $\W^1$.
 
\medskip

  Fix a slice  ${\mathcal Slice}$ containing   a representative $p_0$ of a point in $\W^1$,  regard $\W^1$ locally as a subset of this slice, and consider the vector bundles $\E=\E^{1,r}$ and $\F=\F^{0,r}$ on the slice.   At every $p\in \W^1$,  the operators  $D_p:\E_p\to \F_p$ and $D^*_p:\F_p^{1,r}\to \E_p^{0,r}$  defined by  \eqref{D.formula.lin} and \eqref{4.2.D*def} have $1$-dimensional kernels.  These kernels determine subbundles 
 \bear
 \label{4.EFsplitting}
\E^0\subset \E, \qquad\qquad  \F^0\subset \F
\eear
along $\W^1$, and the projection $\pi_{cok}:\F^0\to \F/\im D$ onto the cokernel bundle along $\W^1$ is an isomorphism.  By choosing sections  of $\E^0$ and  $\F^0$  along the submanifold $\W^1$ and extending, we can find  non-vanishing  $C^2$  local sections $\kappa$ of $\E$ and $c$ of $\F$, defined  in a neighborhood $U$ of $p_0$ in the slice,  such that the restrictions to $\W^1$ are local sections of $\E^0$ and $\F^ 0$ respectively.  Let $\psi: U\to \R$ be the function defined by 
\vspace{-2mm}
\bear
\label{4.defpsi} 
\psi(q)\ =\ \int_C \langle c_q, D_q\kappa_q\rangle,
\eear
 using the same metrics and volume forms as in \eqref{4.2.D*def}. 

Clearly, $\psi$ vanishes along $\W^1\cap U$, where $D_q\kappa_q=0$.  Differentiating $\psi(q_t)$ for any path $q_t$ in $U$ with $q_0=p\in \W^1\cap U$ and initial velocity $\dot{q_0}=u$ yields several terms, including the variation in the inner product and volume form.  All except one vanish at $q_0=p$  because $D_p\kappa_p=0$ and  $D_p^*c_p=0$, showing that 
\vspace{-2mm}
\best
(d\psi)_p(u)= \int_C  \lg c_p, \;  (\de_u D)_p\kappa_p  \rg. 
\eest 
The proof of  Proposition~\ref{MmfdThm}b  produces   variations   showing that $(d\psi)_p\not= 0$ for all  $p\in \W^1\cap U$. 
Thus the restriction of $\psi$ to $\M_{simple}\cap U$ vanishes transversally  along $\W^1\cap U$.   In particular, we have
\bear
\label{4.4.TW=dpsi}
T_p\W^1=(\ker d\psi)_p \qquad \forall p\in\W^1\cap U.
\eear

\smallskip

\begin{lemma}
\label{Alemma}
Inside $\M_{emb}$, the subset $\A$ of  $\W^1$  where the projection \eqref{4.pifromW} fails to be a immersion   is a codimension~1 submanifold of $\W^1$,  and $\M_C$ is transverse to $\A$. 
\end{lemma}
\begin{proof} 
Fix  a point in $\A$ and a slice containing it, and work locally in  a neighborhood $U$  on which $\psi$ is defined by \eqref{4.defpsi}.  Consider the vector field  $v=(\kappa, 0)$ on $U$,  where  $\kappa$ is  the non-vanishing  local section chosen above  \eqref{4.defpsi}. 
The proof of  Proposition~\ref{MmfdThm}a shows that at every $p\in \W^1$,  $\ker (d\pi)_p$ is spanned by   $v_p$,  so  
$\pi_{\W}$ fails to be an immersion at $p$ if  and only if $ v_p\in T_p\W^1$.  Together with  \eqref{4.4.TW=dpsi},  this gives two local descriptions of $\A$:
\bear
\label{4.defofAtwice}
\A\ =\ \left\{p\in\W^1\, \big|\, (\kappa, 0)_p\in T_p\W^1 \right\}    \ =\ 
\left\{p\in\W^1 \, \big|\, (d\psi(v))_p=0  \right\}.  
\eear
By the second description,  it suffices to show that  the restriction of the function $d\psi(v):U\to \R$ to $\W^1$  vanishes transversally at each $p\in \A$.  

Consider variations in $p=(f,J)$  of the form $w=(0,K)$,  where  $K$ is an element of $T_{J}\J^l$ such that both $K=0$ and $\nabla K=0$ along the image  $f(C)$. We will show that for every such  variation
\bear
\label{4.wtangent}
 w=(0,K)\in T_p\W^1
\eear
and 
\bear
\label{4.intrinsicnabladpsi}
\nabla_w(d\psi(v))\big|_p\ =\     \tfrac12\int_C  \langle c,\;(\nabla_{\kappa^N}\nabla_{\kappa^N} K)df j \rangle.
\eear
The argument used to obtain \eqref{4.nabla.K.1} and   \eqref{4.16}  then produces a $K=\frac 12 y_1^2\beta_\ep K_0$  in $T_J\J^l$  with $K=0$ and $\nabla K=0$ along $f(C)$ that makes the integral \eqref{4.intrinsicnabladpsi} non-zero, which shows transversality  at $p$.  Furthermore, if $p\in\M_C$, these variations $w=(0,K)$ are tangent to $\M_C$.
 Thus both parts of the lemma follow from  \eqref{4.wtangent} and \eqref{4.intrinsicnabladpsi}. 

\smallskip

We will prove   \eqref{4.wtangent} and \eqref{4.intrinsicnabladpsi}  by  
 constructing a 2-parameter family of deformations of $p$ in $U$ that is tangent to $v$ and $w$ at $p$.   
For clarity, write $p$ as $p_0=(f_0, J_0)$.   Start by choosing  a path $J_t$ through $J_0$  with initial velocity $K\in T_{J_0}\J^l$ such that the restrictions of $J_t$ and  $\nabla^0J_t$ to $f_0(C)$ are independent of $t$.    Along the path $q_t=(f_0, J_t)$ in $U$,  $\E_{q_t}$, $\F_{q_t}$ and the operators $D_{q_t}$ defined by \eqref{D.formula.lin} are constant, because all three  depend only on the 1-jet of $J_t$ along $f_0(C)$, which is fixed. 
Thus the path $(f_0, J_t)$ is in $\W^1$ for all $t$, so its initial tangent $w$  satisfies \eqref{4.wtangent}. We can assume that the local sections $\kappa$ of $\E$ and $c$ of $\F$ used to define $\psi$ were  chosen so that $\kappa=\kappa_0$ and $c=c_0$ along the path $(f_0, J_t)$.

Next, choose a smooth family of maps $f_s:(C, j_s)\ra X$ with initial tangent vector $\kappa_0$.  Then 
\bear
\label{4.4.family}
p_{s, t}=(f_s, J_t)
\eear
 is a 2-parameter family in  the slice with $p_*\partial_s=(\kappa_0, 0)=v$ along $p_{0,t}$ and $p_*\partial_t=(0, K)=w$ at $s=t=0$.  Write the restrictions of the sections $\kappa$, $c$ and the operators \eqref{D.formula.lin} as $\kappa_{s,t}$, $c_{s,t}$ and $D_{s,t}$ respectively, and set
\bear\label{def.eta.st}
\eta_{s, t}= D_{s, t} \kappa_{s,t}.
\eear
For $s=0$, $\eta_{0,t} = D_{f_0,J_t}\kappa_0$ is $0$  by construction; we also have $c_{0,t}=c_0$,   $\kappa_{0, t}= \kappa_0$,  and therefore
\bear
\label{4.initial_eta}
\eta_{0,0} = 0,  \qquad   (\partial_t \eta)_{0,0} = 0, 
 \qquad  (\partial_t c)_{0,0}=0, \qquad   (\partial_t \kappa)_{0,0}=0.     
\eear

With this notation, the restriction of  \eqref{4.defpsi} to the family \eqref{4.4.family} is the function
 \bear\label{def.phi.st} 
\psi (s, t)\ =\  \psi(p_{s, t})
\ =\ \int_{C} \lg  c_{s,t}, \;  D_{s,t} \kappa_{s,t} \rg_s
\ =\  \int_{C} \lg  c_{s,t}, \;  \eta_{s,t}\rg_s,
\eear
 where the pointwise inner product and the area form depend on $s$ but not on $t$.
Along the path $p_{0, t}$,  we have  $d\psi(v)=d\psi(p_*\partial_s) = \partial_s\psi$.  Differentiating in the $w$ direction and noting that  $w=p_*\partial_t$  at the origin    then gives 
\best
\nabla_w(d\psi(v)) _{p_0}\ =\  (\pd_t \pd_s \psi)_{0,0}\ =\  (\pd_s \pd_t \psi)_{0,0}.  
\eest 
 
To complete the proof,  we will calculate $(\pd_s \pd_t \psi) _{0,0}$ by differentiating  \eqref{def.phi.st}.  For fixed $s$,   $\kappa_{s, t}$  is a path  of sections of the fixed bundle $f_s^*TX$,  while both  $c_{s, t}$ and $\eta_{s, t}$ are paths of 1-forms on $C$ with values in the fixed bundle $f_s^*TX$ (they are (0,1) forms with respect to the pair $(j_s, J_t)$).  Hence   we have \\
\bear\label{var,phi.t}
(\partial_t \psi)_{s,0}=   \int_{C} \lg ( \partial_t c)_{s,0}, \;  \eta_{s,0}\rg_s +  \int_{C} \lg  c_{s,0}, \;  (\partial_t \eta)_{s,0}\rg_s.
\eear
Now differentiate  \eqref{var,phi.t} with respect to $s$  and evaluate at $s=0$ using \eqref{4.initial_eta}.   The contribution of the first integral vanishes because  
$( \partial_t c)_{0,0}=0$ and $\eta_{0, 0}=0$,  leaving
\bear\label{var,phi.t.s}
(\partial_s\partial_t \psi)_{0,0}\ =\  \partial_s \big|_{s=0}  \int_{C} \lg  c_{s,0}, \;  (\partial_t \eta)_{s, 0}\rg_s. 
\eear
By \eqref{def.eta.st},  $\eta_{s,t}$ is given by the operator  \eqref{D.formula.lin} with  $q=(f_s, J_t)$,  applied to $\kappa_{s,t}=(\zeta_{s,t}, k_{s, t})$.  In the resulting formula, fix 
$s$ and differentiate with respect to $t$.  Because  $f_s$  and $j_s$ are independent of $t$ and $(\partial_tJ_t)_0=K$,  one sees that 
 $(\partial_t \eta)_{s, 0}$  has the general form  
\bear\label{var,phi.t2}
(\partial_t\eta)_{s, 0} \ =\   D_{s, 0} ((\pd_t \kappa)_{s, 0})\ +\ \tfrac 14 (\nabla^0_{\zeta_s} K) (df_s j_s +J_0df_s)  + T_{s,t}(K), 
\eear
where  $T(K)$ is a sum of terms, each linear and tensorial in $K$. The contribution of the first term of \eqref{var,phi.t2}  to \eqref{var,phi.t.s}   is 
\best
\partial_s \big|_{s=0} \int_{C} \lg c_{s,0}, \; D_{s, 0}( (\partial_t \kappa)_{s, 0})\rg_s= \partial_s \big|_{s=0} \int_{C} \lg  D^*_{s, 0}c_{s,0}, \;  (\partial_t \kappa)_{s, 0}\rg_s, 
\eest
as in \eqref{4.2.D*def}.   Taking the $s$-derivative at $s=0$ yields three terms, all of which vanish because $D^*_{0,0}c_{0, 0}=0$ and $(\partial_t \kappa)_{0, 0}=0$.  Similarly inserting the remaining terms of  \eqref{var,phi.t2} into \eqref{var,phi.t.s} and differentiating yields many terms;  all but one vanish because   $K$ and $\nabla^0K$ vanish along $f_0(C)$.  After noting that $f_0$ is $J_0$-holomorphic, one is left with
$$
(\partial_s\partial_t \psi)_{0, 0}\ =\  
\tfrac 12 \int_{C} \lg c_0,\;  \nabla^0_{\zeta_0} \nabla^0_{\zeta_0} K df_0 j_0   \rg_0.
$$
Because  $K$ and $\nabla K$  vanish along $f_0(C)$, this expression is  independent of the  connection, and its dependence on $\zeta_0$ involves only  the normal component $\zeta^N_0=\kappa^N_0$.  Thus
\best
\nabla_w(d\psi(v)) \big| _{p_0}\ =\ (\partial_s\partial_t \psi)_{0,0}\ =\  
 \tfrac12 \int_{C} \lg c_0, \; \nabla_{\kappa_0^N} \nabla_{\kappa_0^N} K df_0 j_0 \rg_0.
\eest  
 This verifies \eqref{4.intrinsicnabladpsi} and completes the proof.
  \end{proof}

\vspace{5mm}
\setcounter{equation}{0}
\section{Local models for wall crossings}
\label{section6}
\bigskip

We next  study the local geometry of the moduli space around a point $p$ on the wall 
$\W^1\setminus\A$. We assume that  $p$ corresponds to  a $J$-holomorphic embedding $f:C\to X$ of a smooth curve $C$,  which we  can regard as the inclusion 
$\iota_C:  C\hookrightarrow X$ of its image.  The goal is to show that the restriction $\pi_\gamma:\M^\gamma_{emb}\to \gamma$ of \eqref{5.M.simple} over a generic path $\gamma\subset \J$     is  a Morse function at $p$, and hence is locally described by a quadratic equation.  Two types of smooth paths  in $\J=\J^l$ passing  through $J$  are relevant for our purposes: 
\medskip

 \noindent{\sc Type A.}  $\gamma_A$ is a path in $\J$ such that the projection $\pi$ is transverse to $\gamma_A$ at $p$.

\medskip

\noindent{\sc Type B.}  $\gamma_B$ is a path in $\J_C$  whose lift $\wt \gamma_B=\{ \iota_C\} \ti \gamma_B$ to $\M_C$ is transverse to $\W^1$ at $p$. 

\medskip

The lemmas at the end of this section show that both types of paths are generic.  But first,  we will use Kuranishi's method to construct a local model for the moduli space over these paths.  It is convenient to  study both types  simultaneously by considering  embedded parameterized disks 
\bear
\label{5.displayS}
S=\{(t, s)\} \subset \J
\eear
 whose $t$-axis is  a path $\gamma_B$  of Type~B and whose $s$-axis is a path 
$\gamma_A$ of Type~A.  We then restrict \eqref{5.M.simple} over $S$ to obtain  
\bear\label{pi.to.S}
\pi_S:\M^S\to S, 
\eear where  $\M^S=\pi^{-1}(S)$ is the moduli space over $S$.  As in Section~5, we regard $p$ as a point in a slice ${\mathcal Slice}$ of the form \eqref{D.slice.S} with $m=1$ and, without changing notation,  locally identify  $\M_{emb}$ with the corresponding submanifold of the slice.

\begin{defn}
\label{5.deflocalmodel}
A 3-ball $B\subset {\mathcal Slice}$ with coordinates $(x, y ,z)$ centered at $p=(f,J)\in \W^1\setminus\A$ is    {\em adapted to $S$ at $p$} if
 \begin{enumerate}[(a)]
\item    $\pi:B\to S$ is given by   $  \pi(x,y,z)= (y,z)$. \\[-2mm]
\item $\gamma_B(t) = (t,0)$ is a Type B path whose lift is $\tilde{\gamma}_B(t) = (0, t,0)$.\\[-2mm]
\item $\gamma_A(s) = (0,s)$ is a Type A path. 
\\[-2mm]
\item In terms of the splitting \eqref{4.ETJ}, $T_pB$ is spanned by $\partial_x|_p = (\kappa, 0)$, 
$\partial_y|_p = (0, K_B)$, and $\partial_z|_p =(0,K_A)$,  where $\kappa$ generates $\ker D_p\cong \R$, $K_A=\dot\gamma_A(0)$ and $K_B=\dot\gamma_B(0)$.   
\end{enumerate}
\end{defn}
With these assumptions, the transversality conditions in Type A and B are equivalent to
\bear
\label{type.AB.int}
(a)\ \ K_A\not\in\im d\pi_p
\qquad\mbox{and}\qquad 
(b)\ \ (0, K_B) \not\in T_p \W^1
\eear
respectively.  The requirement  that $\gamma_B\subset \J_C$ also implies that  $K_B$ vanishes along   $f(C)$, so $(0, K_B)\in \ker \L_p=T_p\M_{emb}$;  because $d\pi_p(0, K)=K$ this also means that  $K_B\in  \mathrm {im} \;d\pi_p$ and hence $K_A$ and $K_B$ are linearly independent.

\vspace{.5cm}

\begin{theorem}[Kuranishi model]
\label{Kuranishi5.2}
 For each $p=(f, J)$ in $\M_{emb}\cap(\W^1\setminus \A)$ and each $S\subset \J$ as in \eqref{5.displayS} centered at $J$, there is a 3-ball $B$ adapted to $S$ at $p$ such that  $\M^S$ is locally  the 2-manifold 
\bear
\label{finalS}
V\ =\ \big\{(x,y,z)\in B\, \big|\, z= x\big(ax+by + r(x, y)\,\big) \big\}
\eear
with $a, b \not= 0$,  where   $r(x,y)=O(x^2+y^2)$ near the origin.  Moreover $\W^1\cap \M^S$ is locally modeled on the zero locus of the function   $w:\M^S\to\R$ given by
\best
w(x, y)=z_x(x, y)=2ax+ by+(xr(x, y))_x. 
\eest
The tangent space $T_p(\W^1 \cap \M^S)$ is the  intersection of the kernels of the 1-forms $dz$ and
\best
d w\ \ =\   2a\,dx+ b\,dy
\eest
at the origin. 
\end{theorem}

\begin{figure}[h]
\label{figure2}
  \begin{minipage}[c]{2.1in}
\begin{tikzpicture}
\node[anchor=south west,inner sep=0] at (0,0) {\includegraphics[scale=1.25]{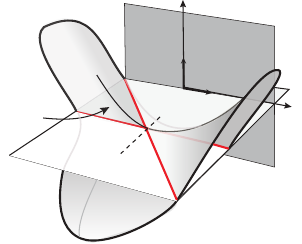}};


\node at (4.6, 3.45) {$K_B$};
\node at (3.5, 3.8) {$K_A$};
\node at (1.48, 4.2) {$\M^S$};
\node at (.1, 2.8) {$\W^1\cap \M^S$};
\node at (2.4, 1.6) {$\ker D_p$};
\node at (2.95, 4.75)  {$S$};
\node at (4.14, 5.1) {$z$};
\node at (6.3, 2.7) {$y$};

\end{tikzpicture}
  \end{minipage} \hspace{11mm}
 \begin{minipage}{3.3in} \begin{flushright} 
\caption{ In the local model, $\M^S$ is a saddle  and $\pi:\M^S\to S$ is the projection onto the $yz$-plane (at the back).  At the origin, $T_p\M^S$ is spanned by $\partial_x \subset \ker D_p$ and $\partial_y=(0, K_B)$, while  $\partial_z=(0,K_A)$ is normal  to $\M^S$. }
\end{flushright}  \end{minipage}
\end{figure}
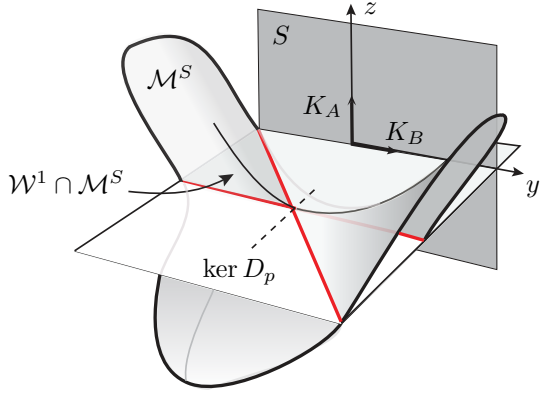

Theorem~\ref{Kuranishi5.2} shows that $\M^S$ is locally a saddle surface in $B\subset\R^3$ given as the graph of a function $z=z(x,y)$ that  has a non-degenerate critical point at the origin. 
Before giving the proof, we record the two cases that will be used in later sections.

\begin{cor}
\label{Cor4.3}  
Suppose that $p=(f,J)\in \M_{emb}\cap(\W^1\setminus \A)$ and  $\gamma$ is a path in  $\J$ through $\gamma(0)=J$.   
\begin{enumerate}[(a)]
\item If $\ga$ is of Type~A,   then $\M^\ga$ is locally modeled at $p$ by
\bear
\label{KModel1}
\big\{(x,  t)\, \big|\, t= ax^2\big\},  
\eear
 with $a\ne 0$ and $\pi_\gamma(x,  t)=t$.

\smallskip  
\item If $\ga$ is of Type~B,  then there is a disk $S\subset \J$ centered at $p$  locally containing $\gamma$ such that $\M^S$ is locally   modeled at $p$ by
\bear
\label{KModel2}
\big\{(x, t, s )\in B\, \big|\, s= x(ax+bt+r(x, t))\big\}, 
\eear
 with $a, b\ne 0$, $r(x,t)=O(x^2+t^2)$, and  $\pi_S(x,t,s)=(t, s)$, and such that  
 the restriction  to $\M^\ga$ is the restriction to the plane $s=0$, namely
\bear
\label{KModel2B}
\big\{(x, t)\, \big|\, 0= x(ax+bt+r(x, t))\big\}.
\eear
\end{enumerate}
\end{cor}

\begin{proof}
In the first case, take $\gamma_A=\gamma$,  choose $\gamma_B\subset \J_C$ a path whose tangent vector $K_B$ satisfies (\ref{type.AB.int}b), and  choose a local embedded disk $S\subset \J$ containing both $\gamma_A$ and $\gamma_B$ locally near $p$.  Then apply  Theorem~\ref{Kuranishi5.2} and restrict to the plane $\{y=0\}$ to get the local model for $\M^\gamma$ of the form 
\best
\big\{(x,  t)\, \big|\, t= x(ax+ r(x))\big\}  \qquad\mbox{with $a\ne0$,  $\pi_\gamma(x,  t)=t$,}
\eest
and with $r(x)=O(x^2)$ for small $x$.  This becomes \eqref{KModel1} after reparameterizing $x$.

Similarly, in the second case,   take  $\gamma_B=\gamma$,  fix a direction $K_A\in T_J \J$ satisfying (\ref{type.AB.int}a), choose $S$  with $T_JS=\mbox{span}(K_A, K_B)$ containing $\gamma$, and again apply  Theorem~\ref{Kuranishi5.2}.  
\end{proof}

The distinction between the  local models \eqref{KModel1} and \eqref{KModel2B} is crucial.  By  Lemma~\ref{4.Jisol.e.0} below,   Model \eqref{KModel1} applies where  a generic path in $\J$ crosses a wall.  This is precisely the local model for the creation (if $a>0$) or annihilation  (if $a<0$) of a pair of curves in the moduli space.   Similarly,  Lemma~\ref{Lemma5.7} shows that Model \eqref{KModel2B} applies where  a generic path in the subspace $\J_C\subset \J$ crosses a wall.  But the model \eqref{KModel2B} is {\em not a  manifold}:  it is the  union of two curves crossing transversely at the origin.    It can be smoothed using the  parameter $s$ in  \eqref{KModel2}, as   will be done in Lemma~\ref{IsotopyLemma6.4}.

\vspace{3mm}

\begin{proof}[Proof of Theorem~\ref{Kuranishi5.2}] 
Specializing \eqref{D.slice.S} and \eqref{4.ETJ},  one sees that the slice over $S$ is  $Map^{1,r}(\Sigma, X)\times   {\mathcal S}_{\tau}  \times S$ and  that we can identify the first two factors with their tangent space at $p$, which is  $\E_p=\E^{1,r}_p$.   We can also trivialize  the bundle  $\F=\F^{0,r}$  over a neighborhood of $p$.    With these identifications,  $\M^S$ is the subset of $\E_p\times S$ that is  the zero set of the $\F_p$--valued  $C^{l-1}$ function $F$ defined by  \eqref{4.defF}.  This  has an expansion
\bear
\label{5.8*}
F(\xi)\ = \ \L_p(\xi) +Q(\xi),
\eear
where  $Q(\xi)$ vanishes to first order at $\xi=0$. Next,  fix  a generator 
$\kappa$ of $\ker D_p=\E^0$ and choose a decompositon $\E_p=\E^0\oplus \E^+$. 
Then   $\xi$ can be written  as  $\xi=(x, y, z, \alpha)$, for coordinates $(x,y,z)$ 
in a $3$-ball ${\mathcal B}  \subset   \E^0 \ti S$ around $0$ and 
$\alpha\in  \E^+$.   Using (\ref{4.formulabigL}), the linear term is
\begin{align*}
\L_p(\xi) & \ =\ xD_p\kappa +\tfrac12(yK_B+zK_A)f_*j +D_p\alpha   \\
& \ =\ z c_p +D_p\alpha,
\end{align*}
where we have set  $c_p=  \tfrac12 K_Af_*j$  and  noted that $K_B=\dot{\gamma}_B(0)$ is tangent to a Type~B path, so $K_B|_{f(C)}=0$.   Furthermore,  the transversality assumption (\ref{type.AB.int}a) ensures that $c_p\notin \im D_p$ as follows:  if $c_p=D_p\mu$,  then  $(-\mu, K_A)$ is an element of $\ker \L_p =T_p\M_{simple}$ with $d\pi(-\mu, K_A) = K_A$,  contradicting (\ref{type.AB.int}a).   Thus 
$\F_p$ decomposes as 
$$
\F_p= \widehat{\F}^0 \oplus \widehat\F^+
$$
where $\widehat\F^0$ is the real span of $c_p$ and $\widehat\F^+= D_p\E^+$. Using this decomposition, write $Q$ as $\left(Q_0\cdot c_p, Q_1 \right)$.

By Proposition~A.4.1 of \cite{ms} there is a bounded linear map $T:  \widehat\F^+\to \E^+$ that is a  pseudo-inverse of $D_p$,   which  implies that  $D_pT$ is the identity on $\widehat\F^+$.   Define a map 
$$
\eta:   \E^0 \ti S \ti  \E^+ \to  \E^0\ti S\ti \E^+ 
$$
by   
$$
 \eta(x, y, z, \al)=(x, y, z,  \eta_{x, y, z}(\al)), \qquad \mbox{where}\qquad   \eta_{x, y, z}(\al)\ =\ \al+ TQ_1(x, y, z, \alpha) \in \E^+.  
$$   
By the Inverse Function Theorem, $\eta$ is a local diffeomorphism near $0$.

Using the above notation, \eqref{5.8*} can be rewritten as the equation
 $$
 F(x,y,z,\alpha)\ =\ \big((z + Q_0)c_p,\,D_p(\eta_{x, y, z}(\alpha)) \big)\ \in\   \widehat\F^0 \oplus \widehat\F^+.
 $$
 This shows that $F(x,y,z,\alpha)=0$ if and only if  both  $\eta_{x,y,z}(\alpha)= TD_p(\eta_{x,y,z}(\al)) = 0$, and $z  = q$ where 
 $$
  q(x,y,z) = -Q_0\circ\eta^{-1}(x,y,z,0). 
 $$
  Thus there is a  local diffeomorphism  
$$
\M^S \cong\ \big\{(x,y,z)\in B\ \big|\ z=q(x,y,z)\big\}.
$$
The  real-valued function $q$ is smooth and vanishes   to first order at the origin, so we can solve for $z$ as a function  of $x$ and $y$ to obtain
\bear\label{4.eq.z}
z=z(x, y)= ax^2+ bxy+cy^2+ r_2(x, y),
\eear 
where the remainder $r_2$ vanishes to second order.  In particular, locally near $p$, $\M^S$ is a  2-manifold  with  coordinates $(x, y)$, in which the projection  \eqref{pi.to.S} is  
\best
\pi_S(x, y)=(y, z(x, y)),
\eest 
so $d\pi_S= (dy, z_x dx+ z_yd y)$.

Next observe that, since $\gamma_B$ is a Type B path, (\ref{type.AB.int}b) implies that $\W^1$ is transverse to $\M^S$ at $p$.  Thus $\W^1 \cap \M^S$ is a 1-dimensional manifold near $p$. On the other hand,  $\pi_S:\M^S\to S$ is a map between 2-manifolds, and the above formula shows that $\mbox{rank}\, d\pi_S \ge 1$.  The  proof of  Proposition~\ref{MmfdThm}a  then shows that the  set of critical points of $\pi_S$ is  contained in $\W^1\cap \M^S$.

 In coordinates, the intersection $\W^1\cap \M^S$ is locally the zero locus of the  function
\best
w(x, y)= z_x(x, y),
\eest
while $T_p(\W^1 \cap \M^S)=T_p \W^1\cap T_p \M^S$ is the kernel of the 1-form on $S$
\best
d w\ =\ z_{xx} dx+ z_{yx} dy\ =\   2adx+ bdy
\eest
at the origin.  Consequently, since $(\kappa, 0)= \partial_x$ and $(0,\dot \gamma_B(0))= \partial_y$, we have 
\begin{itemize}
\item $2a= dw(\kappa, 0) \not= 0$ by \eqref{4.defofAtwice} because $p\notin \A$ so 
$(\kappa, 0) \not \in T_p\W^1$. \vspace{2mm}
\item $b=  dw(0, K_B) \not= 0$ by (\ref{type.AB.int}b)  because  $\wt \gamma_B$ is   transverse to $\W^1$.\vspace{2mm}
\item $c =0$ because,  for small $t$,  the path  $\tilde{\gamma}_B(t)=(0,t,0)$ lies in $\M^S$, so  \eqref{4.eq.z} becomes $0=ct^2 +\O(t^3)$ for all small $t$.
\end{itemize}
In fact, the expansion in the third bullet point shows that  $r_2(0, t)= 0$ for all small $t$, which implies that \eqref{4.eq.z} has the form  
\best
z(x, y)= x( ax+by + r(x, y)), 
\eest
where $r$ vanishes to first order.
\end{proof}

\vspace{4mm}

 To apply Theorem~\ref{Kuranishi5.2} and its corollary, we will need several statements about generic paths in $\J$. These are  consequences of the Sard-Smale Theorem  (\cite{s}, \cite[Theorem A.5.1]{ms})  applied  in the following manner.  Suppose that $\pi:\M\to\J$ and $\rho:\N \to\M$  are Fredholm maps of separable $C^\ell$ Banach manifolds of index $\iota_\pi$ and $\iota_\rho$ respectively, and with $\ell \ge 2+\max(0, \iota_\pi, \iota_\pi+\iota_\rho)$.  Fix points  $J_0, J_1\in\J$, and let $\P=\P\J$ be the space of $C^k$,  $k\gg 1$,  paths $[0,1]\to\J$ from $J_0$ to $J_1$, which is a separable Banach manifold.

In this context, we have the following lemma. 
\begin{lemma} 
\label{L.Sard.smale}
Assume $J_0, J_1$ are regular values of $\pi$ and $J_0, J_1\not\in (\pi\circ\rho)(\N)$. Then there exists a Baire set $\P^*$ of $\P$ so that for each $\gamma\in \P^*$,  $\M^{\gamma}=\pi^{-1}(\gamma)$  is an  $\iota_\pi+1$ dimensional submanifold of $\M$ transverse to  $\rho$.
In particular, when $\iota_\rho+ \iota_\pi\le -2$, $\M^\gamma$ is disjoint from $\rho(\N)$.  
Consequently, for each subset  $S\subset \M$ of codimension $\ge 2$, there is a Baire set $\P^*_S$ so that for each $\gamma\in \P^*_S$,  $\M^{\gamma}$ is a manifold disjoint from $S$. 
\end{lemma}

\begin{proof} 
The evaluation map $\ev: \P \ti I \ra \J$ is a submersion of separable Banach manifolds 
 away from the boundary $\P \ti \bd I$, while the image $\ev(\P\ti \bd I)=\{ J_0, J_1\}$ of the boundary consists of regular values of $\pi$.   It follows that the map 
$$
\phi = \ev\ti \pi: \P \ti  I \ti \M\to \J\ti \J  
$$
  is transverse to the diagonal $\Delta_\J$, so the fiber product 
\best
 \wt \M\ =\ \phi^{-1}(\Delta_\J)\ =\ 
 \big\{(\ga, t, f) \, \big|\, f\in \M^{\gamma(t)}\big\} 
\eest 
is a separable Banach manifold  whose boundary is the fiber product of $\ev|_{\P \ti \bd I}$ and $\pi$.   Furthermore, the projection $\wt \pi:\wt \M \ra \P$ is a Fredholm map of index $\iota_\pi+1$, while the projection $p_3:\wt \M \ra \M$ onto the third factor is a submersion  away from the boundary:
\begin{gather}
\label{5.10}
\xymatrix{\wt \N\ar[r]^{\tilde{\rho}}\ar[d]& \wt \M \ar[d]^{p_3}\ar[r]\ar[d] &\P\ti I \ar[d]^{\ev} \ar[r] &\P
\\
\N\ar[r]^\rho&\M \ar[r]^\pi &\J.
}
\end{gather}
 But by assumption, $\pi\circ \rho (\N)$ is disjoint from the image $(\pi \circ p_3)(\bd \wt \M)=\{J_0,J_1\}$ of the boundary.  
Therefore  the fiber product $\wt\N$ of  $\rho$ and $p_3$ is a manifold and the index of $\tilde{\rho}$ is the index of $\rho$.

 By Sard-Smale Theorem applied to $\wt \pi: \wt \M\ra \P$, there is a Baire subset $\P_1$ of $\P$ such that each point $\gamma\in \P_1$ is a regular value of $\wt \pi$.  It is straightforward to check that $\gamma\in \P$ is a regular value of $\wt \pi$   if and only if the path $\gamma$  is transverse to $\pi:\M \ra \J$.  When $\gamma$ is a submanifold of $\J$, this latter transversality means that $\M^{\gamma}=\pi^{-1}(\gamma)$ is a submanifold of $\M$.

 Similarly,   again by the Sard-Smale Theorem, there is a Baire subset $\P_2$ of $\P$ such that each point $\gamma\in \P_2$ is a regular point of the composition $\wt \N \ra \P$ in the top row of \eqref{5.10},  and again
 this occurs  if and only if $\gamma$ is transverse to $\pi\circ \rho$.  (Note that $J_0, J_1$ are regular values of $\pi\circ \rho$ because they are not in the image of $\pi\circ \rho$.) When $\gamma$ is a submanifold,  this last transversality implies that $\N^\gamma=(\pi\circ \rho)^{-1}(\gamma)=\rho^{-1}(\M^\gamma)$ is a submanifold.

Finally, the  collection $\P_3$ of embedded paths is open and dense in $\P$.  Thus  $\P^*=\P_1\cap\P_2\cap \P_3$  is a  Baire set. For each $\gamma\in \P^*$,  $\M^{\gamma}$ is an $\iota_\pi+1$ dimensional manifold,  $\N^\gamma$ is an  $\iota_\pi+1+\iota_\rho$ dimensional manifold,  and is empty if $\iota_\pi+1+\iota_\rho<0$. The last statement of the lemma follows from Definition~\ref{4.defcodimkset} and the fact that every  countable intersection of Baire sets is a  Baire set.
\end{proof}

We will apply this reasoning twice:  first for paths in $\J= \J^l$, then for paths in $\J_C$.  Recall the notations $\J_E^*$ and $\J^E_{isol}$ from Section~1. For simplicity, we omit the $X$ from the notation $\M(X)$ of the moduli spaces in Lemmas~\ref{4.Jisol.e.0} and \ref{Lemma5.7}.

\medskip

 \begin{lemma}
 \label{4.Jisol.e.0}
 Any path in $\J$ with endpoints in $\J^*_E$ can be deformed, keeping its endpoints, to a path $\gamma$ such  that $\M^{\gamma, E}_{simple}$ is a  1-dimensional manifold, consisting of embeddings, and  intersecting the wall $\W$ transversely in isolated points, all in $\W^1\setminus \A$.  Moreover, any path $\gamma$ with these properties is in $\Ji^{E}$. 
 \end{lemma}
 
 \begin{proof} 
By  Proposition~\ref{5.1new},  $\M_{simple}$ is a  separable Banach  manifold  of class at least $C^2$, and by Corollary~\ref{corA.3}  the set  $\mathcal{NE}$ of all non-embedded simple maps  is a codimension 2 subset.  
 Next, restrict to $\M_{emb}$, noting that $\W^1$ is a submanifold of $\M_{emb}$ by  Proposition~\ref{MmfdThm}b, and that  $\A\subset \M_{emb}$ is a codimension~2  submanifold by Lemma~\ref{Alemma}.  Together with Lemma~\ref{lemma4.7}, this shows that   $\A\cup\mathcal{NE}\cup (\W\setminus \W^1)$ is a set of codimension 2 in $\M_{simple}$. 
   
Now apply Lemma~\ref{L.Sard.smale} with $\M=\M_{simple}$ and $\N$  equal to the disjoint union $\A\sqcup \mathcal{NE}\sqcup (\W\setminus \W^1)$,  noting that the index $\iota_\pi=0$   and $\iota_\rho \le -2$  in this case.
This  gives a  Baire subset $\P_1$ of $\P$  over which $\pi^{-1}(\gamma)$ is  a manifold of dimension $1$ that does not intersect $\N$.  Again apply  Lemma~\ref{L.Sard.smale}, now with $\M=\M_{emb}$ and  $\N$ equal to the  codimension~1 submanifold $\W^1$.  This gives a second  Baire subset  $\P_2$ of $\P$ over which $\pi^{-1}(\gamma)$ is  a manifold of dimension $1$ that is  transverse to $\N$.  

Consequently,  for each $\gamma$ in the Baire set ${\cal P}_1\cap {\cal P}_2$,  $\M^{\gamma, E}_{simple}$ is a 1-manifold (with boundary) transverse to $\W^1$, intersecting the wall $\W$ only along $\W^1 \setminus \A$, and consisting only of embedded curves. Note that the only critical points of $\pi_\gamma:\M^{\gamma, E}_{simple}\ra [0, 1]$ are these wall-crossing points.  The proof is completed by observing that the local model  (\ref{KModel1}) shows  that the wall crossing points are non-degenerate critical points of $\pi$ and therefore (a) they are isolated points of $\M^{\gamma, E}_{simple}$, and (b) all points of the fiber $\M^{\gamma(t), E}_{simple}$ are isolated for each $t$. 
 \end{proof}

 \begin{cor}
 \label{Cor5.6}
 $\J^E_{isol}$  is a dense and path-connected subspace of $\J$. 
 \end{cor}
 \begin{proof}
Density was shown in Corollary~\ref{Cor1.3}.  Path-connectedness follows from 
Lemma~\ref{4.Jisol.e.0} and the fact that  $\J$ is path-connected. 
\end{proof}

We conclude this section by proving a version of Lemma~\ref{4.Jisol.e.0} for paths in the subspace $\J_C$ of $\J$.

 \begin{lemma}
 \label{Lemma5.7}
 Any path in $\J_C$ with endpoints in $\J^*_E$ can be deformed, keeping its endpoints, to a path $\gamma$ in $\J_C\cap \Ji^E$ whose  lift $\wt \ga$  intersects $\W$ transversally at finitely  many points,  all in $\W^1\setminus \A$.
\end{lemma}

\begin{proof} 
 First consider  the subset $\M_C=\{\iota_C\}\ti \J_C$.    This  is  a submanifold of $\M_{emb}$ that is transverse to  $\W^1$  by  Proposition~\ref{MmfdThm}c,  and to $\A$  by   Lemma~\ref{Alemma}.  Furthermore, Lemma~\ref{lemma4.7} shows that 
 $(\W\setminus \W^1)\cap \M_C$ is a codimension~2 subset of $\M_C$.
   By  Lemma~\ref{L.Sard.smale}, there is a  Baire subset of paths $\gamma$  in $\J_C$ for which
the lift $\wt \ga=(\iota_C, \gamma)$ to $\M_{simple}$ intersects the  wall only along $\W^1\setminus \A$, and this intersection is a finite set of transverse points.   For each intersection point $\wt\gamma(t)\in \W^1\setminus \A$,   the local model \eqref{KModel2B}  implies that the core curve $\wt\ga(t)$ is an isolated point of $\M^{\gamma(t)}$ (and is clearly embedded).   The same conclusion is true
for those $t$ with $\wt\gamma(t)\not\in \W$  by  Proposition~\ref{MmfdThm}a.

It remains to find another  Baire subset of paths $\gamma$ for which the  points of $\M^{\gamma(t)}\setminus \M_C$ are embedded and isolated for each $t$. Denote by $\M_{simple}^*\ra \J_C$ the  moduli space of simple $J$-holomorphic maps $f$ that have at least one point $x_i$ on each component of their domain with $f(x_i)\in X\setminus C$.  The results of  Propositions~\ref{5.1new} and \ref{MmfdThm},  Lemmas~\ref{lemma4.7} and \ref{Alemma},  and Corollary~\ref{corA.3} all  extend to the moduli space $\M_{simple}^*\ra \J_C$ by using variations supported around the points $f(x_i)$, but vanishing along $C$; such variations are tangent to $\J_C$.   As in the proof of \cite[Lemma 3.4.3]{ms} a further variation, with  support off $C$, can be used to ensure that all curves in $\M^*_{simple}$ are transverse to $C$.  Again,  such variations are tangent to $\J_C$. 

With this understood, the proof of Lemma~\ref{4.Jisol.e.0} extends to give a Baire subset  $\P^*$ of the space $\P\J_C$ of paths in $\J_C$ so that for each $\gamma\in\P^*$, the points of $\M^{\gamma(t), E}_{simple}$  are embedded and isolated for all $t$.   
\end{proof}

 \vspace{5mm}
\setcounter{equation}{0}
\section{The cluster isotopy theorem}
\label{section7}
\bigskip

For notational simplicity, given two clusters $\O=(C,\ep, J)$ and $\O'=(C', \ep', J')$ whose core curves $C$ and $C'$ have the same genus and homology class, write 
$$
GW^E(\O)\approx GW^E(\O')
$$
to mean that  the difference is a finite sum of terms of the form $\pm GW^E(C_i, \ep_i, J_i)$ of strictly higher level (\ref{defLevel})  compared to that of $C$.   With this notation, for example, the conclusion of the cluster refinement Corollary~\ref{L.cluster.ref}  simply says that for generic $0<\ep'<\ep$ 
\bear
\label{5.1}
GW^E(C,\ep, J)\ \approx GW^E(C, \ep', J).
\eear

We now use the results of Sections~5  and 6   and an isotopy argument to prove that the GW series of every cluster is equivalent, in the above sense, to the series of an elementary cluster.  Recall that, for an elementary cluster $\O_{elem}$, $GW(\O_{elem})$ is the universal series 
\bear\label{7.A}
GW^{elem}_g(q^C, t)
\eear
given by  \eqref{3.defGelem} and \eqref{DefZq} with $q=q^C$.  In general,   we call a  cluster $(C, J, \ep)$  {\em regular} if the embedding $C\hookrightarrow X$ is a regular $J$-holomorphic map.

\begin{theorem}[Cluster Isotopy]
\label{CIT_6.1} 
For a regular cluster $\O=(C, J_0, \ep_0)$ centered at an embedded genus $g$ $J_0$-holomorphic curve $C$, 
\bear
\label{C.contrib}
GW^E(\O)\ \approx\  \sign(C, J_0)\, GW^{elem,E}_{g} (q^C, t),
\eear
 where $GW^{elem,E}_{g}$ is the truncation of \eqref{7.A} below energy $E$.
\end{theorem} 
\begin{proof} 
The proof of Proposition~\ref{3.Junhocurves} shows that there exist $J_1\in \J_C$  and $\ep_1$ so that $\O_{elem}=(C, \ep_1,  J_1)$ is an elementary cluster. In fact, we can assume that  $J_1\in \J_C\cap \J^*_E$ after a perturbation supported outside the $\ep_1/2$-neighborhood of $C$ of the type constructed  in  the proof of \cite[Lemma 3.4.3]{ms}.
Choose a path $\gamma(t)=J_t$ in $\J_C$ from $J_0$ to $J_1$ (the proof of Theorem~A.2 of \cite{IP2} shows that $\J_C$ is connected). By Lemma~\ref{Lemma5.7} we can assume, after a deformation, that $\gamma$ is a path in  $\J_C\cap \Ji^E$ and there is a finite set $Sing=\{t_i\}$, not containing 0 or 1,  such  that $\tilde{\ga}=(f_C, \gamma(t))$
\begin{itemize}
\item   lies in $\M^\ga\setminus \W$ for all $t\notin Sing$, and
 \item  lies in a  2-dimensional surface $V_i$ given by  \eqref{KModel2} for $t\in[t_i-\delta, t_i+\delta]$.
\end{itemize}
Choose $\de>0$ small enough so that the  intervals $[t_i-\delta, t_i+\delta]$ do not overlap, and let their endpoints be $0<\tau_1<\dots <\tau_{2k}<1$. For each $i$, fix a cluster $\O_i=(C, \ep_i, J_{\tau_i})$.   Then  $\wt\ga$ can be regarded as the composition of paths $\wt\ga_i:[\tau_i, \tau_{i+1}]\to \M^\ga$ of two types:
\begin{enumerate}[(i)]
\item  Paths in $\M_{emb}\setminus\W$.  For these,  Lemma~\ref{IsotopyLemma6.2} below shows that $GW^E(\O_i) \approx GW^E(\O_{i+1})$. 
\item  Paths in $\M_{emb}\cap V_i$, crossing  the wall transversally at a single point of $ \W^1\setminus \A$.  For these,  Lemma~\ref{IsotopyLemma6.4}   below shows that $GW^E(\O_i) \approx -GW^E(\O_{i+1})$.
\end{enumerate}
Altogether, we conclude that 
$$
GW^E(\O) \approx  (-1)^\sigma\, GW^E(\O_{elem}),
$$
where $\sigma$ is the number of transverse wall-crossings, which is exactly the spectral flow of the operator $D_p$ along the path $\tilde{\ga}$.  The path ends at an elementary cluster, which has positive sign by  \eqref{logZ}.   Thus $(-1)^\sigma$ is exactly the sign of the initial curve $(C, J_0)$.
 \end{proof}

\medskip

In the above proof, the assertion  in Step~(i) is a fact about isotopies with no wall-crossings.  It can be stated as follows.

\begin{lemma}[Simple Isotopy]
\label{IsotopyLemma6.2} 
Fix $E>0$.  Then for   any path   $(C_t, J_t)$ in  $\M_{emb}\setminus\W$ with $J_t$ in  $\J_{isol}^E$ and any $\ep_0, \ep_1$ such that $(C_0, \ep_0, J_0)$ and $(C_1, \ep_1, J_1)$ are clusters,  
\bear\label{GW.defm.J.0}
GW^E(C_0, \ep_0, J_0)\ \approx\  GW^E(C_1, \ep_1, J_1).
\eear 
\end{lemma}
\begin{proof} 
It follows from  Proposition~\ref{MmfdThm}a and the compactness of $[0,1]$ that there is a $\delta>0$ such that, for each $t\in [0,1]$,  $C_t$ is the only $J_t$-holomorphic curve in its degree and genus in the ball $B(C_t, \de)$ (in Hausdorff distance).  By Lemma~\ref{L.cluster.exists}  we can choose, for each $0\le t\le 1$, an $0<\ep_t<\delta $ such that $(C_t, \ep_t, J_t)$ is a cluster.  Then, by Lemma~\ref{lemma2.1}, $(C_s, \ep_t, J_s)$ has a well-defined contribution $GW^E (C_s, \ep_t, J_s)$ for all $s$ in an open interval around $t$.  These open intervals cover $[0,1]$; take a finite subcover $\{I_k\}$.  Then $GW^E(C_s,\ep_k, J_s)$ is constant for $s$ in each $I_k$ and $C_s$ is the only $J_s$-holomorphic curve in its genus and homology class in that ball.   Corollary~\ref{L.cluster.ref}  shows that on the intersection of two consecutive intervals the corresponding $GW^E$ invariants differ by the contributions of higher-level clusters. The lemma follows.
\end{proof}

\medskip

By Lemma~\ref{4.Jisol.e.0}, each path in $\J$ with endpoints in $\J^*_E$ can be deformed, keeping its endpoints, to a path $\gamma$ in $\Ji^E$ such that the projection
\bear\label{5.eq.pi}
\pi_\gamma:  \M^\gamma_{emb} \to \ga
\eear
has only non-degenerate critical points, none an endpoint, each locally modeled by  \eqref{KModel1}.   If $a>0$ in the local model,  then $\gamma$ can be  parameterized so that  $\gamma(t)=t$ and $\pi^{-1}(t)=\{x\,|\, t=ax^2\}$ is empty for $t<0$ and is two distinct curves $C_t^\pm$ for $0<t<\delta$  (and vice versa   if $a<0$).  A second isotopy lemma relates the GW invariants of clusters centered on these curves $C_t^\pm$.

\begin{lemma}[Wall-crossing in $\J$]
\label{IsotopyLemma6.3}
Fix $E>0$, a path $\ga$  in $\J_{isol}^E$ and  a non-degenerate critical point $(C_0, J_0)$ of \eqref{5.eq.pi} for $J_0=\ga(0)$. 
Then there exists a $\de>0$ and a neighborhood $U$ of $(C_0,J_0)$ in $\M$ such that if $0\not= |t|<\delta$ and the sign of $t$ is such that $\M^{\ga(t)}\cap U=\{C^\pm_t\}$, then the  two clusters $\O^+=(C_t^+, \ep, J_t)$, $\O^-=(C_t^-, \ep', J_t)$ satisfy 
$$
GW^E(\O^+)\ \approx\ -GW^E(\O^-).
$$
 \end{lemma}
\pf 
The local model  (\ref{KModel1}) at $(C_0, J_0)$ implies that there is an $\ep_1>0$  and a ball $U=B(C_0, \ep_1)$ in $\C(X)$ that contains $C^\pm_t$ and no other $J_t$ holomorphic curves in the degree and genus of $C_0$ for all $|t|<\ep_1$. Because $J_t\in \J_{isol}^E$, Lemma~\ref{L.cluster.exists} ensures that $\ep_1$ can be chosen so that  $(C_0, \ep_1, J_0)$ is a cluster.   
\begin{wrapfigure}[8]{r}{0.4\textwidth}
\labellist
\pinlabel $\O^-$ at 10 35   
\pinlabel $\O^+$ at  10 75 
\pinlabel $\M^\gamma\cap\W^1$ at 108 69
\pinlabel $\pi$ at 78 26  
\pinlabel $\J$ at 106 9 
\endlabellist
\centering
  \includegraphics[scale=1]{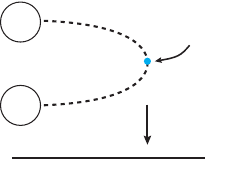}
\end{wrapfigure}
As $J$  varies, the associated invariant    $GW^E(U, J_s)$ is, by Lemma~\ref{lemma2.1},  well-defined and independent of $s$  for  small $s$. 

The local model \eqref{KModel1} shows that $U\cap\C^{J_s, E}$ is $\{C^\pm_t\}$  for $s=t$, and is empty for $s=-t$. Taking $s=t$ and applying  Proposition~\ref{P.cluster.decomp.exist},  one sees that $U$ decomposes into two clusters $\O^\pm_t$ with $GW^E(\O^+) +  GW^E(\O^-) \approx GW^E(U, J_s)$.  
Applying the same theorem with $s=-t$ shows that $GW^E(U, J_s)\approx 0$.
The proof is completed by noting that the invariants  $GW^E(\O^\pm)$ satisfy (\ref{5.1}) as $\ep$ and $\ep'$ vary.
\qed

 \bigskip
 
The core curve of a cluster does not persist through the wall-crossing described by Lemma~\ref{IsotopyLemma6.3}.   But the core curve remains if we fix the complex structure on the core curve $C$ and cross the wall along a path $\ga$ in $\J_C$, as was done in the proof of Theorem~\ref{CIT_6.1}.  After a   perturbation as in Lemma~\ref{Lemma5.7},  the wall-crossing is locally modeled  by  \eqref{KModel2B}.  In this picture, for each $0<t<\delta$, there are four curves to consider: 
 the incoming core curve $(C, J_{-t})$, the outgoing core curve $(C, J_t)$, and a second pair of curves  $(C'_{-t}, J_{-t})$ and $(C'_t, J_t)$.

\begin{lemma}[Wall-crossing in $\J_C$]
\label{IsotopyLemma6.4}
Fix $E>0$ and a path $\ga$  in $\J_C\cap \J_{isol}^E$ so that $\wt\ga(t)$ crosses the wall transversally at $t=0$ at  a point $(C, J_0)$ in $\W^1\setminus \A$. 
Then there exists a $\de>0$ so that each incoming cluster $\O_{-\delta}=(C, \ep, J_{-\delta})$ and each outgoing cluster $\O_{\delta}=(C, \ep', J_\delta)$ satisfy
$$
GW^E(\O_{-\delta})\ \approx\ -GW^E(\O_{\delta}).
$$
 \end{lemma}

\begin{proof} Consider the local model $\M^S\to S$ given by (\ref{KModel2}). Its restriction over $\gamma$, given by  (\ref{KModel2B}), 
is two curves crossing at the origin.  We will perturb this level set $\{ z=0\}$ in two opposite directions.  

\addtocounter{figure}{1}
\begin{figure}[ht!]
\labellist
\small\hair 2pt
\pinlabel {\bf $Wall$} at  137 34
\pinlabel $A_s$ at   10.8 41
\pinlabel $B_s$ at  135 56
\pinlabel $D_s$ at   29.8 84.5
\pinlabel $\pi$ at   82 26
\pinlabel ${\mathcal J}$ at   128 5
\pinlabel $A_s$ at    241 54
\pinlabel $B_s$ at  366.1 42
\pinlabel $D_s$ at   240.8 74.8
\pinlabel ${\mathcal J}$ at  366 5
\pinlabel $\pi$ at   313 26
\endlabellist
\centering
\includegraphics[scale=1.1]{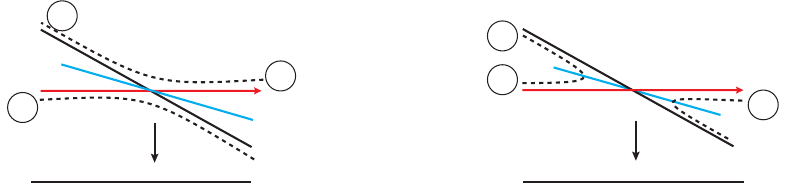}
\caption{These figures show the curve $C$ (horizontal line) and curves $C'_t$ (diagonal line) as $t$ (the horizontal coordinate) varies.  The circled labels refer to clusters at the ends of the dotted paths, with $s>0$ in the first figure, $s<0$ in the second, and $a,b>0$ in both.}
\end{figure}

In the chart  (\ref{KModel2}), with $\delta>0$  fixed and small,  $(C, J_{-\delta})$ has coordinates  $(0,-\delta,  0)$, so can be perturbed to $(C_{-\delta,s}, J_{-\delta,s})$  with coordinates  $(x(s),-\delta,  s)$, where  $x\approx -s/b\delta$ is the unique solution of $s= x(ax-b\delta+ r(x, -\de))$ with $x =O(s)$. The curves  $(C'_{-\delta}, J_{-\delta})$  and $(C, J_{\delta})$  can be similarly perturbed.  By Lemma~\ref{L.cluster.exists},  the  GW invariants of the corresponding clusters
$$
A_s=(C_{-\delta,s}, \ep, J_{-\delta, s}), \qquad
B_s=(C_{\delta,s}, \ep, J_{\delta, s}), \qquad
D_s=(C'_{-\delta,s}, \ep, J_{\delta, s})
$$
are locally constant in $s$: for sufficiently small $s$ (and $\de>0$) we have
\bear
\label{5.6}
GW^E(A_s)=GW^E(\O_{-\delta}), \qquad 
GW^E(B_s)=GW^E(\O_{\delta}), \qquad 
GW^E(D_s)= GW^E(D_{-s}).
\eear

Assume $a>0$ (else change $s\ra -s$), and that $b>0$ (else change $t \ra -t$). The  moduli space $\M^S$ over $S$ is locally near $(C,  J_0)$ the level set    $$
\big\{(x,t,s)\, \big|\, s=x(ax+bt +r(x,t))\big\}.
$$
  For each fixed $s, t$ small, this quadratic equation in $x$ has either no solution or two solutions, except at a single point $x\approx -bt/2a $, where the tangent is in the kernel of the projection to $\J$, which means that this point lies on the wall and is a non-degenerate critical point of (\ref{5.eq.pi}).

For a small positive $s$, the moduli space over $\ga_s(t) = (t, s)$, $-\delta\le t\le \delta$ therefore contains a path in  $\M_{emb}$ from the core of cluster $D_s$ to the core of $B_s$ that does not cross the wall.  After a small perturbation using 
Lemma~\ref{4.Jisol.e.0},  Lemma~\ref{IsotopyLemma6.2} applies to give: 
\bear
\label{5.7}
GW^E(D_s)\ \approx\  GW^E(B_s).
\eear 
For a small {\em negative} $s$ the moduli space over $\ga_s(t) = (t, s)$, $-\delta\le t\le 0$ is a path in  $\M_{emb}$ from the core of cluster $A_s$ to the core of $D_s$, crossing now the wall transversally (at a point in $\W^1\setminus \A$). Perturbing $\ga_s$ by Lemma~\ref{4.Jisol.e.0} gives a path in $\Ji$ so Lemma~\ref{IsotopyLemma6.3} applies in this case to give 
\bear
\label{5.8}
GW^E(A_s)\ \approx\ -GW^E(D_s).
\eear
The proof is completed by combining (\ref{5.6}), (\ref{5.7}) and (\ref{5.8}). 
\end{proof}

\vspace{5mm}
\setcounter{equation}{0}
\section{Structure theorems and the proof of the GV conjecture}
\label{section8}
\bigskip

The isotopy results of the previous section lead quickly to a formula (\ref{GW=creative.1})  that shows that the GW invariants have a remarkably simple structure. This formula is compatible with a simple geometric picture:  if one could find a $J\in\J$ so that all simple $J$-holomorphic maps in $X$ were elementary, then  $GW(X)$ would have exactly the form  \eqref{GW=creative.1},  with $e_{A,g}(X)$ equal to the  count of $J$-holomorphic {\em curves} with homology class $A$ and genus $g$.  However,  it is far from clear whether any such $J$ exists.   Thus the coefficients  $e_{A,g}(X)$ can be regarded as   {\em  virtual counts of elementary clusters in $X$}.

\begin{theorem}
\label{Theorem7.1}
For any closed symplectic Calabi-Yau 6-manifold $X$, there exist {\em unique integer} invariants  $e_{A, g}(X)$ such that 
\bear\label{GW=creative.1} 
GW(X)\ =\ \sum_{A\ne 0}\sum_{g\ge 0}\  e_{A, g}(X) \cdot GW^{elem}_{g}(q^{A}, t).
\eear
\end{theorem}

\begin{proof}  
 The uniqueness of the coefficients in (\ref{GW=creative.1})  is easily shown because  the collection of series
\bear
\label{expansionGWelem}
GW_g^{elem}(q^{A}, t)=t^{2g-2} q^{A}\, \Big(1+\mbox{ higher order in $t$ and $q^A$}\Big)
\eear
for $g\ge 0$ and $A\in H_2(X, \Z)$ is linearly independent. To prove existence, fix $E$, choose any parameter $J\in \J^*_E$, and use Proposition~\ref{P.cluster.decomp.exist} to write $GW^E(X)$ as a sum of finitely many cluster contributions. Formula (\ref{GW=creative.1}) follows from the corresponding formula for each cluster, which is proved in Lemma~\ref{Lemma7.2} below, by taking $E\to\infty$.
\end{proof}

\begin{lemma}
\label{Lemma7.2} 
For any regular $E$-cluster $\O$ centered at a genus $g$ curve $C$ there exist unique integers $e_{d, h}(\O)$, beginning with $e_{1, g}(\O)=\sign(C)$,  such that
\bear\label{6.3}
GW^E(\O) = \sum_{d\ge 1} \sum_{h\ge g} e_{d, h}(\O)\ GW^{elem, E}_{h} (q^{dC}, t),
\eear
where both sides are truncated below energy level $E$. 
\end{lemma} 

\begin{proof}    
Because all $J$-holomorphic maps in $\O$ represent $k[C]$ and have genus at least $g$, $GW^E(\O)$ has the form
\bear\label{GW.o.o}
GW^E(\O)\ =\ \sum_{k\ge 1} \sum_{h\ge g} GW^E_{k, h}(\O) \ q^{kC} t^{2h-2}
\eear 
with $k\omega(C)\le E$ and $h\le E$.   Define the  {\em $(C,g)$--relative level} of the monomial $t^{2h-2}q^{kC}$  to be  $\Omega(k)+ h-g$, and note that all terms in \eqref{GW.o.o}  have non-negative relative level.  

  Using this series \eqref{GW.o.o}, we define the truncation $[GW^E(\O)]_m$ of the lefthand side of \eqref{6.3} to be the sum of the terms in \eqref{GW.o.o} with $(C, g)$-relative level  $\Omega(k)+ h-g\le m$. The righthand side of \eqref{6.3} can be similarly truncated. 
In fact, by \eqref{expansionGWelem} the truncation of  \eqref{6.3} involves  only those $e_{d,h}(\O)$ with $\Omega(d)+h-g\le m$. We will prove the lemma using complete induction on $m$.

 The induction begins with  $m=-1$; in this case, the truncations of both sides of  \eqref{6.3}  vanish. For the induction step,   we assume that for {\em every} regular cluster $\O$, whose core curve corresponds to any  $(A, g)$, 
  there are coefficients $e_{d, h}(\O)\in \Z$ such that \eqref{6.3}  holds when
 truncated at  $(A,g)$--relative level  $m-1$.   Now by Theorem~\ref{CIT_6.1} we have
\best
GW^E(\O)\ =\ \pm GW^{elem, E}_{g} (q^C, t)\ +\  \sum_{i\in I} \pm GW^E(\O_i),
\eest
where the  $\O_i$ are clusters,  indexed by a finite set $I$,  whose core  curves $C_i$ have $[C_i]=k_i[C]$,  genus $g_i\ge g$ and $(C, g)$--relative level $m_i=\Omega(k_i)+ g_i-g>0$.   When  $GW^E(\O)$ is truncated at relative level $m$, each $GW^E(\O_i)$ is truncated at  $(C_i, g_i)$-relative level $m-m_i<m$ so by induction, 
\best
GW^E(\O)\ =\ \pm GW^{elem, E}_{g} (q^C,  t)\ +\  \sum_{i\in I} \pm  \Big(\sum_{d, h} e_{d, h}(\O_i)\;GW^{elem, E}_{h} \ (q^{d k_iC}, t)\Big)
\eest
holds when truncated at  $(C, g)$--relative level $m$. This completes the induction step. 
 \end{proof} 
 
 In fact, we get the following result for {\em any} closed symplectic 6-dimensional 
manifold $X$ as long as we restrict to the GW invariants coming only from  classes  $A\in H_2(X,\Z)$ with vanishing Chern number $c_1(A)=c_1(X)A$.
\begin{theorem}
\label{T.GW=creative.3-fold} 
Assume $X$ is a  closed symplectic 6-manifold. Then there  exist  unique integer invariants  $e_{A, g}(X)$, defined for  homology classes $A$ with  $c_1(A)=0$, such that the GW invariant of $X$ satisfies
\bear\label{GW=creative.1.3fold} 
\sum_{\substack{A\ne 0\\ c_1(A)=0}}\sum_{g\ge 0} GW_{A, g}(X)\ t^{2g-2} q^A\ =\ \sum_{\substack{A\ne 0 \\ c_1(A)=0}}\sum_{g\ge 0}\  e_{A, g}(X) \cdot GW^{elem}_{g}(q^{A}, t).
\eear
\end{theorem}

\begin{proof} 
The dimension \eqref{1.dimformula} is $2c_1(A)$, independent of the genus. It suffices to check that all the results in Sections~1-6 continue to hold as long as we replace  everywhere $\oM(X)$ by the union of its zero dimensional pieces 
\bear\label{M.CY.part}
\ma\bigsqcup_{\substack{A \ne 0 \\ c_1(A)=0}} \oM_{A, g}(X).
\eear
A dimension count shows that for generic $J$ the limit points of \eqref{M.CY.part} in the rough topology (after restricting below fixed energy level $E$)  can only be multiple covers of points of  \eqref{M.CY.part}, and not of points with  $c_1(A)\ne 0$. The rest is straightforward, and   the details are left to the reader.
\end{proof}

  \vspace{2mm}
 
\subsection{Proof of the GV conjecture}  The GV conjecture follows easily from Theorem~\ref{Theorem7.1} and the explicit form of the GW invariant of an elementary cluster.  For simplicity, set
$$
 {\cal E}_h(q, t)\ =\ \sum_{k=1}^\infty \frac 1k\l(2\sin \frac{kt}2\r)^{2h-2}  q^{k}.
$$
With this notation,  the GW invariant \eqref{3.3} of an elementary cluster whose core curve has genus $g$  is
\bear
\label{8.GWelemseries}
GW_g^{elem}(q,  t)  \ =\  \ma\sum_{d\ne 0} \sum_{h\ge g} \ n_{d, h}(g) \ {\cal E}_h(q^d, t),
\eear
where the  $n_{d, h}(g)$ are integers by Proposition~\ref{theorem3.4}a  that vanish unless $h\ge g$ by Proposition~\ref{theorem3.4}b.    The
 GV Conjecture then  takes the following form.
\begin{theorem}
\label{Theorem8.1} 
Let $X$ be a closed symplectic Calabi-Yau 6-manifold.  Then there are unique integers $e_{A, h}(X)$  such that 
\bear\label{GW=creative} 
GW(X)\ =\ \sum_{A\ne 0} \sum_{h} e_{A, h}(X) \ {\cal E}_{h}(q^{A}, t).
\eear
In fact, these BPS numbers $n_{A, h}(X)$  are obtained from the virtual counts $e_{A, h}(X)$ of Theorem~\ref{Theorem7.1}  by the universal formula involving the coefficients  $n_{d,h}(g)$ in  \eqref{8.GWelemseries}:
\bear
\label{7.3}
n_{A, h}(X)= \sum_{\substack{d, B\\ dB=A}}\sum_{g=0}^h \ e_{B, g}(X) \cdot n_{d, h}(g) \in \Z,
\eear
 where the first sum is over all integers $d\ge 1$ and $B\in H_2(X, \Z)$ such that $d B=A$ in $H_2(X, \Z)$.
\end{theorem}

\begin{proof} 
This follows immediately by combining (\ref{GW=creative.1}) and  (\ref{8.GWelemseries})  and rearranging the sums:
\best
GW(X)& = & \sum_{A\ne 0} \sum_{g\ge 0} e_{A, g}(X) \sum_{d\ge 1} \sum_{h\ge g}  n_{d, h}(g)\  {\cal E}_h(q^{dA}, t)  \\ 
 & = &\sum_{A\ne 0} \sum_{d \ge 1}\sum_h \Big( \sum_{g\le h} e_{A, g}(X)\,  n_{d, h}(g) \Big)\  
 {\cal E}_h(q^{dA},  t)   \\
& =& \sum_{A\ne 0}  \sum_{\substack{d, B \\ d B=A}} \sum_h \Big( \sum_{g\le h} e_{B, g}(X) \,  n_{d, h}(g) \Big)\  {\cal E}_h(q^A, t).  
\eest
The rearrangements are justified by first working below an energy level $E(A,g)\le E$, where all sums are finite.
\end{proof}

\vspace{5mm}
\setcounter{equation}{0}
\section{Extensions of the GV structure theorem} 
\label{section9}
\bigskip

  This section extends  Theorem~\ref{Theorem8.1}  in two different directions:  to general symplectic $6$-manifolds,  and to  the genus  zero GW invariants of  closed symplectic $n$-manifolds, $n\ge 6$, that are semipositive (as defined in 
 \cite{ms}),  a class that includes  symplectic Calabi-Yau manifolds.  In fact,  all  transversality results  were proved for simple maps in index zero moduli spaces.  A version of the Cluster Decomposition Proposition~\ref{P.cluster.decomp.exist}  holds provided    the underlying curve map  (\ref{1.maptoC}) does not increase the dimension of such moduli spaces in the sense described below.

\medskip

 We restrict to the {\em primary GW invariants} of $X$, which are defined using the evaluation map (but not the stabilization map) in (\ref{0.1}).  For each collection  $\{\ga_i\} \subset H^*(X, \Z)$ consider the generating function 
\bear\label{D.pGW}
GW^X(\ga_1, \dots, \ga_k)\ =\  \sum_{A\ne 0} \sum_{g\ge  0}\  \lg  [\oM_{A, g, k}^J(X)]^{vir}, \ev^*(\ga_1\ti  \dots \ti \ga_k) \rg \  q^A t^{2g-2}.
\eear
The pairing is defined to be zero unless the  formal dimension is zero, that is, unless 
$\iota=0$, where
\bear\label{dim.M.cut}
\iota=2c_1(A) + (\dim X-6)(1-g) +2k -\sum_{i=1}^k \dim \ga_i.
\eear
As usual, the pairing vanishes unless $\dim \ga_i\ge 2$ for each $i$, so we henceforth assume this inequality. Throughout this section, we  assume that $\dim X \ge 6$. 

The coefficients in \eqref{D.pGW} are obtained by fixing pseudo-cycles $\beta_i:B_i \to X$ representing the Poincar\'{e} duals of $\ga_i$ (cf. \cite[Section~6.5]{ms}), and restricting to the  index $\iota=0$ constrained moduli space
\bear\label{M.cut}
\oM_{A, g, B}^J(X)=  \oM_{A, g, k}^J(X) \, \times_{X^k}  (B_1\ti \dots \ti B_k)
\eear
(the fiber product of  the evaluation map $\ev:   \oM_{A, g, k}^J(X)\to X^k$  and the map $B_1\ti \dots \ti B_k \to X^k$).  This gives rise in the usual way to the primary GW invariant that appears as the coefficients in \eqref{D.pGW}  (cf. page 197 of \cite{ms}). 
 
Lemma~\ref{Lemma1.1} remains true for these  index~0  constrained moduli spaces  in the following form.  Let $\D$ be the (countable) set consisting of the indexing data $(A, g, \ga)$ appearing in \eqref{D.pGW}.  For each $\gamma=(\gamma_1,\dots, \gamma_k)$, choose a set of pseudo-cycle representatives $\beta_1, \dots, \beta_k$ that are in general position.  Standard transversality results show that, for each element of $\D$,   each open stratum of  the constrained universal moduli space 
\bear\label{D.pi.cut}
\oM_{A,g,B}(X)_{simple} \to \J
\eear 
is a manifold.  
The Sard-Smale Theorem gives a Baire set of regular points in  $\J$  for the map \eqref{D.pi.cut} for each $\D$; after intersecting over the  elements of $\D$ we can assume these are regular  for all $\D$.   Parts (a) and (b) of  Lemma~\ref{LemmaA.1}, together with  the Sard-Smale Theorem,  give two similar Baire sets.   Another intersection produces a single Baire set $\J^*$ of $\J$ such that for each $J\in\J^*$ all index~0 moduli spaces \eqref{M.cut} satisfy: 
\begin{enumerate}[(a)]
\item  All simple $J$-holomorphic maps are regular,   and are embeddings  with pairwise disjoint images that are {\em $B$-regular}, meaning that for each $i=1, \dots, k$,  $f(x_i)$ is a regular value of $\beta_i$ and  $(\beta_i)_*(T_{b_i}B_i) \cap f_*(T_{x_i} C) =0$ for each $b_i\in B_i$ with $\beta_i(b_i)=f(x_i)$. 
\item The projection  \eqref{D.pi.cut} is  a local diffeomorphism around each map that is  regular, $B$-regular,  and an embedding.
\end{enumerate}
Moreover,  for each $J\in \J^*$, there are no simple $J$-holomorphic maps in the spaces \eqref{M.cut} with $\iota<0$.  
The universal moduli space constrained by $B$ is 
\bear\label{M.cut.2}
\oM_B(X)=\ma\bigsqcup_{A, g} \ \oM_{A, g, B}(X),
\eear 
where the disjoint union is only over those $(A, g)$ for which $\iota$ in \eqref{dim.M.cut} is zero. 

 As in the proof of Lemma~\ref{Lemma1.4}, any nontrivial $J$-holomorphic map $f:C\ra X$ has an associated ``reduced map'' $\phi:C_{red} \ra X$, which is a simple $J$-holomorphic map with the same image as $f$.  In this context, the underlying curve map \eqref{1.firstunderlyingcurve} extends to a map 
\bear
\label{curve.B}
c:\oM_B(X)\ma\lra  \mathrm{Subsets}(X) \times \J \times B_1\ti \dots \ti B_k
\eear
defined by $(f, x_1, \dots, x_k,  J, b_1, \dots, b_k ) \mapsto (f(C), J, b_1, \dots, b_k)$.  
The examples  below give structure theorems in cases where $c$ does not increase the formal dimension \eqref{dim.M.cut}, that is, where $\iota(\phi)\le \iota(f)$.

Under this assumption, we can replace  $\oM(X)$ everywhere by \eqref{M.cut.2} and all proofs in Sections~1-7, except those in Section~3, hold without change. In particular, there is a dense, path-connected   set $\Ji^E(B)$ corresponding to $\Ji^E$ in Definition~\ref{1.defJisol} but involving only maps in $\oM_B(X)$.  Lemma~\ref{Lemma1.5}  holds under the assumption above with  $\C(X)$ replaced by the image of \eqref{curve.B}. 

To finish, we must  expand the definitions of ``cluster'' and ``elementary cluster''.   Define a {\em $B$-constrained cluster} exactly as in Definition~\ref{D.cluster}, but using only elements $(f, {\bf x}, J, {\bf b})$ in $\oM_B(X)$.  Thus the core $C$ is a smooth embedded  $J$-holomorphic curve $\iota_C:C \ra X$ that we now assume is marked, $B$-regular as defined  in (a) above, and decorated by a choice of  $b_i\in \beta_i^{-1}(\iota_C(x_i))$ for each $i$. The contribution of a $B$-constrained cluster $(C, J, \ep)$  to $GW(\ga)$ depends only on the restriction of $J$ to the $\ep$ neighborhood of the core curve and the restriction of each $\beta_i$ to the $\ep$-ball in $B_i$ centered at $b_i$, for $i=1, \dots, k$. By a diffeomorphism, when $\ep$ is small, we can identify the $\ep$ tubular neighborhood of $C$  with an $\ep$-disk bundle of the normal bundle $N_C\ra C$, with $C$ mapping to the zero section and each $\ep$-ball around $b_i$ in  $B_i$ mapping into a linear subspace of the fiber $N_i$ over the points $p_i=\iota_C(x_i)$. 
One can then declare certain  $B$-constrained clusters to be ``elementary''.   In both of the examples below there is a simple, natural way of doing this.

\subsection{GV-formula for general symplectic 6-manifolds}

For any closed symplectic 6-manifold $X$,  the dimension (\ref{dim.M.cut})  is 
$$
\iota = 2c_1(A) +  \sum_{i=1}^k\  (2-\dim \ga_i), 
$$
independent of the genus.   For the  ``Calabi-Yau classes'' $A$ with $c_1(A)=0$ consider the GV-transform 
\bear\label{eq.GV.c1big-0}
\sum_{\substack{A\ne 0, g\\ c_1(A)=0}} GW_{A,g}\  q^A t^{2g-2}=
\sum_{\substack{A\ne 0, g\\ c_1(A)=0}} n_{A,g}  \sum_{k=1}^\infty \frac 1k \l( 2\sin \frac {kt}2\r)^{2g-2}  q^{kA}.
\eear
For ``Fano  classes'' $A$  with  $c_1(A)>0$ consider the following variation of the GV transform: 
\bear\label{eq.GV.c1big}
\sum_{\substack{A, g\\ c_1(A)>0}} GW_{A,g}(\ga_1, \cdots, \ga_k)\  q^A t^{2g-2}=
\sum_{\substack{A, g\\ c_1(A)>0}} n_{A,g}(\ga_1, \cdots, \ga_k) \l( 2\sin \frac t2\r)^{c_1(A)+ 2g-2}  t^{-c_1(A)}q^A
\eear
for each collection  $\{\ga_i\}\subset H^*(X, \Z)$.  The invariants $GW_{A,g}$ are zero for all  classes $A$ with $c_1(A)<0$ (the moduli space without constraints is empty for $J\in \J^*)$.

\begin{theorem}
\label{Theorem9.1} 
For a closed  symplectic 6-dimensional manifold $X$, the   coefficients of the  primary GW series \eqref{eq.GV.c1big-0}  and \eqref{eq.GV.c1big}  have the following integrality properties:
\best
n_{A, g}\in\Z \ \  \mbox{ if $c_1(A)=0$,  \qquad and \qquad $n_{A, g}(\ga_1, \cdots, \ga_k) \in \Z$ \ \   if $c_1(A)>0$}
\eest
 for all  $\ga_1, \dots, \ga_k\in H^*(X,\Z)$.
\end{theorem}    
\begin{proof} 
Fix $\ga=\{\ga_i\}$,  corresponding constraints $B=\{B_i\}$,  and a class $A\ne 0$ so that 
$$
 \iota=2c_1(A)+\sum (2-\dim \ga_i)
$$
 is zero. For each $J\in \J^*$, the resolution of each  $J$-holomorphic map $f$, factors as $\phi\circ \rho$ as described above, where $\phi:C_{red} \ra X$ is a simple $J$-holomorphic map.   For each component $\Si_i$ of $C_{red}$, let $A_i = \phi_*[\Si_i]\in H_2(X,\Z)$, and let $d_i\ge 1$ denote the degree of $\rho$ over $\Si_i$, that is,   the number of points in $\rho^{-1}(x)$ for a generic point $x\in \Sigma_i$.
Then $A=\sum d_i A_i$. 

Because $\phi:C_{red} \ra X$ is a simple $J$-holomorphic map, it cannot have any components with $c_1(A_i)<0$  (such moduli spaces are empty for $J\in \J^*$). Moreover, the image of $\phi$ passes through all the constraints and  represents $\ma\sum A_i$. 
Hence  
$$
c_1(\phi_*[C_{red}]) \ =\ \sum c_1(A_i) \ \le\  \sum d_i c_1(A_i) \ =\ c_1(A);
$$
 in fact this must be an equality because otherwise the formal dimension of the constrained moduli space containing $\phi$ would be negative, contradiction.   Thus for $J\in\J^*$,  a multiple cover map represents a Calabi-Yau class   if and only if its reduced curve also does, and represents a Fano class  if and only if   its reduction also does and $\deg \rho=1$.

Consequently,  the GW series separates into  two independent contributions: a sum over the Calabi-Yau classes, where  Theorem~\ref{T.GW=creative.3-fold}  applies, and  a sum over the Fano classes that  was studied by A.~Zinger  \cite{z}. Theorem~\ref{T.GW=creative.3-fold} combines with the proof of Theorem~\ref{Theorem8.1}  to give the integrality of $n_{A,g}$ in \eqref{eq.GV.c1big-0}.  The Fano case is much simpler:   there is no need to consider clusters because for  
$J\in\J^*$, every embedded $J$-holomorphic curve $C$ with $c_1(A)>0$ is isolated and super-rigid for the constrained moduli space, and the contribution of degree~1 maps from nodal curves to $C$ is precisely  
\best
GW(C)=\l( 2\sin \frac t2\r)^{c_1(A)+ 2g-2} t^{-c_1(A)} q^{C}
\eest
(see (1.13)  and (1.14) in \cite{z}).  This completes the proof.  
\end{proof}

\subsection{Genus zero invariants of  semipositive manifolds} 

There is a similar structure theorem for the rational (genus zero) GW invariants of  closed semipositive symplectic manifolds of  dimension $\ge 6$. In this context, the appropriate GV transform has two parts:
\begin{enumerate}
\item For $c_1(A)=0$, it is the Aspinwall-Morrison formula in the form given in equation (2) in \cite{k-p}:
\bear\label{GV.semi.1}
\sum_{\substack{A\ne 0\\\ 
c_1(A)=0}} GW^X_{A, 0} (\ga_1, \dots ,\ga_k)\  q^A\ =\  \sum_{\substack{A\ne 0\\ c_1(A)=0}} n^X_{A, 0} (\ga_1, \dots ,\ga_k) \sum_{d\ge 1} d^{k-3} q^{dA}
\eear
\item  For $c_1(A)>0$, it is \eqref{eq.GV.c1big} specialized to genus zero:
\bear\label{GV.semi.2}
GW^X_{A, 0}(\ga_1, \dots, \ga_k)\ =\ n^X_{A, 0}(\ga_1, \dots, \ga_k).
\eear
\end{enumerate}
As before, the invariants $GW_{A,g}$ are zero for all classes $A$ with $c_1(A)<0$ (since $X$ is semipositive, there are no simple $J$-holomorphic spheres with $c_1(A)<0$ for $J\in \J^*)$. 

\begin{theorem}
\label{theorem9.2}  
For a  closed semipositive symplectic  manifold $X$ of dimension at least 6, the coefficients  \eqref{GV.semi.1} and \eqref{GV.semi.2} of the primary genus zero GW series have the following integrality property:
\best
n_{A, 0}^X(\ga_1, \cdots, \ga_k) \in \Z
\eest
for all  $\ga_1, \cdots, \ga_k\in H^*(X,\Z)$.
\end{theorem} 
\begin{proof}
Again fix $\ga$,  $B$,  and  $A$ so that   
$$
\iota= 2c_1(A)+ (\dim X-6)+\sum (2-\dim \ga_j)
$$
 is zero.   As before, assume $f$ is a multiple cover  with   reduced map $\phi$, and $A_i$, $d_i$ are degrees of its components.  If  the domain of $\phi$ has $r\ge 1$  components then its image has at least $r-1$ self-intersection points (since the domain of $f$ was connected). For $J\in \J^*$,  these impose $(r-1)(\dim X-4)$ transversely cut conditions on simple maps, so the dimension of the moduli space containing $\phi$ is  
$$
\ma\sum_{i=1}^r \big[2c_1(A_i) + (\dim X-6)\big]\ -\ (r-1)(\dim X-4) \ +\ \sum_{j=1}^k  (2-\dim \ga_j). 
$$
Since $A=\sum d_iA_i$ and $c_1(A_i)\ge 0$,  this is less than or equal to $\iota -2(r-1)\le \iota = 0$.   But  the moduli space is empty unless this is an equality, so we conclude that $r=1$ and $d_1=1$ whenever $c_1(A)\not= 0$.  Thus the GW series again separates into a sum of $c_1(A)=0$ classes and a sum over $c_1(A)>0$ classes.
 
 \smallskip

The Fano case \eqref{GV.semi.2} is classical: dimension counts imply that for generic $J$ the constrained moduli space consists only of simple maps, without any multiple cover.  Thus the GW invariant is an integer.  
   
    \smallskip
      
For Calabi-Yau classes $A$ in $X^{2m}$ we declare a $B$-constrained cluster to be elementary if its core $C$ is an embedded marked rational curve with normal bundle $N$ biholomorphic to $\O(-1)\oplus \O(-1)\oplus (m-3)\O$, and the constraints $B_i$ are linear subspaces of fibers $N_{f(x_i)}$ of $N$ in general position in the sense that the only holomorphic section of $N$ which intersects $B_i$ for every $i=1,\cdots, k$ is the zero section. (Note that because $\iota=0$, the sum of the codimensions of $B_i$ in $N_{f(x_i)}$ is $2(m-3)=\dim_\R X-6$.)
The core curve $C$  is then super-rigid (the constraints kill the kernel in the $\O$ directions) and, as proved in  \cite{k-p},  the contribution of its multiple covers to the primary $GW(\gamma_1, \dots, \gamma_k)$ invariant is 
\bear\label{9.last}
  \sign_{\! B} C\cdot  \sum_{d\ge 1} d^{k-3} q^{d C},
\eear
where $\sign_{\! B} C=\pm 1$ is the sign of the core curve $C$ as an element of the cutdown moduli space $\oM_{A, 0, B}$.  This sign can be explicitly calculated as the sign of the (transverse) intersection between the oriented linear subspace $B_1\ti \dots \ti B_k$ and the image of the evaluation  
map $\ev_{\bf x}: H^0(C, f^*N) \ra N_{f(x_1)}\ti \dots \ti N_{f(x_k)}$ on the space of holomorphic sections of $f^*N$.

 Now, with $\gamma$, $B$ and $E$ fixed, restrict attention  to the  cut-down moduli spaces $\oM_{A, 0, B}(X)$ for Calabi-Yau classes $A$ with energy at most $E$.  For  $J$ in the   set $\J^*$ constructed after \eqref{D.pi.cut},  we can decompose the fiber $\oM^{J}_{A, 0, B}$ into $B$-constrained  clusters  as in Proposition~\ref{P.cluster.decomp.exist}.    For each such cluster, the proof of    Theorem~\ref{CIT_6.1}  shows that one can isotop $J$ to a  elementary cluster of the above type;  the resulting formula \eqref{C.contrib} then becomes 
   \best
GW^E(\O)\ \approx\   \sign_{\! B} C \cdot GW^{elem,E} (q^C),
\eest
where the righthand side is the contribution \eqref{9.last}  expressed in terms of the formal power series   
  \best
GW^{elem}(q)=\sum_{d\ge 1} d^{k-3} q^d. 
\eest
Since the collection of series $GW^{elem}(q^A)$, for Calabi-Yau classes $A$, are linearly independent as in \eqref{expansionGWelem},   Theorem~\ref{Theorem7.1} also extends 
to the context of $B$-constrained clusters to express the lefthand side of \eqref{GV.semi.1} as a linear combination, with integer coefficients, of the series $GW^{elem}(q^A)$. Thus the coefficients on the righthand side of \eqref{GV.semi.1} are integers.

\end{proof}

 \vspace{5mm}
\setcounter{equation}{0}
\setcounter{section}{0}
\setcounter{theorem}{0}

\renewcommand{\theequation}{A.\arabic{equation}}
\renewcommand{\thetheorem}{A.\arabic{theorem}}
\appendix{}
\section*{A{\sc ppendix} A}
\bigskip

The proofs of Lemma~\ref{Lemma1.1}a and  Lemma~\ref{lemma4.7} were  deferred;    we give the details here and in the second appendix.   The proofs are applications of transversality and 
  the Sard-Smale Theorem.  While we are primarily interested in Calabi-Yau 6-manifolds,  Propositions~\ref{PropA.2} and \ref{PropA4}  below apply to symplectic manifolds $(X, \w)$ with $\dim X\ge 6$. 
 Lemmas~\ref{Lemma1.1}a  and \ref{lemma4.7} are special cases of  Corollary~\ref{corA.3} and  Proposition~\ref{PropA4}  respectively.

  As in the proof of Lemma~\ref{Lemma1.4}, every   $J$-holomorphic map $f:C\ra X$ lifts to a map $\wt f:\wt C\ra X$  from the normalization of $C$.  This lift
 is  $J$-holomorphic and has a smooth (but not necessarily connected) domain, so it suffices to work with such maps.  For every integer $\ell \ge 0$,  let  
 \bear\label{M.k.simple}
\M_{\ell,simple}\ra \J, 
\eear 
denote the  universal moduli space    generalizing \eqref{5.M.simple},  consisting of  equivalence classes of pairs $p=(f, J)$, where $J\in \J= \J^l$ and $f$ is a $W^{l,r}$ simple $J$-holomorphic map whose domain $C=(\Sigma, j, x_1, \dots, x_\ell)$ is  a smooth and compact  (but not necessarily connected)  complex curve with $\ell$ marked points. Write $\M_{0, simple}$ as $\M_{simple}$, and   let 
$$
{\mathcal{NE}}\subset \M_{simple}
$$
be the subset of the universal moduli space consisting of simple maps that are not embedded. Maps in ${\mathcal{NE}}$  
either (a) are not one-to-one, or (b) are not  immersions.  Correspondingly, consider two types of subsets of the moduli spaces \eqref{M.k.simple}: 

\smallskip

\begin{enumerate}[(a)]
\item For a pair of marked points, the inverse image of the diagonal  under the evaluation map 
\bear
\label{ev.2.pts}
\ev:\M_{2,simple} \longrightarrow X\times X
\eear
is the subset 
$$
\ev^{-1}(\De) \subset \M_{2,simple}
$$
consisting of simple $J$-holomorphic maps $f$ with smooth marked domain but whose image has a double point $f(x_1)=f(x_2)$.

\smallskip

\item For a single marked point $x_1$, there is a subset  ${\cal NI}$ of $\M_{1, simple}$   consisting of  simple 
$J$-holomorphic maps $f$ that are not immersions at  $x_1$. 
\end{enumerate}

\smallskip

 As in Section~5, we will  analyze these two subsets by regarding the moduli space $\M_{\ell,simple}$ locally  as a subspace of a slice as in \eqref{5.2bsliceslice}.  To describe ${\cal NI}$, let  $\L$ be the 
 complex line bundle over the slice whose fiber at $p=(f, J)$ is the cotangent line $T_{x_1}^* C$ to the complex curve $C$ at the marked point $x_1$, and let $\ev:\M_{1, simple} \to X$ be evaluation at $x_1$.  The
bundle
\bear
\label{ev.tg.pts}
\xymatrix{ \L\otimes_\cx \ev^* TX\ar[r] & 
\ar@ /_.8pc/ [l]_{\Phi_1}\M_{1, simple}} 
\eear
has a section $\Phi_1$ defined by  $\Phi_1(f,J)=(df)_{x_1}$; note that this lies in $(\L\otimes_\cx ev^*TX)_{p}=  T^*_{x_1} C \otimes_\cx T_{f(x_1)}X$ because $(df)_{x_1}:T_{x_1}C\to T_{f(x_1)}X$ is complex linear for $J$-holomorphic maps $f$.  The zero set of $\Phi^{-1}_1(0)$ is the set ${\cal NI}$ in (b).

\begin{lemma} 
\label{LemmaA.1} 
For each $\ell\ge 0$,  the moduli space \eqref{M.k.simple} is a  $C^k$ separable Banach manifold  for $k$ as in Proposition~\ref{5.1new}.  Furthermore, 
\begin{enumerate} [(a)]
\item The evaluation map \eqref{ev.2.pts} is  $C^k$ and is   transverse to the diagonal. \\[-3mm]
\item The section \eqref{ev.tg.pts} is  $C^k$ and is transverse to the zero section.
\end{enumerate}
\end{lemma}

\begin{proof} 

    Proposition~\ref{5.1new} extends to show that $\M_{\ell,simple}$ is a separable  Banach manifold, locally $C^k$ diffeomorphic to a  $C^{l-m}$ submanifold $M_{\ell, simple}$ of a slice ${\cal Slice}_\tau^{m,r,l}$ for $l\ge 6$, $r>2$, $2k\le l-2$, and $1\le m \le l-k$.  Note that the section \eqref{ev.tg.pts} extends over the slice by the formula
\bear\label{A.lasteqofjob}
 \Phi_1(f,J)\ =\ \tfrac12(df-Jdf j)(x_1),
\eear
 which is equal to $df(x_1)$ if $f$ is $(j,J)$-holomorphic.  The Sobolev embedding theorem implies that, on the slice, the evaluation map defined by $ev(f,J)=(f(x_1), f(x_2))$ is smooth for $m\ge 1$, and that the extension \eqref{A.lasteqofjob} is  $C^{l-m}$  for $m \ge 2$.  Consequently, the map \eqref{ev.2.pts} and the  section $\Phi_1$ in \eqref{ev.tg.pts} are $C^k$  for $k$ in the above range provided that $m\ge 2$. 

   Statement (a)  is true by Proposition~3.4.2 of \cite{ms}.  
  To prove (b), we modify the proof of Lemma~3.4.3 in \cite{ms}.   This involves three steps:  (i) computing $d\Phi_1$ in certain directions, (ii) expressing the needed transversality as a differential equation with constraints on  1-jets (not just  values) at the marked point, and (iii) solving this equation using weighted Sobolev spaces.  

\smallskip
 
 Fix a slice as above with $m\ge 2$ and a point $p=(f,J)$ on the zero set of $\Phi_1$, so $df(x_1)=0$. Consider a variation $p_t=(f_t, J_t)$ of $p=p_0$ in  $\M_{1, simple}$ that fixes the domain (including the complex structure and the marked point $x_1$), and also fixes the image 
point $f(x_1)$. Then  $f_t:C\ra X$ is $(j, J_t)$-holomorphic, and $f_t(x_1)=f(x_1)$ is constant. The tangent vector to $p_t$ at $t=0$ then has the form $(\zeta, 0, K)$,  where 
\bear\label{cond.zeta.K}
\zeta\in W^{m, r}(f^*TX), \quad \zeta(x_1)=0, \quad K\in T_J\J^l, \quad \mbox{ and }\quad  \L_p (\zeta, 0, K)=0. 
\eear 
Here $\L_p$ is the linearization given by \eqref{4.formulabigL} and \eqref{1.formulaL}.  Calculating the first variation of 
 \best
\Phi_1(p_t)= \tfrac12(df_t-J_tdf_tj)(x_1) \in (\Lambda_C^{1, 0}\otimes_\cx f_t^* TX)_{x_1}
\eest
 as in \cite[Proposition 3.1.1]{ms} (with a sign change), and  using the fact that $df(x_1)=0$, one finds that
\bear\label{A.dPhi}
 (d\Phi_1)_p(\zeta, 0, K)\ =\  \tfrac12\left(\nabla \zeta -J\nabla\zeta \circ j  + (\nabla_\zeta J )df j+K df j\right)(x_1)\ =\ (\nabla \zeta)^{1, 0}(x_1).
\eear
The righthand side is independent of the connection because $\zeta$ vanishes at $x_1$.

 To prove (b), it suffices to show that \eqref{A.dPhi} is surjective, i.e. for each $\eta_0\in (\Lambda_C^{1, 0}\otimes_\cx f^* TX)_{x_1}$ there exists a tuple $(\zeta, 0, K)$ satisfying \eqref{cond.zeta.K} and such that 
\bear\label{dfphi=eta}
 (\nabla \zeta)^{1, 0}(x_1)=\eta_0 . 
\eear

Choose a local holomorphic coordinate $z$ on $C$ centered at $x_1$, and write $\eta_0=v_0 dz$, where $v_0\in (f^*TX)_{x_1}=T_{f(x_1)}X$.  Extending $v_0$ to a smooth section of $TX$, pulling back by $f$, and multiplying by a bump function creates a $W^{l,r}$ section $v$ of  $f^*TX$ supported in the coordinate chart with $v(x_1)=v_0$.  
We will seek a solution $(\zeta, 0, K)$ of  \eqref{cond.zeta.K}  and \eqref{dfphi=eta} satisfying
\best
\zeta= z v+\mu, 
\eest
 where $ \mu\in W^{l, r}(f^*TX)$ is a correction satisfying 
\bear\label{cond.mu}
\mu(x_1)=(\nabla \mu)(x_1)=0.
\eear
This ansatz implies that $\zeta\in W^{l, r}$, $\zeta(x_1)=0$ and 
$$
(\nabla \zeta)^{0,1}(x_1)=\big(\partial z\cdot v + z(\nabla v)^{0,1}\big)\big|_{z=0}= v_0 dz=\eta_0.
$$ 
The only  remaining constraint is the last equation in \eqref{cond.zeta.K}, which reduces to 
\bear\label{A.1eq1}
\L_p(\mu, 0, K)\ =\ D^0_p\mu +\tfrac12 K df j\ =\ \alpha \quad \mbox{ for } \quad  \alpha = -D^0_p(zv).
\eear
Using the formula \eqref{D0.formula.lin} for  $D^0_p$,  one sees that $\alpha$ is bounded.  
Furthermore, $|\alpha(z)| =O(|z|)$ for small $z$, as follows.
As in \cite[3.1.5]{ms}, we can write $D_p^0$ as the sum of a first order complex-linear operator $D_p^{0,1}$, and a complex anti-linear zeroth order operator $R$ given in terms of the Nijenhuis tensor of $J$ by $(R\zeta)(w) = \tfrac14 N_J(\zeta, \partial f(w))$.  The calculation
$$
D_p^0(zv)\,=\, D_p^{0,1}(zv) + R(zv)\,=\, zD_p^{0,1}v +\ov{z} Rv\,=\, zD_p^0v + (\ov{z}-z)Rv
$$
then implies  that $\alpha$ is $O(|z|)$.

Because $f$ is simple, we can now use weighted Sobolev spaces to find a solution $(\mu, K)\in W^{1,r}(f^*TX)\oplus T_J\J^l$ of  equation \eqref{A.1eq1} that satisfies \eqref{cond.mu}.   The appropriate weighted spaces are defined below, and the needed facts  listed in Lemmas~\ref{FirstNormLemma} and \ref{weightedLsurjective}.  Using the results and  notation of those lemmas, the proof of (b) is completed as follows.

Fix $\delta\notin \Omega_D$ with $1<\delta<2$,  and define $r>2$  by $\delta=2-2/r$.   Since $C$ is compact and $\alpha$ is bounded and $O(|z|)$, one sees that  $\alpha$ is  in $\F^{0,r}$, and hence is in $W^{0, r, \delta}$  by  Lemma~\ref{FirstNormLemma}a.   Because $f$ is simple, Lemma~\ref{weightedLsurjective}c  shows the existence of a solution $(\mu,0, K)\in W^{1, r, \de}\times \{0\}\times T_J\J^{l}$ to \eqref{A.1eq1}.  This $\mu$ satisfies    $D_p(\mu, 0)=\alpha-\tfrac12 K df j$ on $C'$,  and  both  $\alpha=-D_p(zv,0)$ and $Kdfj=2 \L_p(0,0,K)$ are in $\F^{l-1,r}$, as is seen by taking $m=l$ in  \eqref{def.dl.comp}.  But then Lemma~\ref{FirstNormLemma}b  shows that $(\mu, 0)$ lies in $\E^{l,r}$  and satisfies \eqref{cond.mu} and  \eqref{A.1eq1}.  This completes the proof that $(d\Phi_1)_p$ is surjective. 
\end{proof}

 The  weighted Sobolev spaces used at the end of the above proof are defined as follows. Fix a local  holomorphic coordinate $z:U\to \cx$ with origin at $x_1$, 
 and a
 Riemannian metric $g_0$ on $C$ that is euclidean on $U$.  Also fix   a smooth positive function $\rho$ on $C'=C\setminus\{x_1\}$ that is equal to $|z|$ on $U\setminus\{x_1\}$, and a constant $\delta\in \R$.  Let $g'$ be the metric $\rho^{-2}g_0$ on $C'$.  Writing $z=e^{-(t+i\theta)}$  gives coordinates $(t,\theta)$  with
\bear\label{A.rhozet}
 \rho = |z| = e^{-t}
\eear
 and  $g'= dt^2+d\theta^2$, so $(C', g')$ is a manifold with an end  $C'_{end}$  isometric to the cylinder $[0,\infty)\times S^1$, where $S^1=\R/2\pi\Z$. 
Let 
 $$
\E_0^{1,r,\delta} = W^{1, r, \delta}(f^*TX)
 $$
be the completion of the set of $C^l$  sections of $f^*TX$ with  compact support on $C'$ in the norm
\bear\label{A.Defweightednorm}
 \|\xi\|_{1,r,\delta}^r\ =\    \int_{C'}\, |\rho^{-\delta}  \nabla\xi |^r+ |\rho^{-\delta}  \xi |^r\ \  dvol_{g'},
\eear
 using  the norm and connection on the bundle  $f^*TX$ induced by the metric $g'$ on $C'$ and the metric on $X$.   The spaces  $\E_0^{0,r,\delta}$ and $\F^{0,r,\delta}$ are defined similarly (cf. \eqref{def.ef.comp}), also using the metrics $g'$ on $C'$ and the metric on $X$.

 \medskip

 The following lemma gives ways to translate between these weighted spaces,  which are defined using  the metric $g'$ and its Levi-Civita connection $\nabla'$ on $C'$, and the 
unweighted spaces $\E^{m,r}$  and $\F^{m,r}$, which were defined in Section~\ref{subsection4.2}  using the metric $g_0$ and connection $\nabla$ on $C$.  Part (b) is a regularity result for the   operator $D^0_p$   defined by \eqref{D0.formula.lin}.

  \begin{lemma}
 \label{FirstNormLemma}
 Fix $p=(f,J)$,  where $f\in W^{l,r}$ is $J$-holomorphic,  $J\in \J^l$, $l\ge 6$, and $r>2$.  Then
 \begin{enumerate}[(a)]

 \item   If   $\alpha\in \F^{0,r}$ is a 1-form with   $|\alpha(z)| =O(|z|)$,  then $\alpha\in \F^{0,r,\delta}$ for all $\delta\le 2-2/r$.\\[-3mm]
 \item  If $\mu\in\E_0^{1,r,\delta}$, $\delta>1$, is a  weak solution of $D^0_p\mu=\alpha$ on $C\setminus\{x_1\}$ with $\alpha\in \F^{l-1,r}$, then $\mu$   extends to a solution on $C$ with $(\mu, 0)\in\E^{l,r}$, and with $\mu(0)=(\nabla\mu)(0)=0$. 
 
\end{enumerate}
 \end{lemma}

 \begin{proof}
 
 (a)   For a 1-form $\alpha$, the norms with respect to the metrics $g_0$ and  $g'=\rho^{-2} g_0$ are related by  $|\alpha|_{g'} = \rho |\alpha|_{g_0}$, while $dvol_{g'}= \rho^{-2} dvol_{g_0}$.   Hence
 $$
 \int_{C'}    |\rho^{-\delta}\alpha|^r_{g'}\ dvol_{g'}\ =\  \int_{C'}   |\rho^{1-\delta-2/r}\alpha|^r_{g_0}\ dvol_{g_0}.
 $$
 The righthand integral is finite for $\delta\le 2-2/r$ because  $\alpha\in \F^{0,r}$ and, by assumption,  $|\alpha|_{g_0}\le c_1\rho=c_1 e^{-t}$ on $C'_{end}$.

 \smallskip
 
   (b)  The pointwise norm of a section of $f^*TX$ does not depend on the metric on the domain.  Integrating $|\mu|^r\ dvol_{g_0} = \rho^{2+\delta r}\, |\rho^{-\delta}\mu|^r \ dvol_{g'}$, and again noting that $\rho$ is bounded and  is equal to $|z|$ on the end, shows that $\mu\in W^{0,r}$ and that,   for small $\ep$,  the integral over the disk $B(\ep)$ centered at $x_1$ satisfies
$$
\int_{B(\ep)} |\mu|^r\ dvol_{g_0}\ \le\  \ep^{2+\delta r}\, \|\mu\|^r_{0,r,\delta}  \ \le\ c_2\, \ep^{2+\delta r}.
$$
H\"{o}lder's inequality then shows that  for any $s\le r$,    $\mu\in W^{0,s}$  and there is a constant $c_3=c_3(s)$ such that
\bear\label{A.wtedgrowthrate}
\int_{B(\ep)} |\mu|^s\ dvol_{g_0}\ \le\ c_3\, \ep^{2+\delta s}.
\eear

 To apply elliptic regularity, we first verify that $D^0_p\mu=\alpha$ weakly on all of  $C$.  Choose a smooth 1-parameter family of cutoff functions $\{\gamma_\ep\}$ supported on $B(\ep)$ with $0\le\gamma_\ep\le 1$ and  $|d\gamma_\ep|<4/\ep$.  Given any  $\eta\in \F^{l,r}$,    write $\eta=\gamma_\ep\eta+(1-\gamma_\ep)\eta$     and  integrate:
\bear\label{A.weaksolution}
\int_C \langle (D_p^0)^*\eta, \mu\rangle-\langle\eta,\alpha\rangle\ = \  \int_{B(\ep)}  \langle (D_p^0)^*(\gamma_\ep\eta), \mu\rangle -  \int_{B(\ep)}   \langle \gamma_\ep \eta,\alpha\rangle  +  \int_C \langle (1-\gamma_\ep) \eta,D^0_p\mu-\alpha\rangle.
\eear
The last  integral vanishes because $D^0_p\mu=\alpha$ weakly on $C\setminus\{x_1\}$.  Noting that $\eta$ is bounded by the Sobolev Embedding Theorem,   H\"{o}lder's inequality  shows that the absolute value of the  middle integral on the righthand side is bounded by $\|\eta\|_{\infty}\, \|\alpha\|_{0,r}\, \|\gamma_\ep\|_{0,s; B(\ep)}\le c_4\ep^{2/s}$, where $\frac{1}{s}=1-\frac{1}{r}$. For $\ep \le 1$, the  first integral on the right  is similarly bounded by 
$$
\int_{B(\ep)} \left( |d\gamma_\ep| \, |\eta| + \gamma_\ep |(D_p^0)^*\eta| \right)\,|\mu|\
  \ \le\  \left( \|d\gamma_\ep\|_{0,r;B(\ep)} \|\eta\|_\infty    + \|(D_p^0)^*\eta\|_{0,r}\right)\, \|\mu\|_{0,s; B(\ep)} 
\ \le\  c_5\,  \ep^{1+\delta},
$$
where this last inequality  follows from  \eqref{A.wtedgrowthrate}  because  $\|d\gamma_\ep\|_{0,r;B(\ep)} \le c_6\ep^{2/r-1}$ and $\|(D^0_p)^*\eta\|_{0,r}\le c_7 \|\eta\|_{1,r}$  as in \cite[Proposition~3.1.11]{ms}.   These bounds hold for all small  $\ep>0$, so the lefthand side of \eqref{A.weaksolution} is equal to 0.  Thus  $D_p\mu=\alpha$ weakly on   $C$.  

Elliptic regularity, as in Theorem~C.2.3 of \cite{ms}, then shows  that $\mu \in  W^{l,r}$ and $D_p^0\mu=\alpha$ on $C$.  In particular, $\mu$ is $C^2$ by the  Sobolev Embedding Theorem.  With this,    \eqref{A.wtedgrowthrate} and the hypothesis that $\delta>1$    imply  that $\mu(0)=(\nabla\mu)(0)=0$.
 \end{proof}

  \begin{lemma}
 \label{weightedLsurjective}
 Fix $p=(f, J)$ as in Lemma~\ref{FirstNormLemma}.    Then there is a discrete set $\Omega_D\subset \R$ such that for each $\delta\notin \Omega_D$,  
   \begin{enumerate}[(a)]
   \item The  operator $D^0_p$   defined by \eqref{D0.formula.lin} on $C^l$ sections  with compact support in $C\setminus\{x_1\}$ extends to a  Fredholm operator
\bear\label{A.SobolevD}
D^0_p:\E_0^{1,r,\delta}\to \F^{0,r,\delta}.
\eear
 \item  If $df(x_1)=0$ and $\delta<2$, then  the operator  $\L_p$ defined by    \eqref{4.formulabigL} and \eqref{1.formulaL}, restricted to the subspace defined by
 $k=0$,   induces  a bounded operator  
\bear\label{A.weightedSobolevL}
\L_p:W^{1,r,\delta}(f^*TX) \oplus \{0\} \oplus T_J\J^m\to \F^{0,r,\delta}.
\eear
\item  If, in addition,  $f$ is simple, then   \eqref{A.weightedSobolevL} is surjective.
\end{enumerate}
 \end{lemma}

\begin{proof}

(a)  The Fredholm properties of the operator \eqref{A.SobolevD} are determined by its asymptotic behavior in a neighborhood of $x_1$, which depends on the geometry of $X$ near $f(x_1)$. Choose a local trivialization of the complex vector bundle $(TX, J)$ in  a neighborhood $V$ of $f(x_1)$. Our assumptions imply that $f$ is of class $C^{l}$ so, after shrinking $C'_{end}$, we can assume that $f(C'_{end})$ lies in $V$. Pulling back yields a $C^{l-1}$  trivialization $f^*TX \cong  C'_{end}\times \R^{2N}$ of  $f^*TX$ over the end $C'_{end}$ in which $J$ corresponds to 
 the standard complex structure $J_0$ on $\cx^N=\R^{2N}$.  This, together with the section $d\overline{z}$ of $\Lambda^{0,1}_C$, gives a similar trivialization of $ \Lambda^{0,1}_C\otimes_\cx f^*TX$ on the end.

 Referring to  formula \eqref{D0.formula.lin} and noting that $f$ is $J$-holomorphic, we can write $D^0_p\zeta$ in these trivializations as
 \bear\label{A.3.1}
D^0_p\zeta\ =\ \del_0\zeta \ +\        S \cdot \zeta\  d\overline{z},
\eear
where $\del_0\zeta  =\frac12(d\zeta+J_0d\zeta j)$ and $S$ is a matrix-valued function depending on  the pullbacks of $J$, $\nabla J$, and the  connection form  of  $\nabla$ in the trivialization.  The specific formula for  $S$ shows that it is  at least $C^{l-1}$.

After converting   to $(t,\theta)$ coordinates on $C'_{end}$ and substituting $d\overline{z} = - \overline{z} (dt-id\theta)$, 
 \eqref{A.3.1} becomes
 \bear\label{A.Ddel0+R}
D^0_p\zeta\ =\    \del_0\zeta  +  T\, \zeta,
\eear
where  $\del_0\zeta = \frac12\big(\partial_t\zeta+J_0\partial_\theta\zeta \big)$  and $T = -e^{-t+i\theta}S$.  Since $S$ is bounded on $C$, we have $|T| \le c_1\, e^{-t}$. 
Thus $D^0_p$ is a first order elliptic operator with $C^{l-1}$ coefficients that, in some trivialization on  the end  of $C'$,    is the sum of the translation invariant operator $\del_{0}$ and a $0^{th}$ order  term $T$ that decays to 0 uniformly as $t\to \infty$.   A theorem of Lockhart and McOwen \cite[Theorem 6.2]{LM} then implies that \eqref{A.SobolevD} is bounded and Fredholm for all $\delta$ not in a discrete set $\Omega_D$ (the proof assumes that $D_p$ has smooth coefficients, but applies without change for coefficients that are $C^{2}$ or better).

\medskip

(b)  Using (a), it suffices to bound  the last term in \eqref{4.formulabigL}.  Because $f$ is $C^2$, the assumption that $df(x_1)=0$ implies that $|df(z)|_{g_0}\le c_2 |z|=c_2\rho$  on the end, and hence by compactness there is a constant $c_3$ such that $|df|_{g'}\le c_3\rho^2$ on all of $C'$.  We then have:
$$
\|K df j\|_{0,r,\delta}\ \le\  \|K\|_{C^0} \left(\int_{C'} |\rho^{-\delta}\cdot c_2\rho^2|^r\ dvol_{g'}\right)^{1/r} \le\ c_4\, \|K\|_{C^l},
$$
where the last inequality holds because   $\rho=e^{-t}$ on the end of $C'$ and $\delta<2$.

\medskip

(c)  Following the argument used to prove Proposition~\ref{5.1new},  if $\L_p$ is not surjective, then there exists an element $c$ of the dual space $(\F^{0,r,\delta})^*=\F^{0,s,-\delta}$, $s=\frac{r}{r-1}>1$,   such that $(D^0_p)^*c=0$ and \eqref{new5.1} holds for every   $K$ in $T_J\J^l$.    The proof of Lemma~\ref{lemma4.1kc}, applied on $C'$, shows that   there is an injective point $x\in C'$ such that  $c(x)\not= 0$.  Then $K=\beta_\ep K_0$, as defined after \eqref{eq.beta.delta},  has compact support in a neighborhood of this point $x\in C'$ and  lies in $T_J\J^l$,  giving a contradiction as $\ep\to 0$.  Thus \eqref{A.weightedSobolevL} is  surjective.
\end{proof}

\bigskip

Next, note that for every $\ell\ge 0$ the map 
 \bear\label{A.pi.ell}
\pi_\ell: \M_{\ell, simple}\to \M_{simple}
\eear
that forgets the marked points is a submersion with $\ind \pi_\ell=2\ell$.

\begin{prop} 
\label{PropA.2}
 The set  $\mathcal{NE}\subset \M_{simple}$ of simple maps that are not embeddings has codimension $\dim X-4$ in the sense of  Definition~\ref{4.defcodimkset}.
\end{prop}
\begin{proof}
As above, $\mathcal{NE}$ is the union of  $\pi_1(\iota_1S^1)$ and $\pi_2(\iota_2 S^2)$,  where $\iota_1: S^1= \Phi_1^{-1}(0) \hookrightarrow \M_{1, simple}$ and $\iota_2:S^2=\ev^{-1}(\De)  \hookrightarrow\M_{2, simple}$.  Lemma~\ref{LemmaA.1} implies that $\iota_1$ and $\iota_2$ are inclusions of submanifolds with $\ind \iota_\ell= - \codim S^\ell$  for $\ell =1, 2$.  One sees from \eqref{ev.2.pts} and \eqref{ev.tg.pts} that   $S^1$ and $S^2$ both have codimension $\dim X$.  Hence $\mathcal{NE}$ is a set of codimension~$k$ where $k=-\ind (\pi_\ell\circ \iota_\ell) =-\ind \pi_\ell- \ind \iota_\ell  = -2\ell+ \dim X\ge \dim X-4$.   
\end{proof}

\begin{cor}
\label{corA.3}
If $X$ is a symplectic Calabi-Yau 6-manifold,  there is a Baire set $\J^*\subset \J^l$,  $l\ge 6$ or $l=\infty$,    so that for each $J\in \J^*$ all simple $J$-holomorphic maps are  regular, and are  embeddings with pairwise disjoint images. 
\end{cor}

\begin{proof}
 For $\J=\J^l$, the Sard-Smale Theorem implies that the regular values of   the projection $\pi:\M_{simple}\to \J$ are a Baire set $\J_1$ in $\J$.     When $X$ is a Calabi-Yau 6-manifold, the index of $\pi$ is 0, and  Proposition~\ref{PropA.2} shows that    $\mathcal{NE}$ is a codimension~2 subset of  $\M_{simple}$.  Applying  Sard-Smale again,  there is a Baire set $\J_2$ of $\J$ such that for each $J\in \J^* = \J_1\cap \J_2$, $J$ is a regular value  of $\pi$ and   $\pi^{-1}(J)$ is disjoint from  $\mathcal{NE}$,  which means that all simple $J$-holomorphic curves are embedded.

This proof extends to the space of smooth maps over  $\J=\J^\infty$ by applying Taubes' argument, as in the proof Theorem~3.1.6(II) in \cite{ms}.
\end{proof}

 \vspace{5mm}

\renewcommand{\theequation}{B.\arabic{equation}}
\renewcommand{\thetheorem}{B.\arabic{theorem}}
\section*{A{\sc ppendix} B}
\bigskip

 This second  appendix is devoted to the proof of the following result,  which immediately implies Lemma~\ref{lemma4.7},  and also generalizes parts  (b) and (c) of Proposition~\ref{MmfdThm}.  As in Appendix~A,  our moduli spaces and operators $D_p$ are defined on the  Sobolev completions introduced in Section~\ref{subsection4.2}, but for notational simplicity we omit the superscripts indicating the Sobolev norms.

\begin{prop} 
\label{PropA4} 
Suppose $\dim X\ge 6$ and $\N$ is a component of $\M_{simple}$  such that the projection $\pi:\N \ra \J$ has index 0. Then 
\begin{enumerate}[(a)]\itemsep2mm 
\item $\W^1\cap \N$ is a codimension~$1$ submanifold of $\N$, and 
\item $(\W\setminus \W^1)\cap \N$ is a subset of $\N$ of codimension $\ge 3$.   
\end{enumerate}
Furthermore, (a) and (b)  hold with $\N$ replaced by  $\N\cap \M_C$.  
\end{prop}

The proof is based on another construction involving the spaces \eqref{M.k.simple}   and the projections \eqref{A.pi.ell}.  Below, we will locally regard $\M_{\ell, simple}$ as a subset of a slice, and  write its elements  as pairs $q=(p, {\bf x})$,  where $p=\pi_\ell(q)=(f, J)\in \M_{simple}$ and ${\bf x}=(x_1, \dots, x_\ell)$ are the marked points on the domain $C$ of $f$.  For each $\ell\ge 0$, let 
\best
\M^*_{\ell, simple} \subseteq \M_{\ell, simple} 
\eest
be the open set of all   $q$ such that each of the marked points $x_1, \dots, x_\ell$  is an injective point of $f$.   
 Pull back the bundles $\E$ and $\F$ to $\M_{\ell, simple}^*$ by the   map between slices corresponding to \eqref{A.pi.ell}.  Let $\E^\ell\ra \M^*_{\ell, simple}$ be the subbundle of $\pi_\ell ^*\E$ whose fiber at  $q=(p, \bf x)$ is the set 
\best
\E_q^\ell  = \left\{ \; \xi \in \pi_\ell^*\E_p \;|\; \xi^N(x_i)=0 \mbox { for $i=1, \dots, \ell$ }\right\}
\eest
of  elements of $\pi_\ell^*\E_p$  whose normal component vanishes at $x_1, \dots,x_\ell$.  For each $q=(p, {\bf x})\in  \M_{\ell, simple}^*$,  
$\E_q^\ell$ is a linear subspace of $\pi_\ell^*\E_p$ of codimension $\ell(\dim X-2)$, and  the linearization $D_p$  at $p$, given in \eqref{1.formulaL},  restricts to a linear map
\bear
\label{A.EF}
D^\ell_q: \E^\ell_q \to \pi_\ell^* \F_p.
\eear
Regarding  $D^\ell_q$ as the composition of the inclusion $\E^\ell_q\hookrightarrow \pi_\ell^*\E_p$ with  $D_p$, one sees that 
$D_q^\ell$ is Fredholm with index 
\bear
\label{A.ind.D.l=ind.D-}
\iota_\ell\, = \, \ind D_q^\ell \, = \,  \ind D_p-\ell (\dim X-2)
\eear
for all $q=(p, {\bf x})\in \M_{\ell, simple}^*$. 
 As in \eqref{4.PsiD},  let  $\Fred_{\iota_\ell}\to \M_{\ell, simple}^*$ be the fiber  bundle whose fiber at $q=(p, {\mathbf x})$ is the space of index~$\iota_\ell$  Fredholm operators  from $\E^\ell_q$ to $\pi_\ell^* \F_p$. This is   stratified by submanifolds $\Fred_{\iota_\ell}^s$, and
$$
 \Psi^\ell(q)=D^\ell_q
$$
defines a section of this bundle.   Let 
$$
V^\ell\subseteq  \M_{\ell, simple}^*
$$
 be the open set of all $q=(p, x_1, \dots, x_\ell)$ such that there exists an injective point $y$, distinct from the set $\{x_i\}$, such that    $\ker D^{\ell+1}_{(q, y)}=0$. 

\begin{lemma} 
\label{lemmaA5}
The section $\Psi^\ell$ is transverse to $\Fred_{\iota_\ell}^s$ along  $V^\ell$,   as is its restriction to $\pi^{-1}(\M_C$).    Hence for $s\ge 1$  the sets
\bear
\label{4.defSells2}
S^{\ell, s}\ =\ V^\ell\cap (\Psi^\ell)^{-1}\Fred^s_{\iota_\ell}\ =\ \left\{  q\in V^{\ell}\, |\, \dim \ker D^\ell_q=s   \right\}
\eear
and $S^{\ell, s}_C= S^{\ell, s} \cap \pi_\ell^{-1}(\M_C)$ are  submanifolds of  codimension  $s(s-\iota_\ell)$, where $\iota_\ell$ is given by \eqref{A.ind.D.l=ind.D-}. 
\end{lemma}

\begin{proof} To prove transversality  at   $q\in S^{\ell,s}$  we must show that  the image of $(d\Psi^\ell)_q$ projects surjectively onto the normal space $\Hom(\ker D^\ell_q, \cok D^\ell_q)$ to $\Fred^s_{\iota_\ell}$ at $D^\ell_q$.   Fix a slice containing $q$ and identify $\cok D^\ell_q$ with the kernel  of the adjoint operator $(D^\ell_q)^*$ defined as in \eqref{4.2.D*def}.   By contradiction, assume there exists a non-zero element of  the normal space, regarded as a linear map $A_q:\ker D^\ell_q\ra \ker (D^\ell_q)^*$, such that $ \langle A_q, (\delta_v D_q^\ell)\rangle_{L^2}=0$ for every variation $v\in T_q\M_{\ell, simple}$.  Since $A_q \ne 0$, there exists an $L^2$-normalized   $\kappa\in \ker D^\ell_q$ such that $c=A_q\kappa\in \ker (D^\ell_q)^*$ is nonzero. Fix an $L^2$ orthonormal basis $\{\kappa_i\}$ of $\ker D^\ell_q$ with  $\kappa_1=\kappa$.   We then have
\bear
\label{4.AdeltaD}
0\ =\ \langle A_q, (\delta_v D_q^\ell)\rangle_{L^2}
\ =\   \sum_i \int_C \lg A_q\kappa_i, \, (\delta_v D^\ell_q)\kappa_i\rg
\eear
for all $v\in T_q\M_{\ell, simple}$.

The assumption that $q\in V^\ell$ means, by definition,  that there is an injective point $y\notin \{x_1, \dots, x_\ell\}$ such that the map 
  \bear
\label{ellthevaluationmap}
\ev_x: \ker D_q^\ell \ra  N_{f(x)}  \qquad \mbox{given by $\kappa\mapsto \kappa^N(x)$}
\eear 
is injective for $x=y$, and hence  for all $x$  in a neighborhood of  $y$. As in  the proof of Lemma~\ref{lemma4.1kc},  there exits an injective point $x$ in that neighborhood with 
  $c(x)\ne 0$.  For this $x$, the values $\{\kappa^N_i(x)\}$ are linearly independent because \eqref{ellthevaluationmap} is injective. 
 Applying  Lemma~\ref{LemmaA6} below with $\xi=\kappa_1^N(x)\ne 0$ and $V=\mbox{span}\{\kappa^N_i(x)\, |\, i\ge 2\}$ produces a $K$ satisfying \eqref{nablaKconditions}  below.
  
 Now proceed as in the proof of Proposition~\ref{MmfdThm}b,  taking  $v_\ep=(0, 2\beta_\ep K)$ in \eqref{4.AdeltaD}.  These variations do not affect the map  $f$, the complex structure on the domain or the marked points, and hence do not change the domain and range of the operators \eqref{A.EF}.     Because  $D_p$ and $D_q^\ell$ are differential operators with the same formula, the variation $(\delta_{v_\ep}D^\ell)_q$ is  again  given by  \eqref{4.variationDkappa}  with $K$ replaced by $2\beta_\ep K$.

After substituting and taking the limit $\ep\to 0$ as in the  proof of Proposition~\ref{MmfdThm}b, equation \eqref{4.AdeltaD} implies that

\bear
\label{4.AdeltaD2}
0\ =\  \sum_i  \lg (A_q\kappa_i)(x), \,  (\nabla_{\kappa^N_i(x)} K)f_*j\rg \ =\ |c(x)|^2,
\eear
where the last equality holds because $(\nabla_{\kappa^N_i(x)} K)f_*j=0$ for all $i\ge 2$ by  \eqref{nablaKconditions}. This contradicts the fact that $c(x)\not=0$, and hence establishes the  transversality of $\Psi^\ell$ at $q\in S^{\ell, s}$.
 The restriction of $\Psi$ to $\pi_\ell^{-1}(\M_C)$ is also transverse to $\Fred^s_{\iota_\ell}$ because for each  embedding  $q\in S^{\ell, s}$ with image $C$,  the  variations $v_\ep$ above  are tangent to $\pi_\ell^{-1}(\M_C)$.  Thus  $S^{\ell,s}$ and $S_C^{\ell,s}$ are manifolds whose codimension, in both cases, is the dimension of the normal space to $\mbox{Fred}^s_{-4\ell}$, which is $(\dim \ker D^\ell_q)(\dim \cok D^\ell_q)=s(s-\iota_\ell)$.   
\end{proof}
\medskip

 \begin{lemma}
 \label{LemmaA6} 
 Fix $p=(f, J)\in \M_{simple}$, an injective  point $x\in C$, and a neighborhood $U$ of $f(x)$.  For  any nonzero $\xi\in N_{f(x)}$, any subspace $V\subseteq N_{f(x)}$ not containing $\xi$, and any   $c\in (\Lambda^{1,0}_C \otimes_\cx f^*TX)_x$,  there exists a  $K\in T_J \J$, supported on $U$, vanishing along $f(C)$  such that, at the single point $x$: 
\bear
\label{nablaKconditions}
\mbox{$(\nabla_\xi K) f_*j=c$ \quad and\quad   $(\nabla_{w} K) f_*j=0$ \ $\forall w\in V$. }
\eear
 \end{lemma}
 
 \begin{proof} 
Still following the proof of  Proposition~\ref{MmfdThm}b,  there is a  $K_0\in T_J\J$ such that 
 $K_0 f_* j=c$ at $x$.  Choose a local coordinate system $\{z, y_1, y_2, \dots\}$ centered  at $f(x)$ with $z$ a local complex coordinate  on $f(C)$,  and $\{y_i\}$ real coordinates  vanishing along $f(C)$, and with $\frac{\partial\ }{\partial y_1}{\big|}_{f(x)}=\xi$ and $\frac{\partial\ }{\partial y_k}{\big|}_{f(x)}\in V$ for  $2\le k\le \dim V+1$. Then $K= y_1 \beta K_0$ has the  required properties where $\beta$ is any smooth function supported in $U$ with $\beta\equiv 1$ near the origin.
 \end{proof} 

\medskip 
 
\begin{proof}[Proof of Proposition~\ref{PropA4}]  We begin by making a series of observations about the images of the sets $V^\ell$ and $S^{\ell, s}$ under the forgetful map \eqref{A.pi.ell}.  \smallskip
\begin{enumerate}[(i)]
 \setlength\itemsep{1em}

\item {\it The images of  the $V^\ell$ cover $\M_{simple}$.}  If not, there would be a map  $p\in \M_{simple}$ not in the image of any $V^\ell$.  Choose a dense sequence 
$\{ x_1, x_2, \dots \}$ of distinct injective points in the domain $C$ of $p$.  Then for each $\ell$, $q_\ell = (p, x_1, \dots, x_\ell)\notin V^\ell$, which implies that $\ker D^{\ell+1}_{q_{\ell+1}} \not=0$.   But then $\ker D^{\ell+1}_{q_{\ell+1}}\subseteq \ker D^{\ell}_{q_\ell}$ are nontrivial  nested subspaces of the finite-dimensional vector space $\ker D_p$,  so have a nonzero intersection.  Hence there is a nonzero $\kappa\in \ker D_p$ whose normal component vanishes at all $x_i$,  and therefore everywhere, contradicting Lemma~\ref{lemma4.1kc}.    

\item {\it The images of  the  $S^{\ell,s}$ with  $s\ge 1$  cover $\W$.}   Given $p\in\W$, we have  $\ker D_p\not= 0$.   
As in (i), there is a sequence $\{x_i\}\subset C$ and an $m>0$ such that $\ker D^m_{q_m}=0$.  Let $\ell$ be the largest $k$ such that $\ker D_{q_k}^k\not= 0$.  Then  by \eqref{4.defSells2}, $q_\ell\in S^{\ell, s}$ for $s=\dim \ker D^\ell_{q_\ell}\ge 1$, and hence $p\in\pi_\ell(S^{\ell, s})$.

\item  {\it $S^{0,1}=\W^1$ is a submanifold of $\M_{simple}$.}   Equation~\eqref{4.defSells2} shows that $S^{0,1}\subseteq \W^1$, while Lemma~\ref{lemma4.1kc} implies $\W^1\subseteq S^{0,1}$. Hence, by Lemma~\ref{lemmaA5},  $\W^1$ is a submanifold of $\M_{simple}$ of codimension $1-\iota_0$.     In particular,  for each component $\N$ of $\M_{simple}$ with  $\ind D_p=0$, we have $\iota_0=0$ by \eqref{A.ind.D.l=ind.D-}, so
the restriction  $\W^1\cap \N$  is  a codimension 1 submanifold of $\N$. 
 
\item {\it $\pi_\ell: S^{\ell, s}\to\M_{simple}$ is a Fredholm map of index $2\ell+s(\iota_\ell- s)$.} This  map is the composition of the inclusion $S^{\ell, s}\to\M_{\ell, simple}$, which has index $s(\iota_\ell-s)$ by Lemma~\ref{lemmaA5}, and the map \eqref{A.pi.ell} which has  index $2\ell$. 
\end{enumerate}  
\smallskip

By Facts~(ii) and (iii),   $\W\setminus \W^1$ is covered by the sets $\pi_\ell({S}^{\ell, s})$ for  $\ell\ge 0$,  $s\ge 1$, and $(\ell, s)\not= (0,1)$.  By Fact~(iv) and \eqref{A.ind.D.l=ind.D-}, 
the intersection of each of these   sets  with $\N$ is the image of a Fredholm map of index 
$2\ell-\ell s(\dim X-2) -s^2 \le 2\ell-4\ell s-s^2 \le -3$. Thus $(\W\setminus \W^1)\cap \N$  is a set of  codimension~3 in the sense of Definition~\ref{4.defcodimkset}.

 The same proof applies if we restrict everything to  $\M_C$ instead of $\M_{simple}$. 
  
\end{proof}

\vspace{5mm}

\end{document}